\documentclass[a4paper,twoside]{article}
\usepackage[utf8]{inputenc}
\usepackage{enumitem}
\usepackage[T1]{fontenc}
\usepackage{amsmath}
\usepackage{amsfonts}
\usepackage{amssymb}
\usepackage[left=4cm,right=4cm,top=4cm,bottom=4cm]{geometry}
\usepackage{amsthm}
\usepackage{graphicx}
\usepackage{caption}
\usepackage{subcaption}
\usepackage{mathtools}
\usepackage{tikz}
\usepackage{import}
\usepackage{fancyhdr}
\usepackage[hidelinks]{hyperref}
\usepackage{float}

\usetikzlibrary{matrix}

\hypersetup{
  colorlinks   = true, 
  urlcolor     = black,
  linkcolor    = red, 
  citecolor   = blue 
}

\graphicspath{{drawings/}}

\newcommand{\Q}{\ensuremath{\mathbb{Q}}}
\newcommand{\Z}{\ensuremath{\mathbb{Z}}}
\newcommand{\N}{\ensuremath{\mathbb{N}}}

\newcommand{\C}{\ensuremath{\mathbb{C}}}

\DeclareRobustCommand{\qbinom}{\genfrac[]{0pt}{}}

\theoremstyle{plain}
\newtheorem{thm}{Theorem}
\newtheorem{prop}[thm]{Proposition}
\newtheorem{lemme}[thm]{Lemma}
\newtheorem{cor}[thm]{Corollary}
\newtheorem{res}{Result}

\theoremstyle{definition}
\newtheorem{defn}[thm]{Definition}
				
\theoremstyle{remark}
\newtheorem{rem}[thm]{Remark}

\fancypagestyle{mypagestyle}{
  \fancyhf{}
  \fancyhead[OC]{\textsc{s.willetts}}
  \fancyhead[EC]{\textsc{a unification of the ado and colored jones polynomials of a knot}}
  \fancyfoot[C]{\thepage}
  
}
\newcommand{\address}{{
  \bigskip
  \footnotesize

  \textsc{Institut de Mathématiques de Toulouse, UMR5219,}\par\nopagebreak
  \textsc{UPS, F-31062 Toulouse Cedex 9, France}\par\nopagebreak
  \textit{E-mail address}: \texttt{sonny.willetts@math.univ-toulouse.fr}

}}

\begin{document}
\title{\textbf{\textsc{a unification of the ado and colored jones polynomials of a knot}}}
\author{\textsc{Sonny Willetts}}
\date{}
\maketitle

\begin{abstract}
In this paper we prove that the family of colored Jones polynomials of a knot in $S^3$ determines the family of ADO polynomials of this knot. More precisely, we construct a two variables knot invariant unifying both the ADO and the colored Jones polynomials. On one hand, the first variable $q$ can be evaluated at $2r$ roots of unity with $r \in \N^*$ and we obtain the ADO polynomial over the Alexander polynomial. On the other hand, the second variable $A$ evaluated at $A=q^n$ gives the colored Jones polynomials. From this, we exhibit a map sending, for any knot, the family of colored Jones polynomials to the family of ADO polynomials. As a direct application of this fact, we will prove that every ADO polynomial is holonomic and is annihilated by the same polynomial as of the colored Jones function. The construction of the unified invariant will use completions of rings and algebra. We will also show how to recover our invariant from Habiro's quantum $\mathfrak{sl}_2$ completion studied in \cite{habiro2007integral}.
\end{abstract}

\section{Introduction}
\paragraph{Main results:}~\medskip

In \cite{akutsu1992invariants}, Akutsu, Deguchi and Ohtsuki gave a generalisation of the Alexander polynomial, building a colored link invariant at each root of unity. These ADO invariants, also known as colored Alexander's polynomials, can be obtained as the action on 1-1 tangles of the usual ribbon functor on some representation category of a version of quantum $\mathfrak{sl}_2$ at roots of unity (see \cite{costantino2014quantum}, \cite{geer2009modified}). On the other hand, we have the colored Jones polynomials, a family of invariants obtained by taking the usual ribbon functor of quantum $\mathfrak{sl}_2$ on finite dimensional representations. It is known (\cite{costantino2015relations}) that given the ADO polynomials of a knot, one can recover the colored Jones polynomials of this knot. One of the results of the present paper is to show the other way around: given the Jones polynomials of a knot, one can recover the ADO polynomials of this knot.\\
We denote $ADO_r(A,\mathcal{K})$ the ADO invariant at $2r$ root of unity seen as a polynomial in the variable $A$, $J_n(q, \mathcal{K})$ the $n$-th colored Jones polynomial in the variable $q$ and $A_{\mathcal{K}}(A)$ the Alexander polynomial in the variable $A$.

\begin{res}
There is a well defined map such that for any knot $\mathcal{K}$ in $S^3$,
\[ \{ J_n(q, \mathcal{K}) \}_{n \in \N^*} \mapsto \{ ADO_r(A,\mathcal{K}) \}_{r \in \N^*},\]
(Detailed version: Theorem \ref{Thm_ADO_Jones}).
\end{res}

The above result is a consequence of the construction of a unified knot invariant containing both the ADO polynomials and the colored Jones polynomials of the knot. Briefly put, we get it by looking at the action of the universal invariant (see \cite{lawrence1988universal}  \cite{lawrence1990universal}, and also \cite{ohtsuki2002quantum}) on some Verma module with coefficients in some ring completion. For the sake of simplicity let's state the result for $0$ framed knots.

\begin{res}
In some ring completion of $\Z[q^{\pm1},A^{\pm1}]$ equipped with suitable evaluation maps, for any $0$-framed knot $\mathcal{K}$ in $S^3$, there exists a well defined knot invariant $F_{\infty}(q,A,\mathcal{K})$ such that:
\[ F_{\infty}(\zeta_{2r},A,\mathcal{K})=\frac{ADO_r(A, \mathcal{K})}{A_{\mathcal{K}} (A^{2r})}, \ \ \  F_{\infty}(q,q^n,\mathcal{K})= J_n(q^2, \mathcal{K}) \]
(Detailed version: Theorem \ref{ADO_factor_thm} and Corollary \ref{cor_Jones}).
\end{res}

A visual representation of the relationship between all these invariants is given at Figure~\ref{intro_schema}.

Lets denote $J_{\bullet}(q^2, \mathcal{K}) = \{ J_n(q, \mathcal{K}) \}_{n \in \N^*}$ and call it colored Jones function of $\mathcal{K}$.\\
The holonomy of the unified invariant and of the ADO polynomials will follow as a simple application of the two previous results and of the $q$-holonomy of the colored Jones function as shown in \cite{garoufalidis2005colored}. Mainly, there are two operators $Q$ and $E$ on the set of discrete function over $\Z[q^{\pm1}]$ that forms a quantum plane and for any knot $\mathcal{K}$, there is a two variable polynomial $\alpha_{\mathcal{K}}$ such that $\alpha_{\mathcal{K}}(Q,E)J_{\bullet}(q^2, \mathcal{K})=0$.
We say that the colored Jones function is $q$-holonomic.\\
This paper gives a proof that the same polynomial $\alpha_{\mathcal{K}}$, in some similar operators as $Q$ and $E$, annihilates the unified invariant $F_{\infty}(q,A,\mathcal{K})$ and, at roots of unity, annihilates $ADO_r(A, \mathcal{K})$.

\begin{res}
For any $0$-framed knot $\mathcal{K}$ and any $r\in \N^*$:
\begin{itemize}
\item The unified invariant $F_{\infty}(q,A,\mathcal{K})$ is $q$-holonomic.
\item The ADO invariant $ADO_r(A, \mathcal{K})$ is $\zeta_{2r}$-holonomic.
\end{itemize}
Moreover they are annihilated by the same polynomial as of the colored Jones function.\\
(Detailed version: Theorems \ref{thm_unified_holonomy} and \ref{thm_ado_holonomy}).
\end{res}

\begin{rem}
Keep in mind that these results cover only the case of a knot $\mathcal{K}$ in $S^3$.
\end{rem}

\begin{figure}[h!]
\begin{subfigure}[b]{1\textwidth}
 \centering
  \def\svgwidth{190mm}
    \resizebox{120mm}{!}{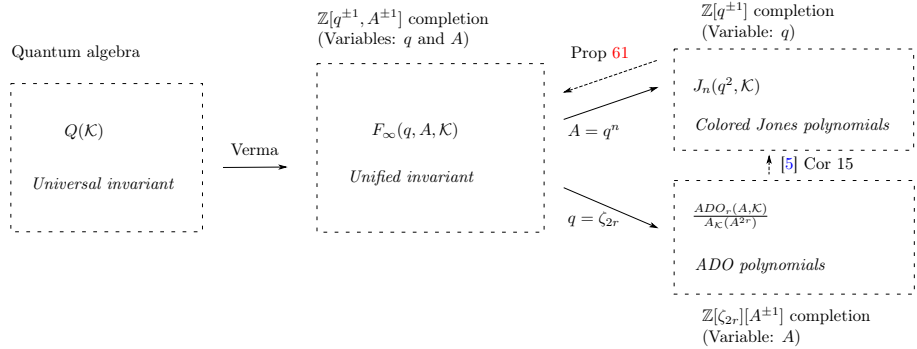}
 \end{subfigure}%
   \caption{Visual representation of the unified knot invariant.}
   \label{intro_schema}
 \end{figure}

\paragraph*{Summary of the paper:}~
\medskip

A way to build a unified element for ADO invariants is to do it by hand. First, one can explicit a formula for the ADO invariant at a $2r$ root of unity by decomposing it as a sum of what we will call state diagrams. This explicit formula will allow us to see what are the obstructions to unify the invariants: first, it will depends on the root of unity $\zeta_{2r}$, secondly the range of the sums coming from the action of the truncated $R$ matrix depends on the order $2r$ of the root of unity. The first obstruction is easy to overcome since taking a formal variable $q$ instead of each occurrence of $\zeta_{2r}$ will do the trick. But for the second one, one could ask that the ranges go to infinity, and this will bring some convergence issues. A way to make these sums convergent is to use a completion of the ring $\Z[q^{\pm 1}, A^{\pm 1}]$ denoted $\hat{R}^{\hat{I}}$, this will allows us to define a good candidate for the unification.\\
But then, we will have to check that this element contains the ADO invariants. We will show that at each root of unity of order $2r$ with $r \in \N^*$, one can define an evaluation map that evaluates $q$ in $\zeta_{2r}$, and that the result can be factorized into a product of an invertible element of the complete ring and the ADO invariant.\\
So we'll get an element containing ADO invariants, but the way we built this element depends on the chosen diagram of the knot. A way to prove that this element is really a knot invariant itself is to recover it with a more advance machinery: the universal invariant of a knot. The universal invariant of a knot was introduced in \cite{lawrence1988universal} and \cite{lawrence1990universal}, the construction can also be found in \cite{ohtsuki2002quantum}. It is a knot invariant and an element of the $h$-adic version of quantum $\mathfrak{sl}_2$, we will use this fact to construct an integral subalgebra in which the universal invariant of a $0$-framed knot will lie. The integrality of the subalgebra will allow us to build a Verma module of it whose coefficients will lie in $\hat{R}^{\hat{I}}$, and on which the scalar action of the universal invariant gives our unified element. A corollary will be that the unified element is a knot invariant.\\
Completions were studied by Habiro in \cite{habiro2007integral}. For the sole purpose of the factorization at roots of unity, we had to use a different completion than the ones mentioned in \cite{habiro2007integral}. But as we will see, we can also recover our unified invariant from his algebraic setup.\\
Once we have this connection between quantum $\mathfrak{sl}_2$ and this unified invariant, we can henceforth relate it also to the colored Jones polynomials, this will allow us to use the Melvin Morton Rozansky conjecture proved by Bar-Natan and Garoufalidis in \cite{bar1996melvin} in order to get some information on the factorization at roots of unity: briefly put, the unified invariant factorize at root of unity as ADO polynomial over the Alexander polynomial.\\
Now we have a unified invariant for both the ADO polynomials and the colored Jones polynomials, the maps recovering them are also well understood. This will allow us to prove that, given the colored Jones polynomials, one may recover the ADO polynomials.\\
From the fact that the colored Jones polynomials recovers the unified invariant and from the factorisation at roots of unity, we will prove that the unified invariant and ADO polynomials follow the same holonomic rule as of the colored Jones function (see \cite{garoufalidis2005colored}). In the same time this paper was made, Brown, Dimofte and Geer got a more general result covering the case of links in their Theorem 4.3 \cite{brown_dimofte_geer2020holonomic}.\\
We will also see that the unified invariant is an integral version of the $h$-adic loop expansion of the colored Jones function and remark that even if it is not clear in general if it's a power serie, it has similar properties as the power serie invariant conjectured by Gukov and Manolescu in \cite{gukov2019two} (Conjectures 1.5, 1.6).\\
Finally we will give some computations of the unified invariant and its factorization at roots of unity, showing how the inverse of the Alexander polynomial appears.
\paragraph*{Acknowledgments:} I would like to thank my Ph.D. advisors François Costantino and Bertrand Patureau-Mirand for their help and guidance.\\ I would also like to thank Christian Blanchet and Anna Beliakova for their useful remarks.

\section{The ADO invariant for knots} \label{ADOsection}
In this article, any knots and links are in $S^3$ and supposed oriented and framed.

\subsection{Definition of the ADO invariants from quantum algebra}
We will expose in this section how to obtain ADO invariants for links \cite{akutsu1992invariants}, also called colored Alexander's polynomials, from a non semi-simple category of module over an unrolled version of $U_q(\mathfrak{sl}_2)$. A more detailed and thorough construction can be found in \cite{costantino2014quantum} \cite{geer2009modified}.

\medskip
For any variable $q$, we denote $ \{ n \} = q^{n}-q^{-n}$, $ [n]= \frac{ \{ n \} }{\{ 1 \}}$, $\{ n \} ! = \prod_{i=1}^n \{ i \}$, $[n]! = \prod_{i=1}^n [i]$, $\qbinom{n}{k}_q = \frac{[n]!}{[n-k]![k]!}$.

\medskip
In order to define ADO invariants for knots, and for the sake of simplicity, in this section $q$ will be an even root of unity.
\begin{defn}
Let $q=e^{\frac{i\pi}{r}}=\zeta_{2r}$ root of unity.\\
We work with an "unrolled" version of $U_{\zeta_{2r}}(\mathfrak{sl}_2)$ denoted by $U_{\zeta_{2r}}^H (\mathfrak{sl}_2)$ and defined as follow:\\
Generators: \[ E, F, K, K^{-1}, H \]
Relations:
$$
\begin{array}{llll}
KK^{-1}=K^{-1} K=1 & KE=\zeta_{2r}^2 EK & KF =\zeta_{2r}^{-2} FK & [E,F]= \frac{K-K^{-1}}{\zeta_{2r}-\zeta_{2r}^{-1}}\\
KH=HK & [H,E]=2E &  [H,F]=-2F &  E^r=F^r=0
\end{array}
$$
\end{defn}
This algebra has a Hopf algebra structure:
$$
\begin{array}{lll}
\Delta(E)=1 \otimes E + E \otimes K & \epsilon(E)=0 &S(E)=-EK^{-1} \\
\Delta(F)=K^{-1} \otimes F + F \otimes 1 & \epsilon(F)=0 &S(E)=-KF \\
\Delta(H)=1 \otimes H + H \otimes 1 & \epsilon(H)=0 &S(H)=-H \\
\Delta(K)=K \otimes K  & \epsilon(K)=1 &S(K)=K^{-1} \\
\Delta(K^{-1})=K^{-1} \otimes K^{-1} & \epsilon(K^{-1})=1 &S(K^{-1})=K
\end{array}
$$

Now, we can look at some category of finite dimensional representation of this algebra and endow it with a ribbon structure.

\begin{defn}
Let $Rep$ be the category of finite dimensional $U_{\zeta_{2r}}^H (\mathfrak{sl}_2)$-modules such that:
\begin{enumerate}
\item The action of $H$ is diagonalizable.
\item The action of $K$ and of $\zeta_{2r}^H$ are the same.
\end{enumerate}
\end{defn}

\begin{prop}
The irreducible representations of $U_{\zeta_{2r}}^H(\mathfrak{sl}_2)$ are $V_{\alpha}$ for $\alpha \in (\C - \Z) \cup r\Z$ and $S_i$ for $i \in \{0, \dots, r-2 \}$ where $S_i$ is the highest weight module of  weight $i$ and dimension $i+1$, $V_{\alpha}$ is the highest weight module of weight $\alpha+r-1$ and dimension $r$.
\end{prop}

\begin{defn}
In $V_{\alpha}$, we say that $x$ has weight level $n$ if $Kx=\zeta_{2r}^{\alpha +r-1 -2n} x$.
\end{defn}

We can endow $Rep$ with a ribbon structure by giving the action of a $R$-matrix and a ribbon element.

We set $F^{(n)}=\frac{\{1\}^n F^n}{[n]!}$ for $0\leq n < r-1$.
\begin{prop}
$R= \zeta_{2r}^{\frac{H \otimes H}{2}}\underset{n=0}{\overset{r-1}{\sum}} \zeta_{2r}^{\frac{n(n-1)}{2}} E^n \otimes F^{(n)}$  is an $R$-matrix whose action is well defined on $Rep$ and it's inverse is $R^{-1} = (\underset{n=0}{\overset{r-1}{\sum}} (-1)^n \zeta_{2r}^{-\frac{n(n-1)}{2}} E^n \otimes F^{(n)}) \zeta_{2r}^{-\frac{H \otimes H}{2}}$.
\end{prop}

\begin{prop}
$K^{1-r}$ is a pivotal element for $U_{\zeta_{2r}} ^H (\mathfrak{sl}_2)$ compatible with the braiding.
\end{prop}

We can now take the usual ribbon functor $RT$ in order to obtain a link invariant, but on the $V_{\alpha}$ it will be $0$ (since the quantum trace is $0$). Hence we need to be more subtle in order to retrieve some information.

On irreducible representations, a 1-1 tangle can be seen as a scalar:\\
If $L$ is a link obtain by closure of a 1-1 tangle $T$, we denote $RT(T) v_0 = ADO(T) v_0$ where $v_0$ is a highest weight vector. Notice that it depends on the 1-1 tangle $T$ chosen. In order to have a link invariant, we must multiply it by a "modified trace".
\bigskip

\noindent We denote $ \{ \alpha \}_{\zeta_{2r}} = \zeta_{2r}^{\alpha}-\zeta_{2r}^{\alpha}$, $\{ \alpha +k \}_{\zeta_{2r}} = \zeta_{2r}^{\alpha +k } - \zeta_{2r}^{-\alpha -k }$, $\{ \alpha;n \}_{\zeta_{2r}}= \prod_{i=0}^{n-1} \{ \alpha -i\}_{\zeta_{2r}}$.

\begin{prop}
If $L$ a link and $T$ is any 1-1 tangle whose closure is $L$ such that the open component is colored with $V_{\alpha}$, set $d(\alpha)= \frac{\{ \alpha \} }{ \{ r \alpha \} }$, then $ADO_r'(L):= d(\alpha) ADO_r(T) $ is a link invariant.
\end{prop}

Although we don't have to specify $\alpha$ and obtain a polynomial in $q^{\alpha}$, we cannot do the same for $q$. The root of unity $q$ must be fixed in order to define the invariant, hence it is a natural question for one to ask how such invariants behave when the root of unity changes.

\subsection{Useful form of the ADO invariant}
From now on we will only work with knots.

\medskip
\noindent To see how the ADO polynomials behave when the root of unity changes, we will explicit a formula for the invariant using the ribbon functor on a diagram $D$ of a knot.

\bigskip
Let $\mathcal{K}$ a knot colored by $V_{\alpha -r+1}$ and $T$ a 1-1 tangle whose closure is $\mathcal{K}$, since we are working with knots $ADO_r(A, \mathcal{K}):=ADO_r(T)$ is well defined. Where $A$ is the free variable $\zeta_{2r}^{\alpha}$.

Let's study this element, by choosing a basis of $V_{\alpha-r+1}$ and computing the invariant with state diagrams.

\begin{rem}
$V_{\alpha-r+1}$ is generated by $v_0, v_1, \dots , v_{r-1}$ where $v_0$ is a highest weight vector, and $v_i= \frac{F^{(i)} v_0}{\{ \alpha;i \}_{\zeta_{2r}} }$.
\end{rem}

\begin{prop}~ \\
$
\begin{array}{ll}
 E v_0=0 & E v_i=v_{i-1}\\
 F v_i= [i+1] [\alpha -i] v_{i+1}& F^{(k)}v_i=\qbinom{k+i}{k}_{\zeta_{2r}} \{ \alpha -i ; k \}_{\zeta_{2r}} v_{k+i} \\
 K v_i=\zeta_{2r}^{\alpha-2i} v_i & \zeta_{2r}^{\frac{H \otimes H}{2}} v_i \otimes v_j = \zeta_{2r}^{\frac{\alpha^2}{2}} \zeta_{2r}^{-(i+j)\alpha} \zeta_{2r}^{2ij} v_i \otimes v_j
\end{array}
$
\end{prop}

\begin{cor}
 We have $ ADO_r(A,\mathcal{K}) \in \zeta_{2r}^{\frac{f \alpha^2}{2}} \Z[\zeta_{2r}, \zeta_{2r}^{\pm \alpha}]$ where $f$ is the framing of the knot.
\end{cor}

\paragraph{}
More precisely, in order to calculate a useful form of this invariant one may look at \textit{state diagram} of a knot. For any knot seen as a $(1,1)$ tangle, take a diagram $D$, label the top and bottom strands $0$ and starting from the bottom strand, label the strand after the $k$-th crossing encountered with the rule described in Figure \ref{crossings_simple}. The resulting diagram is called a \textit{state diagram} of $D$.

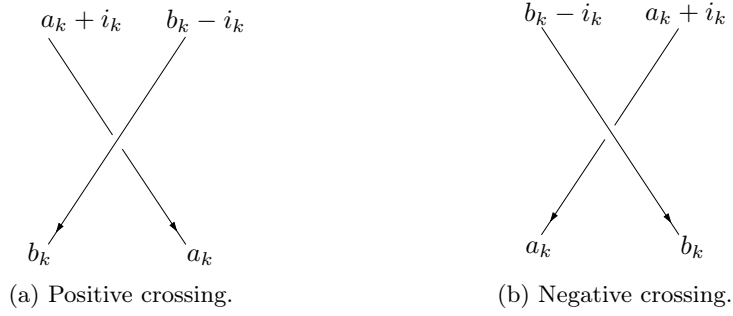
\begin{figure}[h!]
\begin{subfigure}[b]{0.5\textwidth}
 \centering
  \def\svgwidth{25mm}
\begingroup%
  \makeatletter%
  \providecommand\color[2][]{%
    \errmessage{(Inkscape) Color is used for the text in Inkscape, but the package 'color.sty' is not loaded}%
    \renewcommand\color[2][]{}%
  }%
  \providecommand\transparent[1]{%
    \errmessage{(Inkscape) Transparency is used (non-zero) for the text in Inkscape, but the package 'transparent.sty' is not loaded}%
    \renewcommand\transparent[1]{}%
  }%
  \providecommand\rotatebox[2]{#2}%
  \newcommand*\fsize{\dimexpr\f@size pt\relax}%
  \newcommand*\lineheight[1]{\fontsize{\fsize}{#1\fsize}\selectfont}%
  \ifx\svgwidth\undefined%
    \setlength{\unitlength}{666.14173228bp}%
    \ifx\svgscale\undefined%
      \relax%
    \else%
      \setlength{\unitlength}{\unitlength * \real{\svgscale}}%
    \fi%
  \else%
    \setlength{\unitlength}{\svgwidth}%
  \fi%
  \global\let\svgwidth\undefined%
  \global\let\svgscale\undefined%
  \makeatother%
  \begin{picture}(1,1.34042553)%
    \lineheight{1}%
    \setlength\tabcolsep{0pt}%
    \put(0,0){\includegraphics[width=\unitlength,page=1]{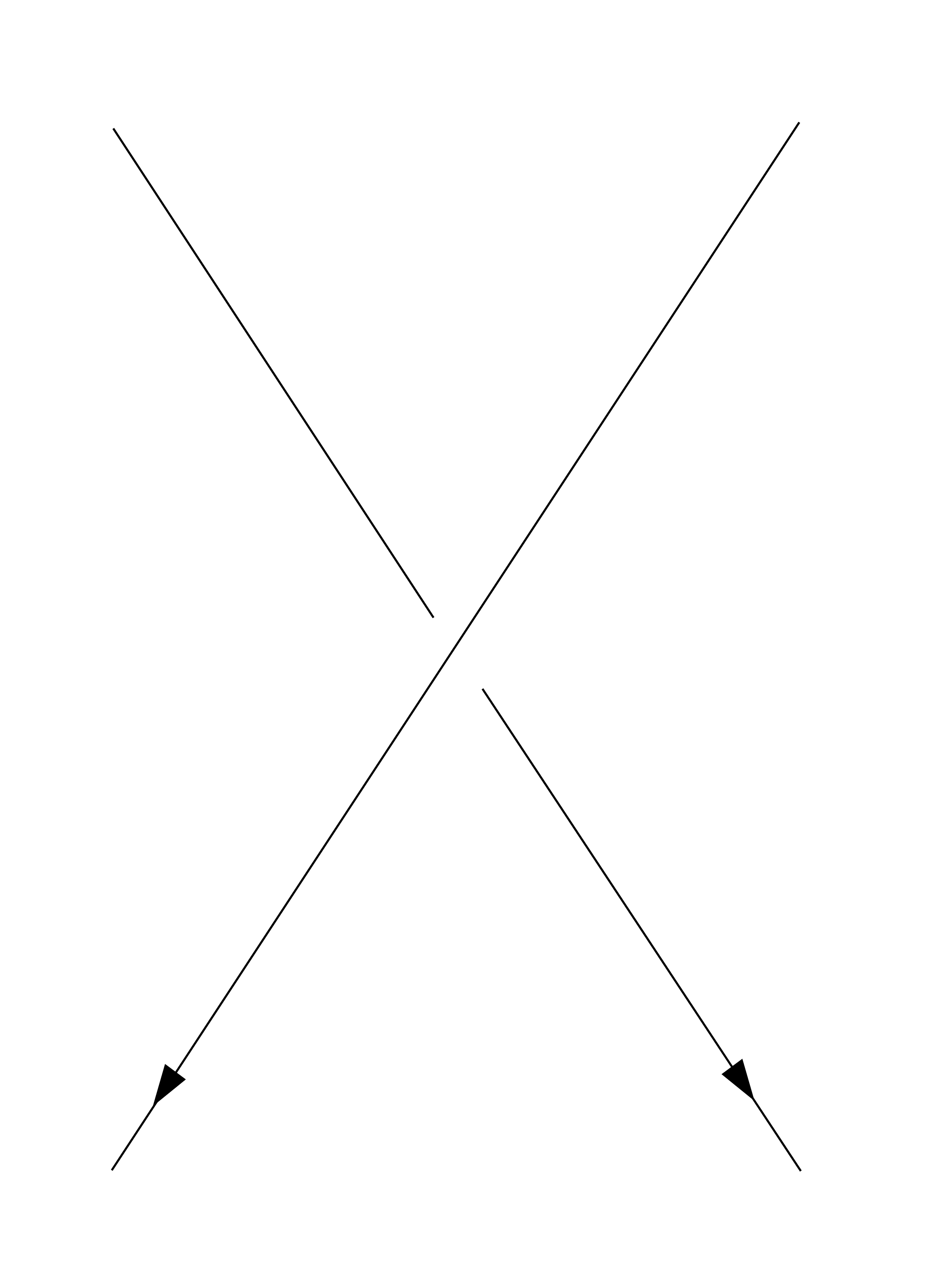}}%
    \put(0.84572423,0.02332179){\color[rgb]{0,0,0}\makebox(0,0)[lt]{\lineheight{1.25}\smash{\begin{tabular}[t]{l}$a_k$\end{tabular}}}}%
    \put(0.00726549,0.02199668){\color[rgb]{0,0,0}\makebox(0,0)[lt]{\lineheight{1.25}\smash{\begin{tabular}[t]{l}$b_k$\end{tabular}}}}%
    \put(0.73748325,1.25886345){\color[rgb]{0,0,0}\makebox(0,0)[lt]{\lineheight{1.25}\smash{\begin{tabular}[t]{l}$b_k -i_k$\end{tabular}}}}%
    \put(0.07803548,1.2652971){\color[rgb]{0,0,0}\makebox(0,0)[lt]{\lineheight{1.25}\smash{\begin{tabular}[t]{l}$a_k +i_k$\end{tabular}}}}%
  \end{picture}%
\endgroup%
  \label{cross_simple}
   \caption{Positive crossing.}
 \end{subfigure}%
 \begin{subfigure}[b]{0.5\textwidth}
 \centering
  \def\svgwidth{25mm}
\begingroup%
  \makeatletter%
  \providecommand\color[2][]{%
    \errmessage{(Inkscape) Color is used for the text in Inkscape, but the package 'color.sty' is not loaded}%
    \renewcommand\color[2][]{}%
  }%
  \providecommand\transparent[1]{%
    \errmessage{(Inkscape) Transparency is used (non-zero) for the text in Inkscape, but the package 'transparent.sty' is not loaded}%
    \renewcommand\transparent[1]{}%
  }%
  \providecommand\rotatebox[2]{#2}%
  \newcommand*\fsize{\dimexpr\f@size pt\relax}%
  \newcommand*\lineheight[1]{\fontsize{\fsize}{#1\fsize}\selectfont}%
  \ifx\svgwidth\undefined%
    \setlength{\unitlength}{666.14173228bp}%
    \ifx\svgscale\undefined%
      \relax%
    \else%
      \setlength{\unitlength}{\unitlength * \real{\svgscale}}%
    \fi%
  \else%
    \setlength{\unitlength}{\svgwidth}%
  \fi%
  \global\let\svgwidth\undefined%
  \global\let\svgscale\undefined%
  \makeatother%
  \begin{picture}(1,1.36170213)%
    \lineheight{1}%
    \setlength\tabcolsep{0pt}%
    \put(0.02895136,0.06594472){\color[rgb]{0,0,0}\makebox(0,0)[lt]{\lineheight{1.25}\smash{\begin{tabular}[t]{l}$a_k$\end{tabular}}}}%
    \put(0.84393733,0.06301101){\color[rgb]{0,0,0}\makebox(0,0)[lt]{\lineheight{1.25}\smash{\begin{tabular}[t]{l}$b_k$\end{tabular}}}}%
    \put(0.02204024,1.29666105){\color[rgb]{0,0,0}\makebox(0,0)[lt]{\lineheight{1.25}\smash{\begin{tabular}[t]{l}$b_k -i_k$\end{tabular}}}}%
    \put(0.65897041,1.30148629){\color[rgb]{0,0,0}\makebox(0,0)[lt]{\lineheight{1.25}\smash{\begin{tabular}[t]{l}$a_k +i_k$\end{tabular}}}}%
    \put(0,0){\includegraphics[width=\unitlength,page=1]{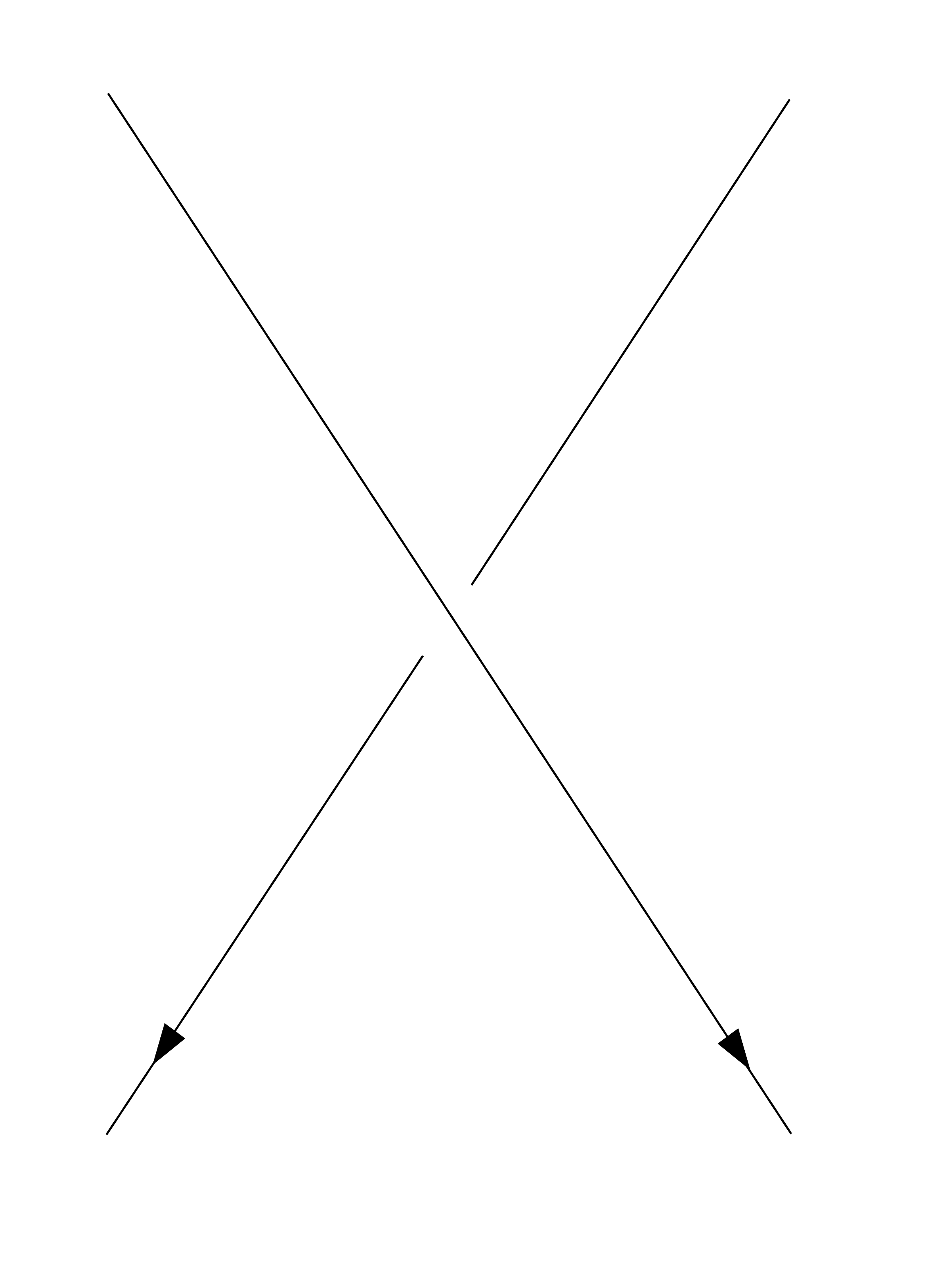}}%
  \end{picture}%
\endgroup%
  \label{negcross_simple}
   \caption{Negative crossing.}
 \end{subfigure}%
 \caption{The two possibilities for the k-th crossing in $D$.}
 \label{crossings_simple}
 \end{figure}

\paragraph{}
Let $\mathcal{K}$ a knot and $D$ a diagram of the knot seen as a $(1,1)$ tangle. Suppose the diagram has $N$ crossings.
Now for any state diagram of $D$ we can associate an element:
\begin{align*}
D_r(i_1, \dots, i_N) = &(\prod_{j=1}^{S} \zeta_{2r}^{\pm (r-1)(\alpha -2 \epsilon_j)})\prod_{k \in pos} \zeta_{2r}^{\frac{i_k(i_k-1)}{2}} \qbinom{a_k+i_k}{i_k}_{\zeta_{2r}}  \{\alpha-a_k; i_k\}_{\zeta_{2r}} \\ &  \times  \zeta_{2r}^{-(a_k+b_k) \alpha} \zeta_{2r}^{2(a_k+ i_k)(b_k-i_k)} \prod_{k \in neg}  (-1)^{i_k}\zeta_{2r}^{-\frac{i_k(i_k-1)}{2}} \\ &  \times  \qbinom{a_k+i_k}{i_k}_{\zeta_{2r}}  \{\alpha-a_k; i_k\}_{\zeta_{2r}} \zeta_{2r}^{(a_k+b_k) \alpha} \zeta_{2r}^{- 2a_k b_k}
\end{align*}

where $neg \ \cup \ pos = [|1, N|]$ and $k \in pos$ if the $k$-th crossing of D is positive, else $k \in neg$. $a_k, b_k$ are the strands labels at the $k$-th crossing of the state diagram (see Figure \ref{crossings_simple}), $S$ is the number of \def\svgwidth{5mm}
\begingroup%
  \makeatletter%
  \providecommand\color[2][]{%
    \errmessage{(Inkscape) Color is used for the text in Inkscape, but the package 'color.sty' is not loaded}%
    \renewcommand\color[2][]{}%
  }%
  \providecommand\transparent[1]{%
    \errmessage{(Inkscape) Transparency is used (non-zero) for the text in Inkscape, but the package 'transparent.sty' is not loaded}%
    \renewcommand\transparent[1]{}%
  }%
  \providecommand\rotatebox[2]{#2}%
  \newcommand*\fsize{\dimexpr\f@size pt\relax}%
  \newcommand*\lineheight[1]{\fontsize{\fsize}{#1\fsize}\selectfont}%
  \ifx\svgwidth\undefined%
    \setlength{\unitlength}{867.71048959bp}%
    \ifx\svgscale\undefined%
      \relax%
    \else%
      \setlength{\unitlength}{\unitlength * \real{\svgscale}}%
    \fi%
  \else%
    \setlength{\unitlength}{\svgwidth}%
  \fi%
  \global\let\svgwidth\undefined%
  \global\let\svgscale\undefined%
  \makeatother%
  \begin{picture}(1,0.55772942)%
    \lineheight{1}%
    \setlength\tabcolsep{0pt}%
    \put(0,0){\includegraphics[width=\unitlength,page=1]{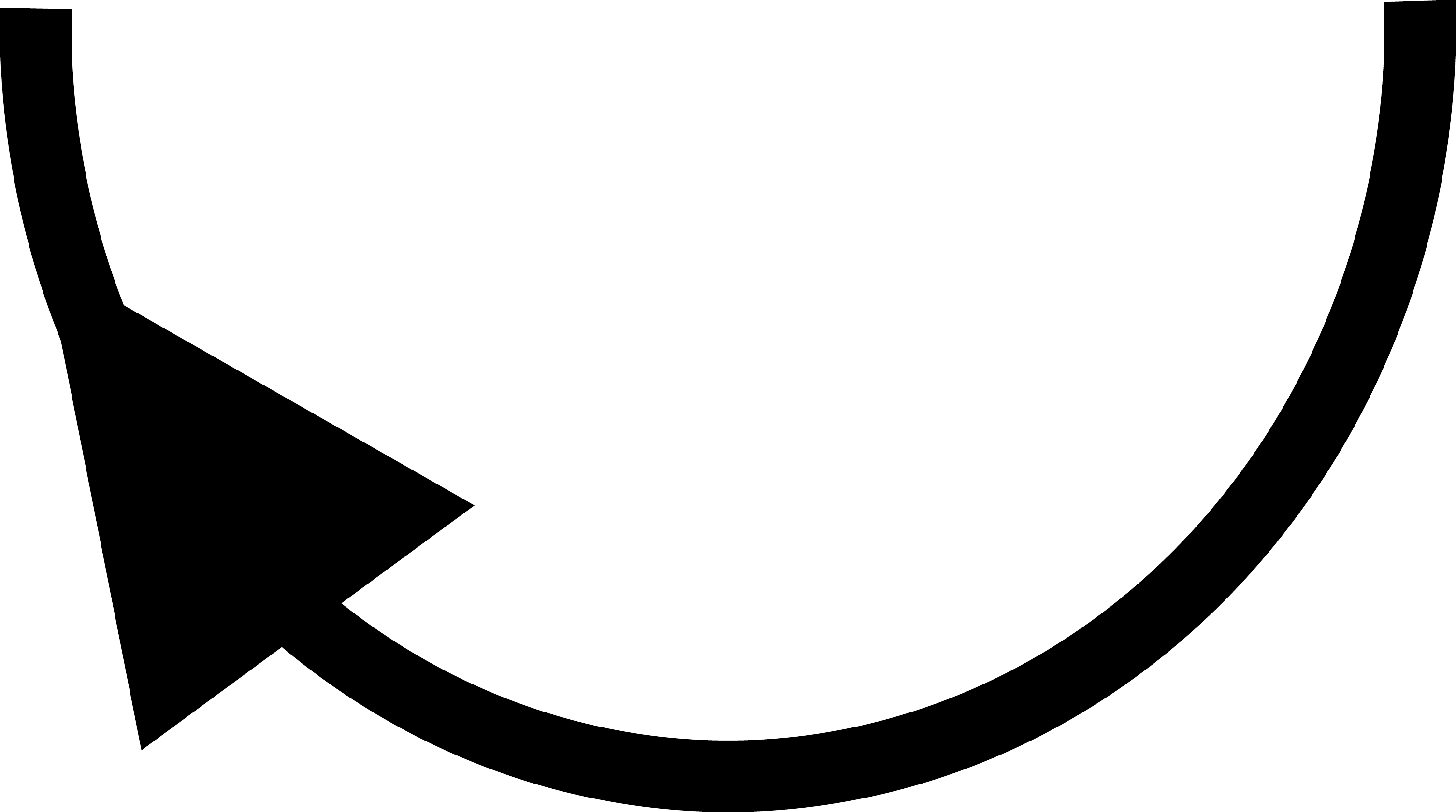}}%
  \end{picture}%
\endgroup%
 $+$ \def\svgwidth{5mm}
\begingroup%
  \makeatletter%
  \providecommand\color[2][]{%
    \errmessage{(Inkscape) Color is used for the text in Inkscape, but the package 'color.sty' is not loaded}%
    \renewcommand\color[2][]{}%
  }%
  \providecommand\transparent[1]{%
    \errmessage{(Inkscape) Transparency is used (non-zero) for the text in Inkscape, but the package 'transparent.sty' is not loaded}%
    \renewcommand\transparent[1]{}%
  }%
  \providecommand\rotatebox[2]{#2}%
  \newcommand*\fsize{\dimexpr\f@size pt\relax}%
  \newcommand*\lineheight[1]{\fontsize{\fsize}{#1\fsize}\selectfont}%
  \ifx\svgwidth\undefined%
    \setlength{\unitlength}{867.71048959bp}%
    \ifx\svgscale\undefined%
      \relax%
    \else%
      \setlength{\unitlength}{\unitlength * \real{\svgscale}}%
    \fi%
  \else%
    \setlength{\unitlength}{\svgwidth}%
  \fi%
  \global\let\svgwidth\undefined%
  \global\let\svgscale\undefined%
  \makeatother%
  \begin{picture}(1,0.55772942)%
    \lineheight{1}%
    \setlength\tabcolsep{0pt}%
    \put(0,0){\includegraphics[width=\unitlength,page=1]{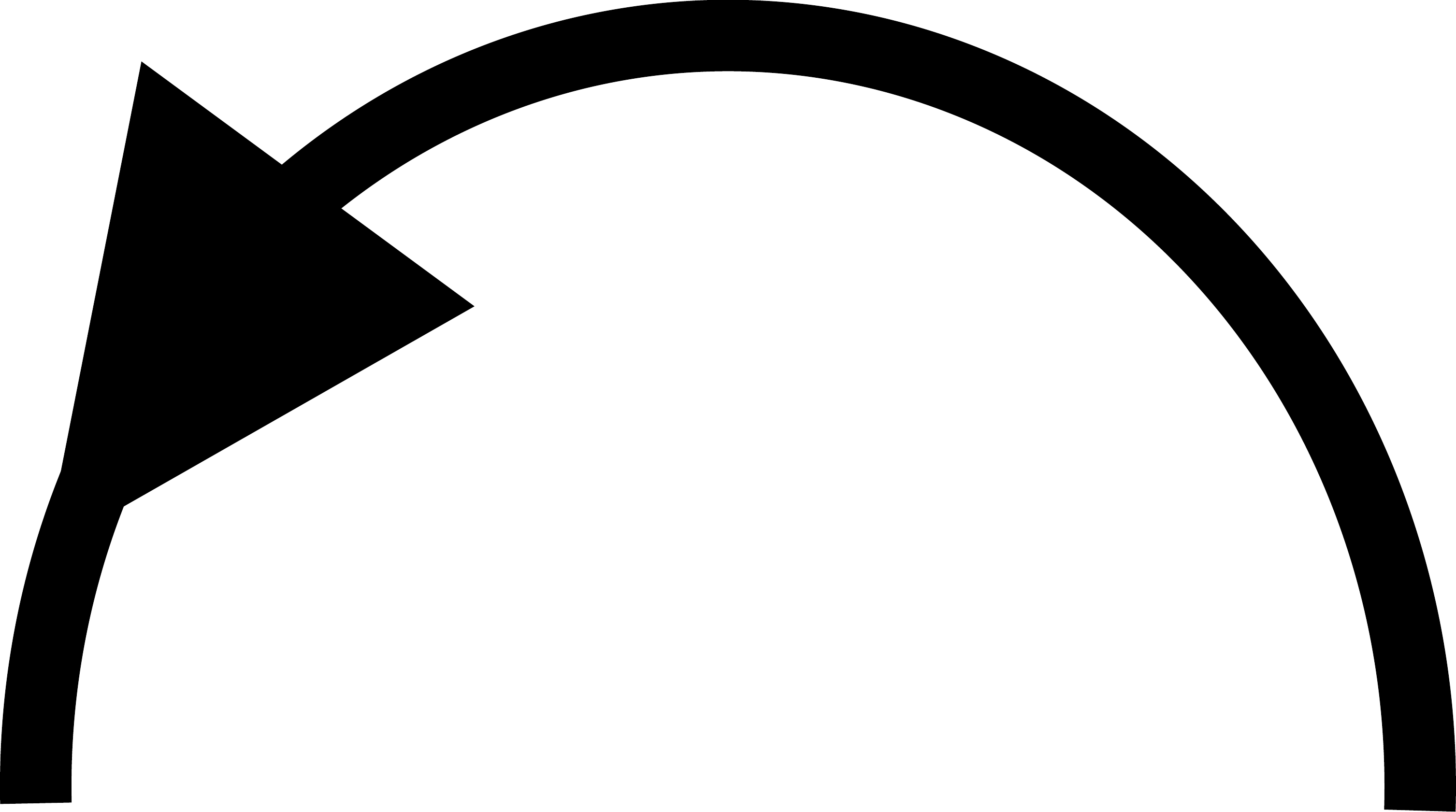}}%
  \end{picture}%
\endgroup%
 appearing in the diagram, and $\epsilon_j$ the strand label at the j-th \def\svgwidth{5mm} or \def\svgwidth{5mm}, the $\pm$ sign is positive if \def\svgwidth{5mm} and negative if \def\svgwidth{5mm}.

\begin{rem}
Note that the $a_k$ and $b_k$ appearing are defined in terms of $i_j$. You can find some examples of state diagrams in Section \ref{section_compute} Figure \ref{trefoil_w}, \ref{figure8_w}, \ref{torus_2_5_w}, \ref{figure_of_nine_w}.
\end{rem}

\begin{prop} \label{ADO_form}
If $D$ is a diagram of $\mathcal{K}$ seen as a a 1-1 tangle we have:
\begin{align*}
ADO_r(A,\mathcal{K})& = \zeta_{2r}^{\frac{f \alpha^2}{2}} \underset{\overline{i}=0}{\overset{r-1}{\sum}} D_r(i_1, \dots, i_N)\\
&= \zeta_{2r}^{\frac{f \alpha^2}{2}} \underset{\overline{i}=0}{\overset{r-1}{\sum}} (\prod_{j=1}^{S} \zeta_{2r}^{\pm (r-1)(\alpha -2 \epsilon_j)})\prod_{k \in pos} \zeta_{2r}^{\frac{i_k(i_k-1)}{2}} \qbinom{a_k+i_k}{i_k}_{\zeta_{2r}} \\ &  \times  \{\alpha-a_k; i_k\}_{\zeta_{2r}} \zeta_{2r}^{-(a_k+b_k) \alpha} \zeta_{2r}^{2(a_k+ i_k)(b_k-i_k)}  \prod_{k \in neg}  (-1)^{i_k} \\ &  \times \zeta_{2r}^{-\frac{i_k(i_k-1)}{2}} \qbinom{a_k+i_k}{i_k}_{\zeta_{2r}}  \{\alpha-a_k; i_k\}_{\zeta_{2r}} \zeta_{2r}^{(a_k+b_k) \alpha} \zeta_{2r}^{- 2a_k b_k}
\end{align*}

where $\overline{i}= (i_1, \dots, i_N)$, $N$ is the number of crossings, $S$ the number of \def\svgwidth{5mm} $+$ \def\svgwidth{5mm} and $f$ is the framing of the knot.

\end{prop}

\begin{proof}
Notice that  $\zeta_{2r}^{\frac{f \alpha^2}{2}}  D_r(i_1, \dots, i_N)$ is the element obtained by adding to the $k$-th crossing  a coupon labeled with: $q^{\frac{H \otimes H}{2}} q^{\frac{i_k(i_k-1)}{2}}  E^{i_k}  \otimes F^{(i_k)}$ if positive and\\ $q^{\frac{-H \otimes H}{2}} q^{\frac{-i_k(i_k-1)}{2}} E^{i_k}  \otimes F^{(i_k)}$ if negative. Then add a coupon to \def\svgwidth{5mm} labeled $K^{r-1}$ and \def\svgwidth{5mm} labeled $K^{1-r}$. We get an element of $U_{\zeta_{2r}}^H(\mathfrak{sl}_2)$, its action on $v_0 \in V_{\alpha+r-1}$ the highest weight vector gives the element $\zeta_{2r}^{\frac{f \alpha^2}{2}}  D_r(i_1, \dots, i_N)$. Summing them over $i_k$ for all $k$ gives the ADO polynomial.

\end{proof}
\medskip

Now that we have an explicit formula, can we construct from it a suitable element that can be evaluated at roots of unity and recover the ADO invariants?

We have two main issues here, first of all, $\zeta_{2r}$ appears in the formula, we will have to replace each occurrence with some variable $q$, in order to see it as a polynomial or a formal series.

The second one is more difficult to solve: the action of the $R$ matrices makes appear sums that range to $r-1$, which depends on the order of the root of unity. A solution to this problem, as we will explicit it, is to let the sum range to infinity and define a ring in which such sums converge. Then we will see how to factorize the ADO invariant from this new unified form.

\section{Unified form for ADO invariants of knots} \label{unifysection}

The approach here will be to unify the invariants: using completions of rings and algebras, we will explicit an integral invariant in some variable $q$ that can be evaluated at any root of unity, recovering ADO invariants defined previously.

The first subsection will create the right setup to define a unified form inspired by the useful form of the ADO invariant in Proposition \ref{ADO_form}. Using a completion of the ring of integral Laurent polynomials in two variables $q, A$, we define a unified form $F_{\infty}(q,A,D)$ by taking the previous form of ADO, replacing the root of unity $ \zeta_{2r}$ by $q$, $\zeta_{2r}^{\alpha}$ by $A$ and letting the truncated sums coming from the $R$-matrices action go to infinity. Note that at this point, the defined form is not a knot invariant, as it \textit{a priori} depends on the diagram $D$ of the knot.

The second subsection will make the bridge between the first two sections. By evaluating the unified form at roots of unity $\zeta_{2r}$ with $r \in \N^*$, we factor out the ADO invariant. We will then explicit a map sending the unified form of a knot to the corresponding ADO invariants. This will show that the ADO invariants are contained in the unified form and that we can recover them from it.

\subsection{Ring completion for the unified form}
Let's lay the groundwork for an unified form to exist. It must be a ring in which infinite sums previously mentioned converge.
\bigskip

Let $R=\Z[q^{\pm 1}, A^{\pm 1}]$, we will construct a completion of that ring. For the sake of simplicity, we will denote $q^{\alpha} := A$ and use previous notation for quantum numbers. Keep in mind that, here, $\alpha$ is just a notation, not a complex number.

\noindent We denote $ \{ \alpha \}_q = q^{\alpha}-q^{\alpha}$, $\{ \alpha +k \}_q = q^{\alpha +k } - q^{-\alpha -k }$, $\{ \alpha;n \}_q= \prod_{i=0}^{n-1} \{ \alpha -i\}_q$.

\begin{defn}
Let $I_n$ be the ideal of $R$ generated by the following set $\{ \ \{ \alpha+l; n \}_q , \ l \in \Z \}$.
\end{defn}

\begin{lemme}
$I_n$ is generated by elements of the form $\{n;i\} \{ \alpha;n-i\}$, $i \in \{0, \dots n \}$.
\end{lemme}

\begin{proof}
The proof can be found in Habiro's article \cite{habiro2007integral}:\\
Replacing $K$ (resp. $K^{-1}$) by $q^{\alpha}$ (resp. $q^{-\alpha}$) in Proposition 5.1, one gets the proof of this lemma. 
\end{proof}

We then have a projective system : \[ \hat{I} : I_1 \supset I_2 \supset \dots \supset I_n \supset \dots \]
From which we can define the completion of $R$, taking the projective limit:

\begin{defn}
Let $\hat{R}^{\hat{I}} = \underset{\underset{n}{\leftarrow}}{\lim} \dfrac{R}{I_n} = \{ (a_n)_{n \in \N^*} \in \prod_{i=1}^{\infty} \frac{R}{I_n} \ | \ \rho_n(a_{n+1})=a_n \}$  where $\rho_n: \frac{R}{I_{n+1}} \to \frac{R}{I_{n}}$ is the projection map.
\end{defn}

This completion is a bigger ring containing $R$:
\begin{prop}
The canonical projection maps induce an injective map $R \xhookrightarrow{} \hat{R}^{\hat{I}}$
\end{prop}
\begin{proof}
It is sufficient to prove that $\underset{n \in \N^*}{\bigcap} I_n =\{0\}$.

\noindent Since $R= \Z[q^{\pm 1}][A^{\pm1}]$, it is a Laurent polynomial ring. Let's define $\deg_q(x)$, $val_q(x)$ the degree and valuation of $x$ in the variable $q$.
\smallskip

Let $f_k: \Z[q^{\pm1}, A^{\pm1}] \to \Z[ q^{\pm1}]$, $A \mapsto q^k$. We have  $f_k (I_n) \subset \{n\}!  \Z[ q^{\pm1}]$ because $I_n$ is generated by elements of the form $\{n;i\} \{ \alpha;n-i\}$ that maps to $\{n;i\} \{ k;n-i\}$, which is divisible by $\{n\}! $.

Hence if $x \in \underset{n \in \N^*}{\bigcap} I_n$, $f_k(x) \in \{n\}!\Z[ q^{\pm1}]$ for all $n$, since $\Z[ q^{\pm1}]$ is factorial, $f_k(x)=0$ for all $k$.

Take $x \in \underset{n \in \N^*}{\bigcap} I_n$, written $x= \sum a_n A^n$ with $a_n \in \Z[ q^{\pm1}]$. Take $N$ such that $\deg_q(x)<N$ and $val_q(x)>-N$.\\
This implies that $deg_q(a_n) <N$ and $val_q(a_n)>-N$ (since it is the case for $x$ and any higher or lower terms could not compensate since the power of $A$ is different before each $a_n$).

Thus since $f_{2N}(x)=0$, $\sum a_n q^{2Nn}=0$, we have $\deg_q(a_n q^{2Nn}) < N(1+2n)$ and $val_q(a_n q^{2Nn})>N(2n-1)$, then each terms $a_n q^{2Nn}$ must be $0$. Hence $a_n=0$ for all $n$, meaning that $x=0$.

\end{proof}

\begin{rem}
 If $b_0 \in R$ and $b_n \in I_{n-1}$ for $n \geq 1$, the partial sums $ \underset{i=0}{\overset{N}{\sum}} b_n $ converges in $\hat{R}^{\hat{I}}$ as $N$ goes to infinity.\\
 We denote the limit $\underset{i=0}{\overset{+\infty}{\sum}} b_n := (\overline{\underset{i=0}{\overset{N}{\sum}} b_n})_{N \in \N^*}$.\\
 Conversely, if $a= (\overline{a_N})_{N  \in \N^*} \in \hat{R}^{\hat{I}}$, let $a_n \in R$ be any representative of $\overline{a_n}$ in $R$, then $a= \underset{i=0}{\overset{+\infty}{\sum}} b_n$ where $b_0=a_1$ and $b_{n}=a_{n+1}-a_n$ for $n \in \N^*$.
\end{rem}

\paragraph{}
We proceed similarly as in the paragraph preceding Proposition \ref{ADO_form}. Let $\mathcal{K}$ be a knot seen as a $(1,1)$ tangle and $D$ a diagram of it. For a state diagram of $D$ we define:
\begin{align*}
D(i_1, \dots, i_N) = &(\prod_{j=1}^{S} q^{\mp(\alpha -2 \epsilon_j)})\prod_{k \in pos} q^{\frac{i_k(i_k-1)}{2}} \qbinom{a_k+i_k}{i_k}_{q}  \{\alpha-a_k; i_k\}_{q} \\ &  \times  q^{-(a_k+b_k) \alpha} q^{2(a_k+ i_k)(b_k-i_k)} \prod_{k \in neg}  (-1)^{i_k} q^{-\frac{i_k(i_k-1)}{2}} \qbinom{a_k+i_k}{i_k}_{q} \\ &  \times  \{\alpha-a_k; i_k\}_{q} q^{(a_k+b_k) \alpha} q^{- 2a_k b_k}
\end{align*}
where $neg \ \cup \ pos = [|1, N|]$ and $k \in pos$ if the $k$-th crossing of D is positive, else $k \in neg$. $a_k, b_k$ are the strands labels at the $k$-th crossing of the state diagram (see Figure \ref{crossings_simple}), $S$ is the number of \def\svgwidth{5mm} $+$ \def\svgwidth{5mm} appearing in the diagram, and $\epsilon_j$ the strand label at the j-th \def\svgwidth{5mm} or \def\svgwidth{5mm}, the $\mp$ sign is negative if \def\svgwidth{5mm} and positive if \def\svgwidth{5mm}.

\begin{rem}
Note that the $a_k$ and $b_k$ appearing are defined in terms of $i_j$. As mentioned previously, you can find some examples of state diagrams in Section \ref{section_compute} Figure \ref{trefoil_w}, \ref{figure8_w}, \ref{torus_2_5_w}, \ref{figure_of_nine_w}.
\end{rem}

\begin{defn}\label{unified_form}
Let $\mathcal{K}$ a knot and $T$ 1-1 tangle whose closure is $\mathcal{K}$. Let $D$ be a diagram of $T$.\\
We define:
\begin{align*}
F_{\infty}(q,A,D)&:= q^{\frac{f \alpha^2}{2}} \underset{\overline{i}=0}{\overset{+\infty}{\sum}}D(i_1, \dots, i_N)\\
&= q^{\frac{f \alpha^2}{2}} \underset{\overline{i}=0}{\overset{+\infty}{\sum}} (\prod_{j=1}^{S} q^{\mp(\alpha -2 \epsilon_j)})\prod_{k \in pos} q^{\frac{i_k(i_k-1)}{2}} \qbinom{a_k+i_k}{i_k}_{q}  \{\alpha-a_k; i_k\}_{q} \\ &  \times  q^{-(a_k+b_k) \alpha} q^{2(a_k+ i_k)(b_k-i_k)} \prod_{k \in neg}  (-1)^{i_k} q^{-\frac{i_k(i_k-1)}{2}} \qbinom{a_k+i_k}{i_k}_{q}  \\ &  \times \{\alpha-a_k; i_k\}_{q} q^{(a_k+b_k) \alpha} q^{- 2a_k b_k}
\end{align*}
where $\overline{i}= (i_1, \dots, i_N)$, $N$ is the number of crossings, $S$ the number of \def\svgwidth{5mm} $+$ \def\svgwidth{5mm} and $f$ is the framing of the knot.

We have that $F_{\infty}(q, A, D)$ is a well defined element of $q^{\frac{f \alpha^2}{2}} \hat{R}^{\hat{I}}$.
\end{defn}

Note that it is not clear that this element is a knot invariant, it could depend \textit{a priori} on the diagram $D$ and we will have to prove later that it does not.

\subsection{Recovering the ADO invariant}
In this subsection we will see how to evaluate at a root of unity an element of $\hat{R}^{\hat{I}}$. We will first need some useful lemma.

Let $r$ any integer, $R_r=\Z[\zeta_{2r}, A^{\pm1}]$, we use the same previous notations and $\zeta_{2r}^{\alpha} := A$.
\begin{lemme} For any $k$,
$\{ \alpha -k ; r\}_{\zeta_{2r}}= (-1)^k \zeta_{2r}^{\frac{-r(r-1)}{2}} \{r\alpha\}_{\zeta_{2r}}$.
\label{ralphaLemma}
\end{lemme}
\begin{proof}~

$
 \begin{array}{lll}
 \{ \alpha-k; r\} &= \{ \alpha-k \} \dots \{ \alpha-k-r+1 \} \\
  &= (-1)  \{ \alpha-k+1 \} \dots \{ \alpha-k-r+2 \} \\
  &= (-1)^k  \{ \alpha \} \dots \{ \alpha-r+1 \} \\
  &= (-1)^k  \{ \alpha; r \}
 \end{array}
 $

$
\begin{array}{lll}
\{ \alpha; r \} &= \prod_{j=0}^{r-1} (\zeta_{2r}^{\alpha-j} - \zeta_{2r}^{- \alpha+j}) \\
&= \zeta_{2r}^{\frac{-r(r-1)}{2}} \zeta_{2r}^{-r \alpha} \prod_{j=0}^{r-1} (\zeta_{2r}^{2 \alpha} -\zeta_{2r}^{2 j}) \\
&=\zeta_{2r}^{\frac{-r(r-1)}{2}} \zeta_{2r}^{-r \alpha} (\zeta_{2r}^{2 r\alpha} -1)\\
&= \zeta_{2r}^{\frac{-r(r-1)}{2}} \{r \alpha \}
\end{array}
$

where the fourth equality is just developing the factorized form of $X^r -1$ at $X=\zeta_{2r}^{2\alpha}$.

\end{proof}

\medskip
Let $I=\{ r \alpha \}_{\zeta_{2r}} R_r$ and we build the $I$-adic completion of $R_r$:
\begin{defn}
$\hat{R_r^I}=\underset{\underset{n}{\leftarrow}}{\lim} \dfrac{R_r}{I^n} = \{ (a_n)_{n \in \N^*} \in \prod_{i=1}^{\infty} \frac{R_r}{I^n} \ | \ \rho'_n(a_{n+1})=a_n \}$  where $\rho'_n: \frac{R_r}{I^{n+1}} \to \frac{R_r}{I^{n}}$ is the projection map.
\end{defn}

This completion is a bigger ring containing $R_r$:
\begin{prop}
The canonical projection maps induce an injective map $R_r \xhookrightarrow{} \hat{R_r^I}$
\end{prop}
\begin{proof}
It is sufficient to prove that $\underset{n \in \N^*}{\bigcap} I^n =\{0\}$.\\
Since $R_r= \Z[\zeta_{2r}][A^{\pm 1}]$, it is a Laurent polynomial ring. Hence, any non zero element $x$ can be uniquely written $x=\sum_{i=l}^n a_n A^n$ where $a_k \in \Z[\zeta_{2r}], \forall k \in \{l, l+1, \dots, n-1, n \}$ and $a_n, a_l \neq 0$.\\
Let's define $len(x)= n-l$. We have that $len(xy)=len(x) + len(y)$.\\
Thus if $x \in \underset{n \in \N^*}{\bigcap} I^n$ is non zero, of length $n$, $\exists y \in R_r$ such that $x=\{r \alpha \}^n y$ hence $len(x) =2rn + len(y)$, contradiction.

\end{proof}
\bigskip
Let's now define the evaluation map from $\hat{R}^{\hat{I}}$ to $\hat{R_r^I}$.

At the level of $R$ and $R_r$ we have a well defined evaluation map, $ev_{\zeta_{2r}}: R \to R_r, \ q \mapsto \zeta_{2r}$. We will extend this map to the completions.

\begin{prop}
$ev_{\zeta_{2r}}(I_{rn}) = I^n$
\end{prop}
\begin{proof}
Direct application of Lemma~\ref{ralphaLemma}. 
\end{proof}
\bigskip
Hence, $ev_r$ factorize into maps $\psi_n: R/I_{rn} \to R_r / I^n$, we can then define the map extension:
\begin{prop}
We have a well defined map: \[ ev_r: \hat{R}^{\hat{I}} \to \hat{R_r^I} \] such that, if $(a_n)_{n \in \N^*} \in \hat{R}^{\hat{I}}$, $ev_r((a_n)_{n \in \N^*}) = (\psi_n(a_{rn}))_{n \in \N^*}$.

\end{prop}

\begin{proof}
If we denote $\lambda_n : R/I_{r(n+1)} \to R/I_{rn}$ the projective maps, it lie on the fact that the following diagram is commutative:
\begin{center}
\begin{tikzpicture}
  \matrix (m) [matrix of math nodes,row sep=3em,column sep=4em,minimum width=2em]
  {
     R/I_{r(n+1)} & R_r / I^{n+1} \\
     R/I_{rn} & R_r / I^n \\};
  \path[-stealth]
    (m-1-1) edge node [left] {$\lambda_n$} (m-2-1)
            edge node [below] {$\psi_{n+1}$} (m-1-2)
    (m-2-1.east|-m-2-2) edge node [below] {$\psi_{n}$} (m-2-2)
    (m-1-2) edge node [right] {$\rho_n'$} (m-2-2);
\end{tikzpicture}
\end{center}

\end{proof}

\bigskip

It is now time to study the element $F_{\infty} (\zeta_{2r}, A, D) := ev_r( F_{\infty}(q,A,D))$, we will see that the ADO invariant $ADO_r(A,\mathcal{K})$ can be factorized from it.\\
In order to do so, we will need some useful computations:

\begin{lemme}
We have the following factorizations:
\begin{itemize}
\item $\zeta_{2r}^{\frac{(i+rl)(i+rl-1)}{2}}=(-1)^{il} \zeta_{2r}^{\frac{rl(rl-1)}{2}} \zeta_{2r}^{\frac{i(i-1)}{2}}$,
\item $\{ \alpha -a -ru; i +rl\}_{\zeta_{2r}} = (-1)^{al+rul +ui+li} \zeta_{2r}^{\frac{-rl(r-1)}{2}} \zeta_{2r}^{\frac{-rl(l-1)}{2}} \{ r\alpha \}_{\zeta_{2r}}^l \{ \alpha -a; i\}_{\zeta_{2r}}$,
\item $\qbinom{a+i+ r(u+l)}{i+rl}_{\zeta_{2r}} =(-1)^{al+rul+ui} \binom{u+l}{l} \qbinom{a+i}{i}_{\zeta_{2r}}$
\item $ \zeta_{2r}^{\frac{-rl(r-1)}{2}} \zeta_{2r}^{\frac{-rl(l-1)}{2}} = \zeta_{2r}^{\frac{-rl(rl-1)}{2}}$
\end{itemize}
\label{factorLemma}
\end{lemme}

\begin{proof}~
\begin{enumerate}
\item The first dot is obtained by developing the product.
\item The second dot is an application of Lemma \ref{ralphaLemma}.\\
First $\{ \alpha -a -ru;i +rl\}= \zeta_{2r}^{(i+rl)ru} \{ \alpha -a;i +rl\}=(-1)^{iu}(-1)^{rul}\{ \alpha -a;i +rl\}$.\\
Then,
\begin{flalign*}
\{ \alpha -a;i +rl\}& = \{ \alpha -a;rl\}\{ \alpha -a-rl;i\} \\
& = (-1)^{al} \zeta_{2r}^{\frac{-rl(l-1)}{2}} \{\alpha; rl\} \{ \alpha -a -rl;i\} \\
& =(-1)^{al} \zeta_{2r}^{\frac{-rl(l-1)}{2}} \zeta_{2r}^{\frac{-rl(r-1)}{2}} \{r\alpha\}^l \{ \alpha -a -rl;i\} 
\end{flalign*}
Finally $\{ \alpha -a -rl;i\}= (-1)^{li} \{ \alpha -a ; i \}$.\\
Put together, we get \begin{multline*} \{ \alpha -a -ru; i +rl\}_{\zeta_{2r}} = (-1)^{al+rul+ui+li}\zeta_{2r}^{\frac{-rl(r-1)}{2}} \zeta_{2r}^{\frac{-rl(l-1)}{2}} \{ r\alpha \}_{\zeta_{2r}}^l \\ \times \{ \alpha -a; i\}_{\zeta_{2r}}. \end{multline*} 
\item The third dot follows from the fact that $ev_{\zeta_{2r}}( \frac{\{rk \}_q }{\{r \}_q })=(-1)^{1-k} k$. In $\qbinom{a+i+ r(u+l)}{i+rl}_{\zeta_{2r}}$ seen as $\frac{\{a+i+ r(u+l)\}!}{\{ a+ru\}! \{ i + rl\}!}$, taking only the terms $\{ rk \}$, we extract $(-1)^{ul} \binom{u+l}{l}$. Now we only have to deal with non multiples of quantum $r$. We use the equality $\{t +r \}= (-1) \{t\}$ in order to have consecutive terms in the denominators (excepted from multiple of $r$), indeed  $\{a + ru \}!= \{ ru\}! \{ a+ru;a \}$ and $ \{i+rl\}) (-1)^{u(i+rl)} \{ i + rl +ru ; i + rl\} $, hence $\frac{\{a+i+ r(u+l)\}!}{\{ a+ru\}! \{ i + rl\}!}=(-1)^{u(i+rl)}  \frac{\{a+i+ r(u+l); a\}}{\{ a+ru;a\}}= (-1)^{u(i+rl)} (-1)^{au} (-1)^{a(u+l)} \frac{\{a+i; a\}}{\{ a;a\}}= (-1)^{ui}(-1)^{rul} (-1)^{al} \qbinom{a+i}{i}_{\zeta_{2r}}$. \\
Putting things together with the quantum $r$ multiple part, we get the desired result.
\item The last dot is obtained as follow: \begin{multline*}\zeta_{2r}^{\frac{-rl(rl-1)}{2}} = \prod_{k=0}^{rl-1} \zeta_{2r}^{-k}=\prod_{j=0}^{l}\prod_{k=0}^{r-1} \zeta_{2r}^{-k-rj}= \prod_{j=0}^{l}\zeta_{2r}^{-rj} \prod_{k=0}^{r-1} \zeta_{2r}^{-k}\\ = \zeta_{2r}^{\frac{-rl(r-1)}{2}} \zeta_{2r}^{\frac{-rl(l-1)}{2}}. \end{multline*}
\end{enumerate}
\end{proof}

\paragraph{}
We proceed similarly as in the paragraph preceding Definition \ref{unified_form} and define an element to each state diagram of $D$ that will be used to factorise $F_{\infty}(q,A,D)$. Let $\mathcal{K}$ be a knot seen as a $(1,1)$ tangle and $D$ a diagram of it. For a state diagram of $D$ we define:
\begin{align*}
D_{C,r}(l_1, \dots, l_N) = &(\prod_{j=1}^{S} \zeta_{2r}^{\mp r\alpha})\prod_{k \in pos} \binom{u_k+l_k}{l_k} \{r\alpha\}_{\zeta_{2r}}^{l_k} \zeta_{2r}^{-(u_k+v_k) r\alpha} \\ & \times  \prod_{k \in neg} (-1)^{l_k} \binom{u_k+l_k}{l_k} \{r\alpha\}_{\zeta_{2r}}^{l_k} \zeta_{2r}^{(u_k+v_k) r\alpha}
\end{align*}
where $neg \ \cup \ pos = [|1, N|]$ and $k \in pos$ if the $k$-th crossing of D is positive, else $k \in neg$, $a_k, b_k \in [|0, \dots, r-1 |]$, $a_k+ ru_k, b_k+ r v_k$ are the strands labels at the $k$-th crossing of the state diagram (see Figure \ref{crossings_factor}), $S$ is the number of \def\svgwidth{5mm} $+$ \def\svgwidth{5mm} appearing in the diagram, and $\epsilon_j$ the strand label at the j-th \def\svgwidth{5mm} or \def\svgwidth{5mm}, the $\mp$ sign is negative if \def\svgwidth{5mm} and positive if \def\svgwidth{5mm}.

\begin{prop} ~\\ \label{ADO_factor_prop}
For a knot $\mathcal{K}$ and a diagram of the knot $D$, $r \in \N^*$, we have the following factorization in $\hat{R_r^I}$:
\[ F_{\infty} (\zeta_{2r}, A, D) = C_{\infty}(r, A, D) \times ADO_r(A,\mathcal{K}) \]
where : \begin{align*}
C_{\infty}(r, A, D)& =\underset{\overline{l}=0}{\overset{+\infty}{\sum}} D_{C,r}(l_1, \dots, l_N)\\
&=\underset{\overline{l}=0}{\overset{+\infty}{\sum}} (\prod_{j=1}^{S} \zeta_{2r}^{\mp r\alpha})\prod_{k \in pos} \binom{u_k+l_k}{l_k} \{r\alpha\}_{\zeta_{2r}}^{l_k} \zeta_{2r}^{-(u_k+v_k) r\alpha} \\ & \hspace{80pt}\times  \prod_{k \in neg} (-1)^{l_k} \binom{u_k+l_k}{l_k} \{r\alpha\}_{\zeta_{2r}}^{l_k} \zeta_{2r}^{(u_k+v_k) r\alpha}
\end{align*}
where $\overline{l}= (l_1, \dots, l_N)$, $N$ is the number of crossings, $S$ the number of \def\svgwidth{5mm} $+$ \def\svgwidth{5mm} and $f$ is the framing of the knot.
\end{prop}
\begin{proof}
For the sake of simplicity, we will only consider positive crossings in the following proof. We factorize as follows:
\begin{align*}
 F_{\infty}(\zeta_{2r},A, D)&=\zeta_{2r}^{\frac{f \alpha^2}{2}} \underset{\overline{s}=0}{\overset{+\infty}{\sum}} (\prod_{j=1}^{S} \zeta_{2r}^{\mp(\alpha -2 \epsilon_j)})\prod_{k=1}^N \zeta_{2r}^{\frac{s_k(s_k-1)}{2}} \qbinom{z_k+s_k}{s_k}_{\zeta_{2r}} \\ & \times \{\alpha-z_k; s_k\}_{\zeta_{2r}} \zeta_{2r}^{(-z_k-y_k) \alpha} \zeta_{2r}^{ 2(z_k+s_k)(y_k-s_k)}\\
 &=\zeta_{2r}^{\frac{f \alpha^2}{2}} \underset{\overline{i+rl}=0}{\overset{+\infty}{\sum}} (\prod_{j=1}^{S} \zeta_{2r}^{\mp(\alpha -2 \epsilon_j)})\prod_{k=1}^N \zeta_{2r}^{\frac{(i_k+rl_k)(i_k+rl_k-1)}{2}} \\ & \times \qbinom{a_k+i_k+r(u_k+l_k)}{i_k+rl_k}_{\zeta_{2r}}  \{\alpha-(a_k+ru_k); i_k+rl_k\}_{\zeta_{2r}} \\ & \times \zeta_{2r}^{(-(a_k+ru_k)-(b_k+rv_k)) \alpha} \zeta_{2r}^{ 2((a_k+ru_k)+(i_k+rl_k))(b_k+rv_k-(i_k+rl_k))}\\
 &=\zeta_{2r}^{\frac{f \alpha^2}{2}} \underset{\overline{i}=0}{\overset{r-1}{\sum}} (\prod_{j=1}^{S} \zeta_{2r}^{\pm (r-1)(\alpha -2 \epsilon_j)})\prod_{k=1}^N \zeta_{2r}^{\frac{i_k(i_k-1)}{2}} \qbinom{a_k+i_k}{i_k}_{\zeta_{2r}}\\ & \times \{\alpha-a_k; i_k\}_{\zeta_{2r}}  \zeta_{2r}^{(-a_k-b_k) \alpha} \zeta_{2r}^{ 2(a_k+i_k)(b_k-i_k)} \\ & \times \underset{\overline{l}=0}{\overset{+\infty}{\sum}} (\prod_{j=1}^{S} \zeta_{2r}^{\mp r\alpha}) \prod_{k=1}^N \binom{u_k+l_k}{l_k} \{r\alpha\}_{\zeta_{2r}}^{l_k} \zeta_{2r}^{(-u_k-v_k) r\alpha}
\end{align*}

The second equality is obtained by changing variables $s_k=i_k+rl_k$ $0 \leq i_k \leq r-1$ and writing the strands labels at crossings $z_k$ as $z_k=a_k + ru_k$ $0 \leq a_k \leq r-1$ and $y_k$ as $y_k=b_k + rv_k$ $0 \leq b_k \leq r-1$. \\ Note that $a_k$, $b_k$ solely depends on $i_k$ and $u_k$, $b_k$ on $l_k$. This relies on the fact that $\qbinom{n+m}{n}_q =0$ at $q=\zeta_{2r}$ if $n,m \leq r-1$ and $n+m \geq r$.

\medskip
\noindent The third one is obtained by replacing each term with its factorization given by Lemma \ref{factorLemma}, the crossed terms between $i_k$ and $l_k$ are just signs, that eventually compensate. Hence, we have the factorization. 
\end{proof}

\begin{figure}[h!]
\begin{subfigure}[b]{0.5\textwidth}
 \centering
  \def\svgwidth{25mm}
\begingroup%
  \makeatletter%
  \providecommand\color[2][]{%
    \errmessage{(Inkscape) Color is used for the text in Inkscape, but the package 'color.sty' is not loaded}%
    \renewcommand\color[2][]{}%
  }%
  \providecommand\transparent[1]{%
    \errmessage{(Inkscape) Transparency is used (non-zero) for the text in Inkscape, but the package 'transparent.sty' is not loaded}%
    \renewcommand\transparent[1]{}%
  }%
  \providecommand\rotatebox[2]{#2}%
  \newcommand*\fsize{\dimexpr\f@size pt\relax}%
  \newcommand*\lineheight[1]{\fontsize{\fsize}{#1\fsize}\selectfont}%
  \ifx\svgwidth\undefined%
    \setlength{\unitlength}{666.14173228bp}%
    \ifx\svgscale\undefined%
      \relax%
    \else%
      \setlength{\unitlength}{\unitlength * \real{\svgscale}}%
    \fi%
  \else%
    \setlength{\unitlength}{\svgwidth}%
  \fi%
  \global\let\svgwidth\undefined%
  \global\let\svgscale\undefined%
  \makeatother%
  \begin{picture}(1,1.34042553)%
    \lineheight{1}%
    \setlength\tabcolsep{0pt}%
    \put(0,0){\includegraphics[width=\unitlength,page=1]{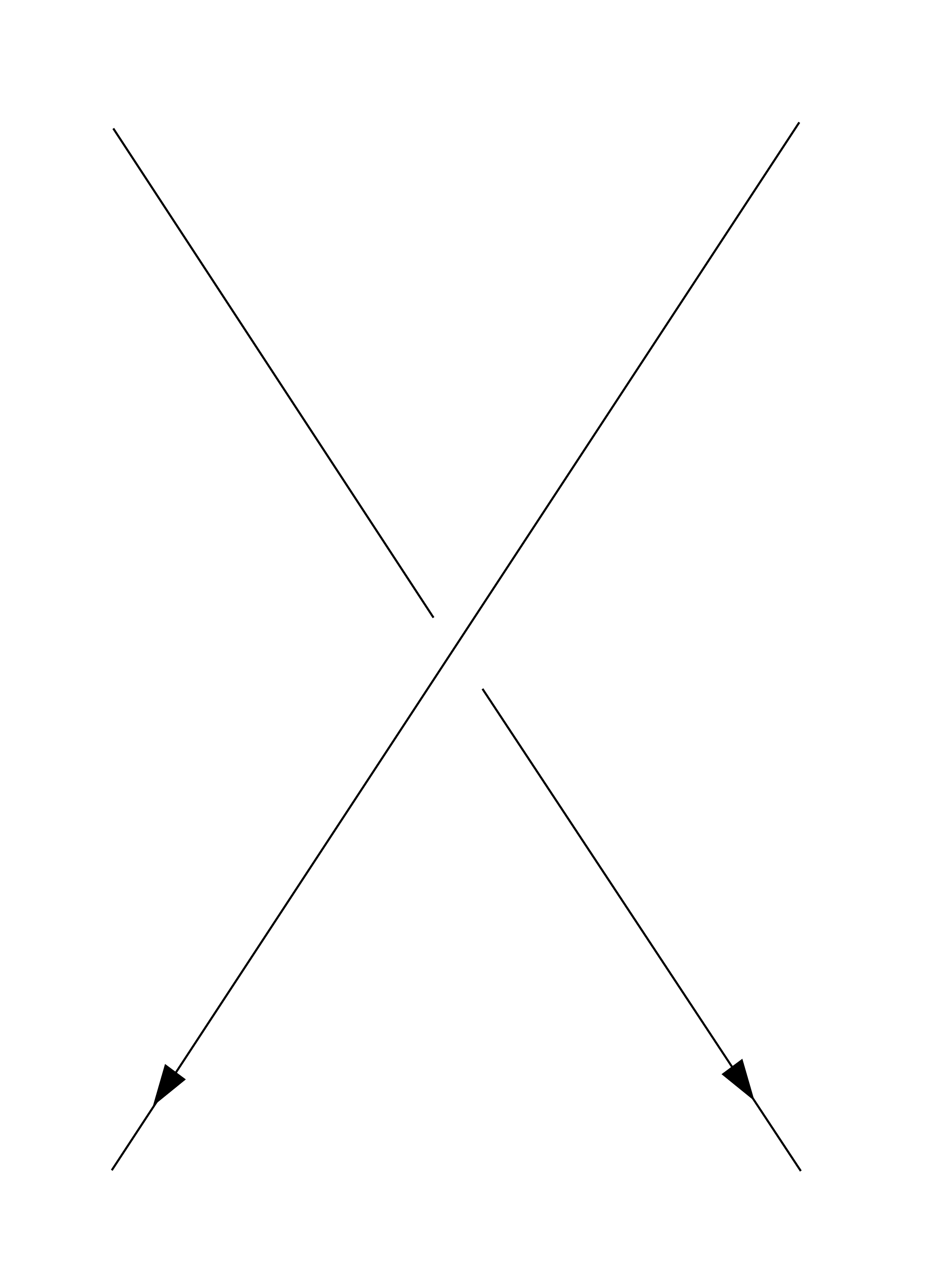}}%
    \put(0.84572423,0.02332179){\color[rgb]{0,0,0}\makebox(0,0)[lt]{\lineheight{1.25}\smash{\begin{tabular}[t]{l}$a_k + r u_k$\end{tabular}}}}%
    \put(-0.29029025,0.02038821){\color[rgb]{0,0,0}\makebox(0,0)[lt]{\lineheight{1.25}\smash{\begin{tabular}[t]{l}$b_k +r v_k$\end{tabular}}}}%
    \put(0.73748325,1.25886345){\color[rgb]{0,0,0}\makebox(0,0)[lt]{\lineheight{1.25}\smash{\begin{tabular}[t]{l}$b_k +r v_k -i_k - r l_k$\end{tabular}}}}%
    \put(-0.6103637,1.26690551){\color[rgb]{0,0,0}\makebox(0,0)[lt]{\lineheight{1.25}\smash{\begin{tabular}[t]{l}$a_k+r u_k +i_k +r l_k$\end{tabular}}}}%
  \end{picture}%
\endgroup%

  \label{cross_factor}
   \caption{Positive crossing.}
 \end{subfigure}%
 \begin{subfigure}[b]{0.5\textwidth}
 \centering
  \def\svgwidth{25mm}
\begingroup%
  \makeatletter%
  \providecommand\color[2][]{%
    \errmessage{(Inkscape) Color is used for the text in Inkscape, but the package 'color.sty' is not loaded}%
    \renewcommand\color[2][]{}%
  }%
  \providecommand\transparent[1]{%
    \errmessage{(Inkscape) Transparency is used (non-zero) for the text in Inkscape, but the package 'transparent.sty' is not loaded}%
    \renewcommand\transparent[1]{}%
  }%
  \providecommand\rotatebox[2]{#2}%
  \newcommand*\fsize{\dimexpr\f@size pt\relax}%
  \newcommand*\lineheight[1]{\fontsize{\fsize}{#1\fsize}\selectfont}%
  \ifx\svgwidth\undefined%
    \setlength{\unitlength}{666.14173228bp}%
    \ifx\svgscale\undefined%
      \relax%
    \else%
      \setlength{\unitlength}{\unitlength * \real{\svgscale}}%
    \fi%
  \else%
    \setlength{\unitlength}{\svgwidth}%
  \fi%
  \global\let\svgwidth\undefined%
  \global\let\svgscale\undefined%
  \makeatother%
  \begin{picture}(1,1.34042553)%
    \lineheight{1}%
    \setlength\tabcolsep{0pt}%
    \put(0,0){\includegraphics[width=\unitlength,page=1]{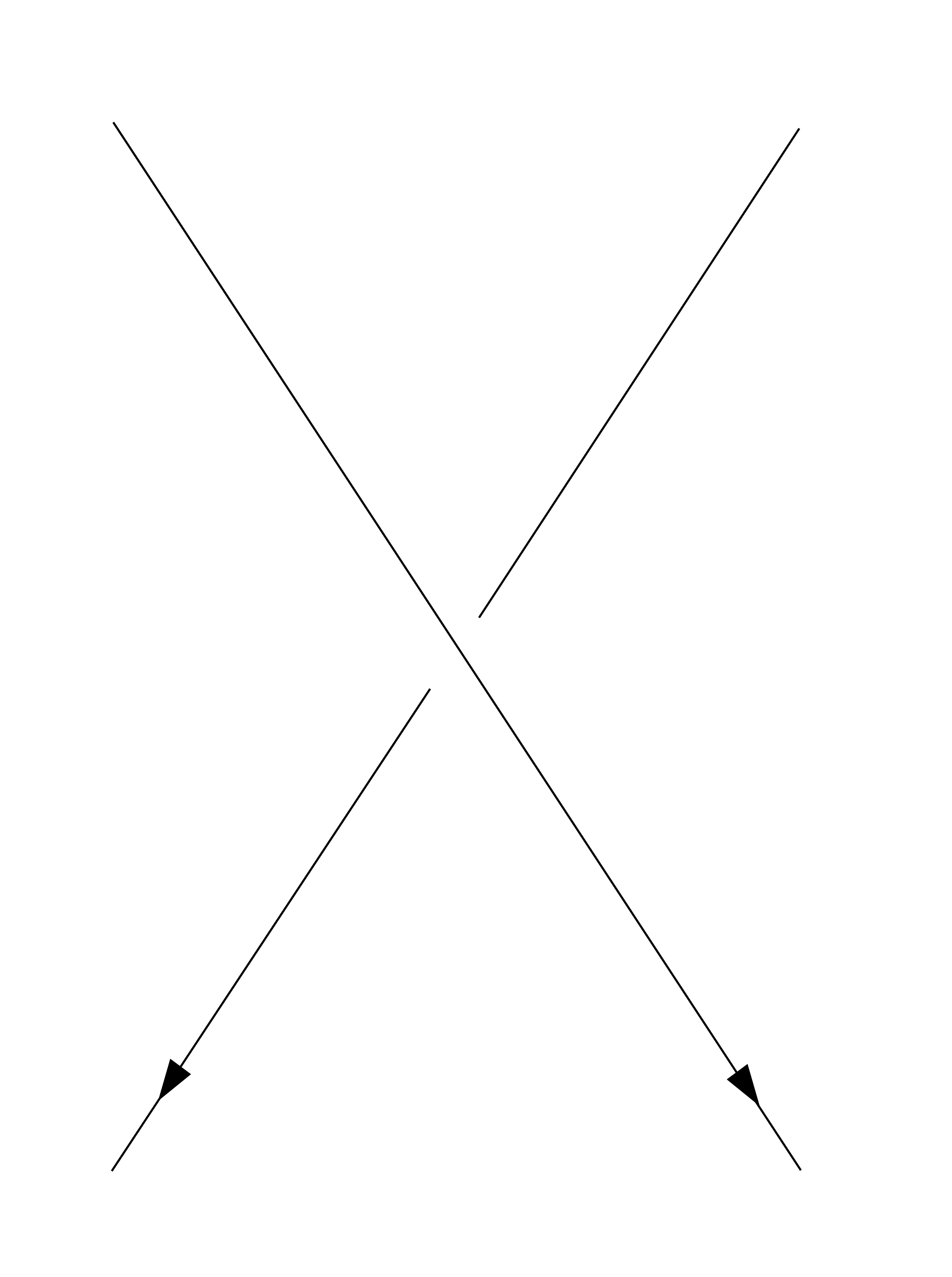}}%
    \put(-0.28820434,0.02975555){\color[rgb]{0,0,0}\makebox(0,0)[lt]{\lineheight{1.25}\smash{\begin{tabular}[t]{l}$a_k + r u_k$\end{tabular}}}}%
    \put(0.83881304,0.01234624){\color[rgb]{0,0,0}\makebox(0,0)[lt]{\lineheight{1.25}\smash{\begin{tabular}[t]{l}$b_k +r v_k$\end{tabular}}}}%
    \put(-0.59749652,1.26047186){\color[rgb]{0,0,0}\makebox(0,0)[lt]{\lineheight{1.25}\smash{\begin{tabular}[t]{l}$b_k +r v_k -i_k - r l_k$\end{tabular}}}}%
    \put(0.77286831,1.26047187){\color[rgb]{0,0,0}\makebox(0,0)[lt]{\lineheight{1.25}\smash{\begin{tabular}[t]{l}$a_k+r u_k +i_k +r l_k$\end{tabular}}}}%
  \end{picture}%
\endgroup%

  \label{negcross_factor}
   \caption{Negative crossing.}
 \end{subfigure}%
 \caption{The two possibilities for the k-th crossing in $D$ when factorizing.}
 \label{crossings_factor}
 \end{figure}
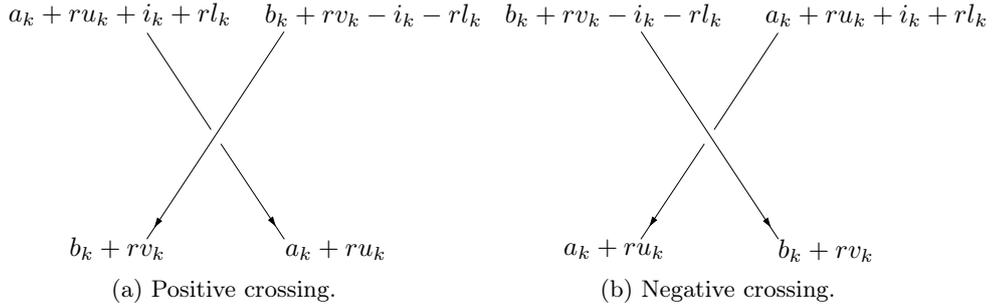
 
\medskip
In order to get back $ADO_r(A,\mathcal{K})$ from $F_{\infty}(\zeta_{2r},A,D)$, we need to prove that $C_{\infty}(r, A,D)$ is a unit in $\hat{R_r^I}$.

\begin{prop}
If $a=(a_n)_{n \in \N^*} \in \hat{R_r^I}$ and $a_1 \in R_r / I$ is a unit, then $a$ is a unit in $\hat{R_r^I}$.
\end{prop}
\begin{proof}
Let $a=(a_n)_{n \in \N^*} \in \hat{R_r^I}$ such that $a_1$ is a unit of $R_r / I$.\\
Let's prove that $a_n$ is also a unit in $R_r / I^n$. Indeed, if $y$ is an element of $R_r/I^n$ such that $a_n y = a_1 y= 1$ mod $I$ then $\exists z \in I . R_r / I^n$ such that $a_n y= 1 + z$, $z=a_n y -1$ thus $0=z^n= (a_ny-1)^n$, which proves that $a_n$ is invertible.\\
Hence, $a^{-1}=(a_n^{-1})_{n \in \N^*}$ is the inverse of $a$ in $\hat{R_r^I}$. 
\end{proof}

\medskip
Since $C_{\infty}(r, A,D)=(\prod_{j=1}^{S} \zeta_{2r}^{\mp r\alpha})$ mod $\{ r \alpha \}_{\zeta_{2r}} $ is an invertible element of $R_r/ I$, then $C_{\infty}(r, A, D)$ is a unit of $\hat{R_r^I}$.

\begin{cor}
\[ADO_r(A,\mathcal{K})  = F_{\infty} (\zeta_{2r}, A, D) C_{\infty}(r, A, D)^{-1}  \]
\end{cor}

\bigskip
Finally, one can recover $C_{\infty}(r, A, D)$ with $F_{\infty}(q,A,D)$, this will prove that not only that ADO is contained in $F_{\infty}(q,A,D)$ but that it's possible to extract them with the sole datum of $F_{\infty}(q,A,D)$.

For $r=1$, one gets
\begin{align*}
ev_1(F_{\infty} (q,A,D))&=F_{\infty} (\zeta_{2},A,D)\\ &= C_{\infty}(1,A,D) \times ADO_1 (A,\mathcal{K})\\ &= q^{\frac{f \alpha^2}{2}}C_{\infty}(1,A,D) 
\end{align*}
\begin{rem}
Note that $ADO_1(A,\mathcal{K})$ is only defined as the case $r=1$ in Prop \ref{ADO_form}, which is well defined. Nevertheless, the algebraic setup at Section \ref{ADOsection} fails at $r=1$ since $[E,F]$ is not well defined.
\end{rem}

But then $ C_{\infty}(1,A,D) \in \widehat{\Z[A^{\pm 1}]^{\{ \alpha\} }}:= \underset{\underset{n}{\leftarrow}}{\lim} \dfrac{\Z[A^{\pm 1}]}{\{ \alpha\}^n} $, for each $r$ we have a well defined map: \[g_r:\widehat{\Z[A^{\pm 1}]^{\{ \alpha\} }} \to \widehat{\Z[A^{\pm 1}]^{\{ r \alpha\} }}, \ \ q^{\alpha} \mapsto q^{r\alpha} \] such that $g_r(C_{\infty}(1,D))= C_{\infty}(r, A, D)$.

This proves the following proposition:
\begin{prop}\label{propFC}
For all $r$, we have a well defined map $FC_r=g_r \circ ev_{1}:\hat{R}^{\hat{I}} \to \widehat{\Z[A^{\pm 1}]^{\{ r \alpha\} }}$ and for any knot $\mathcal{K}$ and any diagram $D$ of the knot, $F_{\infty}(q,A,D) \mapsto C_{\infty}(r, A, D)$.
\end{prop}

\begin{cor}
For all $r$, we have a well defined map $ev_r \times \frac{1}{FC_r}:(\hat{R}^{\hat{I}})^{\times} \to (\hat{R_r^I})^{\times} $ and for any knot $\mathcal{K}$ and any diagram $D$ of the knot, $F_{\infty}(q,A,D) \mapsto ADO_r(A,\mathcal{K})$.
\end{cor}
\begin{proof}
Let $x \in (\hat{R}^{\hat{I}})^{\times}$ an invertible element, since $FC_r$ is a ring morphism, $FC_r (x)$ is invertible. Then $Id \times \frac{1}{FC_r} (F_{\infty}(q,A,D)) = F_{\infty}(\zeta_{2r},A,D) \times C_{\infty}(r, A, D)^{-1}= ADO_r(A,\mathcal{K})$. 
\end{proof}

\section{Universal invariant and Verma module} \label{universalVermasection}
We have built by hand an element $F_{\infty}(q,A,D)$ in some completion of a ring, from which we have evaluation maps that recovers the ADO invariants. This element is built from the diagram of a knot, thus it depends \textit{a priori} on it. In order to prove that this element is indeed a knot invariant, we will see how to obtain it using Hopf algebra machinery.

The first subsection will be dedicated to create an integral subalgebra of the $h$-adic version of quantum $\mathfrak{sl}_2$ containing the universal invariant of a $0$ framed knot.

This will allow us to define, in the second subsection, a Verma module on it. Since the algebra previously defined is integral, this will also be the case for the Verma module, whose coefficients will lie in $\hat{R}^{\hat{I}}$. The unified form $F_{\infty}(q,A,D)$ will be seen as the scalar action of the universal invariant on this Verma module. Since the universal invariant is a knot invariant, so will be $F_{\infty}(q,A,D)$.

This algebraic setup is made to get back the unified form and prove its invariance, and it is a completion which is very close to that of Habiro's in \cite{habiro2007integral}. But they are not the same, and we will see in the third subsection how to connect this work to Habiro's setup in the article. We will interpret our ring completion $\hat{R}^{\hat{I}}$ as some subalgebra completion found in \cite{habiro2007integral}, allowing to prove some nice properties on the ring structure (integral domain, subring of some h adic ring). Moreover we will show that the unified invariant can also be recovered from Habiro's algebraic setup, using the same process as in the second subsection, but with his completions.

Afterwards, we will see that we can also recover the colored Jones polynomials from the unified invariant. First this will allow us to study the factorisation in Proposition \ref{ADO_factor_prop}, and find that $C_{\infty}(r, A, D)$ is just the inverse of the Alexander polynomial. Lastly, using the unified invariant as a bridge between the family of colored Jones polynomials and the family of ADO polynomials, we will show that they are equivalent, meaning that we can recover one family with the other.

As a direct application of this facts, we will show that the unified invariant and every ADO polynomials follow the same holonomic rule as the colored Jones function (see \cite{garoufalidis2005colored}).

\subsection{The universal invariant}
In order to build $F_{\infty}(q,A,D)$ from Hopf algebra, we will need some "big enough" integral version quantum $\mathfrak{sl}_2$, but not too big in order to have a $\hat{R}^{\hat{I}}$ Verma module on it.
\medskip

First let's define the biggest integral quantum $\mathfrak{sl}_2$, $U_h$.

\begin{defn}
We set $U_h := U_h(\mathfrak{sl}_2)$ the $\Q[[h]]$ algebra topologically generated by $H,E,F$ and relations \[ [H,E]=2E, \ [H,F]=-2F, \ [E,F]=\frac{K-K^{-1}}{q-q^{-1}} \]
where $q=e^h$ and $K=q^K= e^{hH}$.
\end{defn}

It is endowed with an Hopf algebra structure:
$$
\begin{array}{lll}
\Delta(E)=1 \otimes E + E \otimes K & \epsilon(E)=0 &S(E)=-EK^{-1} \\
\Delta(F)=K^{-1} \otimes F + F \otimes 1 & \epsilon(F)=0 &S(E)=-KF \\
\Delta(H)=1 \otimes H + H \otimes 1 & \epsilon(H)=0 &S(H)=-H \\
\end{array}
$$

And an $R$-matrix:

\[R= q^{\frac{H \otimes H}{2}}\underset{i=0}{\overset{\infty}{\sum}} \frac{\{1\}^n q^{\frac{n(n-1)}{2}}}{[n]!} E^n \otimes F^{n}\] \[ R^{-1} = \underset{i=0}{\overset{\infty}{\sum}} \frac{(-1)^n \{1\}^n q^{ \frac{-n(n-1)}{2}}}{[n]!} E^n \otimes F^{n} q^{-\frac{H \otimes H}{2}}\]

Altogether with a ribbon element: $K^{-1} u$ where $u=\sum S(\beta) \alpha$ if $R= \sum \alpha \otimes \beta$.

\bigskip
\noindent Hence, if $\mathcal{K}$ is a knot and $T$ a 1-1 tangle whose closure is $\mathcal{K}$.\\
We set $Q^{U_h}(\mathcal{K}) \in U_h$ the universal invariant associated to $T$ in $U_h$. The definition of this element is given in Ohtsuki's book \cite{ohtsuki2002quantum} subsection 4.2.\\
It is a knot invariant.
\bigskip

Let us now build a suitable subalgebra of $U_h$ and a $\hat{R}^{\hat{I}}$ Verma module on it. We will then see that the universal invariant is in some extent inside the subalgebra and its scalar action on the Verma module will give us $F_{\infty}(q,A,D)$.

The subalgebra considered is an integral version of $U_q(\mathfrak{sl}_2)$ defined by:

\begin{defn}
Let $\mathcal{U}:=U_q^{D}(\mathfrak{sl}_2)$ the $\Z[q^{\pm 1}]$ subalgebra of $U_h$ generated by $E,\  F^{(n)},\ K$ where $F^{(n)}= \frac{\{1\}^n F^n }{[n]!}$.
\end{defn}

It inherits the Hopf Algebra structure from $U_h$.

\begin{rem}
The $R$-matrix is not an element of $\mathcal{U}$ but we have:
\[ R= q^{\frac{H \otimes H}{2}}\underset{i=0}{\overset{\infty}{\sum}} \frac{ \{ 1 \}^n q^{\frac{n(n-1)}{2}}}{[n]!} E^n \otimes F^{n} = q^{\frac{H \otimes H}{2}}\underset{i=0}{\overset{\infty}{\sum}} q^{\frac{n(n-1)}{2}} E^n \otimes F^{(n)} \]
\end{rem}

Hence, aside from $q^{\frac{H \otimes H}{2}}$ (that we can control in the universal invariant as we will see further on), we need the convergence of $\underset{i=0}{\overset{\infty}{\sum}} q^{\frac{n(n-1)}{2}} E^n \otimes F^{(n)}$ in some tensor product of the algebra with itself. \\
Thus we need to complete the algebra $\mathcal{U}$.

\medskip
\noindent We denote $\{H+m \}_q=Kq^m-K^{-m}q^{-m}$, $\{H+m;n \}_q= \prod_{i=0}^{n-1} \{H+m-i\}_q$.

\begin{defn}
Let $L_n$ be the $\Z[q^{\pm 1} ] $ ideal generated by $\{ n\}!$.\\
Let $J_n$ be the $\mathcal{U}$ two sided ideal generated by the following elements:
\[ F^{(i)} \{H+m;n-i \}_q \] where $m \in \Z$ and $i \in \{0, \dots, n \}$.
\end{defn}

\begin{lemme}
$J_n$ is finitely generated by elements of the form $F^{(i)} \{n-i;j\} \{ H;n-i-j\}$, $j \in \{0, \dots, n-i \}$, $i \in \{ 0, \dots, n \}$.
\end{lemme}
\begin{proof}
The proof can be found in Habiro's article \cite{habiro2007integral}, Proposition 5.1. 
\end{proof}

\medskip
Following the completion described by Habiro in is article \cite{habiro2007integral} Section 4.
We have:
\begin{enumerate}
\item $L_n \subset J_n$ (see Prop 5.1 in Habiro's article \cite{habiro2007integral} ),
\item $\Delta( J_n) \subset \underset{i+j=n}{\sum} J_i \otimes J_j$,
\item $ \epsilon(J_n) \subset L_n$,
\item $S(J_n) \subset J_n$.
\end{enumerate}

Thus we can define the completion $\mathcal{\hat{U}}:= \underset{\underset{n}{\leftarrow}}{\lim} \dfrac{\mathcal{U}}{J_n}$ as a $\widehat{\Z[q^{\pm 1}]}:= \underset{\underset{n}{\leftarrow}}{\lim} \dfrac{\Z[q^{\pm 1}]}{L_n}$ algebra ($\widehat{\Z[q^{\pm 1}]}$ is Habiro's ring).\\
And it is endowed with a complete Hopf algebra structure: \[ \hat{\Delta}: \mathcal{\hat{U}} \to \mathcal{\hat{U}} \hat{\otimes} \mathcal{\hat{U}}, \ \ \hat{\epsilon}: \mathcal{\hat{U}} \to \widehat{\Z[q^{\pm 1}]},\ \  \hat{S}: \mathcal{\hat{U}} \to \mathcal{\hat{U}}\]

where : $\mathcal{\hat{U}} \hat{\otimes} \mathcal{\hat{U}} = \underset{\underset{k,l}{\leftarrow}}{\lim} \dfrac{\mathcal{\hat{U}} \otimes_{\widehat{\Z[q^{\pm 1}]}} \mathcal{\hat{U}}}{\mathcal{\hat{U}}  \otimes_{\widehat{\Z[q^{\pm 1}]}} \overline{J_k} + \overline{J_l} \otimes_{\widehat{\Z[q^{\pm 1}]}} \mathcal{\hat{U}} }$ and $ \overline{J_n}$ is the closure of $J_n$ in $\mathcal{\hat{U}}$.

\bigskip
This completion is a bigger algebra than $\mathcal{U}$:
\begin{prop}
The canonical projection maps induce an injective map $\mathcal{U} \xhookrightarrow{}  \mathcal{\hat{U}}$.
\end{prop}
\begin{proof}
Take the projective maps $j_n : \mathcal{U} \to \mathcal{U}/J_n$, they induce a map $j: \mathcal{U} \to \mathcal{\hat{U}}$.
This map in injective because if $j(x)=0$ then $x \in \underset{n \in \N^*}{\bigcap} J_n $. But since $J_n \subset h^n U_h$ then $\underset{n \in \N^*}{\bigcap} J_n  \subset \underset{n \in \N^*}{\bigcap} h^n U_h$. It is a well known fact that $\underset{n \in \N^*}{\bigcap} h^n U_h=\{0\}$. 
\end{proof}

\medskip

Moreover, since $J_n \subset h^n U_h$ we have a map $i: \mathcal{\hat{U}} \to U_h$.
Since we do not know if this map is injective, we consider $\tilde{\mathcal{U}}:= i(\mathcal{\hat{U}})$ the image in $U_h$. It is also an Hopf algebra.

\begin{rem}
$\underset{i=0}{\overset{\infty}{\sum}} q^{\frac{n(n-1)}{2}} E^n \otimes F^{(n)} \in \mathcal{\tilde{U}} \hat{\otimes} \mathcal{\tilde{U}} $
\end{rem}

\bigskip
We will need a lemma to compute some commutation rules.
\begin{lemme}~\\
$(E \otimes 1)  \times q^{\frac{H \otimes H}{2}}  =q^{\frac{H \otimes H}{2}} \times (E \otimes 1)  \times (1 \otimes K) $\\
$ (F^{(n)} \otimes 1) \times q^{\frac{H \otimes H}{2}} =q^{\frac{H \otimes H}{2}} \times (F^{(n)} \otimes 1) \times (1 \otimes K^{-n})$
\label{lemmaCommutation}
\end{lemme}
\begin{proof}
Notice that since $EH^n=(H+2)^nE$ and $q^{\frac{H \otimes H}{2}}= \sum (\frac{h^n}{2^n n!}) H^n \otimes H^n$, then \[(E \otimes 1)  \times q^{\frac{H \otimes H}{2}}  =q^{\frac{(H+2) \otimes H}{2}} \times (E \otimes 1)=q^{\frac{H \otimes H}{2}} \times (1 \otimes K) \times (E \otimes 1) .\]\\
The same can be done for $F^{(n)}$. 
\end{proof}

\bigskip

Let us now construct the universal invariant $Q^{U_h}(\mathcal{K})$ by hand, seeing it as the $(1,1)$-tangle with coupons.

We can picture it as a 1-1 tangle with (2,2) coupons for $R$-matrix i.e. \def \svgwidth{7mm} 
\begingroup%
  \makeatletter%
  \providecommand\color[2][]{%
    \errmessage{(Inkscape) Color is used for the text in Inkscape, but the package 'color.sty' is not loaded}%
    \renewcommand\color[2][]{}%
  }%
  \providecommand\transparent[1]{%
    \errmessage{(Inkscape) Transparency is used (non-zero) for the text in Inkscape, but the package 'transparent.sty' is not loaded}%
    \renewcommand\transparent[1]{}%
  }%
  \providecommand\rotatebox[2]{#2}%
  \newcommand*\fsize{\dimexpr\f@size pt\relax}%
  \newcommand*\lineheight[1]{\fontsize{\fsize}{#1\fsize}\selectfont}%
  \ifx\svgwidth\undefined%
    \setlength{\unitlength}{666.14173228bp}%
    \ifx\svgscale\undefined%
      \relax%
    \else%
      \setlength{\unitlength}{\unitlength * \real{\svgscale}}%
    \fi%
  \else%
    \setlength{\unitlength}{\svgwidth}%
  \fi%
  \global\let\svgwidth\undefined%
  \global\let\svgscale\undefined%
  \makeatother%
  \begin{picture}(1,1.34042553)%
    \lineheight{1}%
    \setlength\tabcolsep{0pt}%
    \put(0,0){\includegraphics[width=\unitlength,page=1]{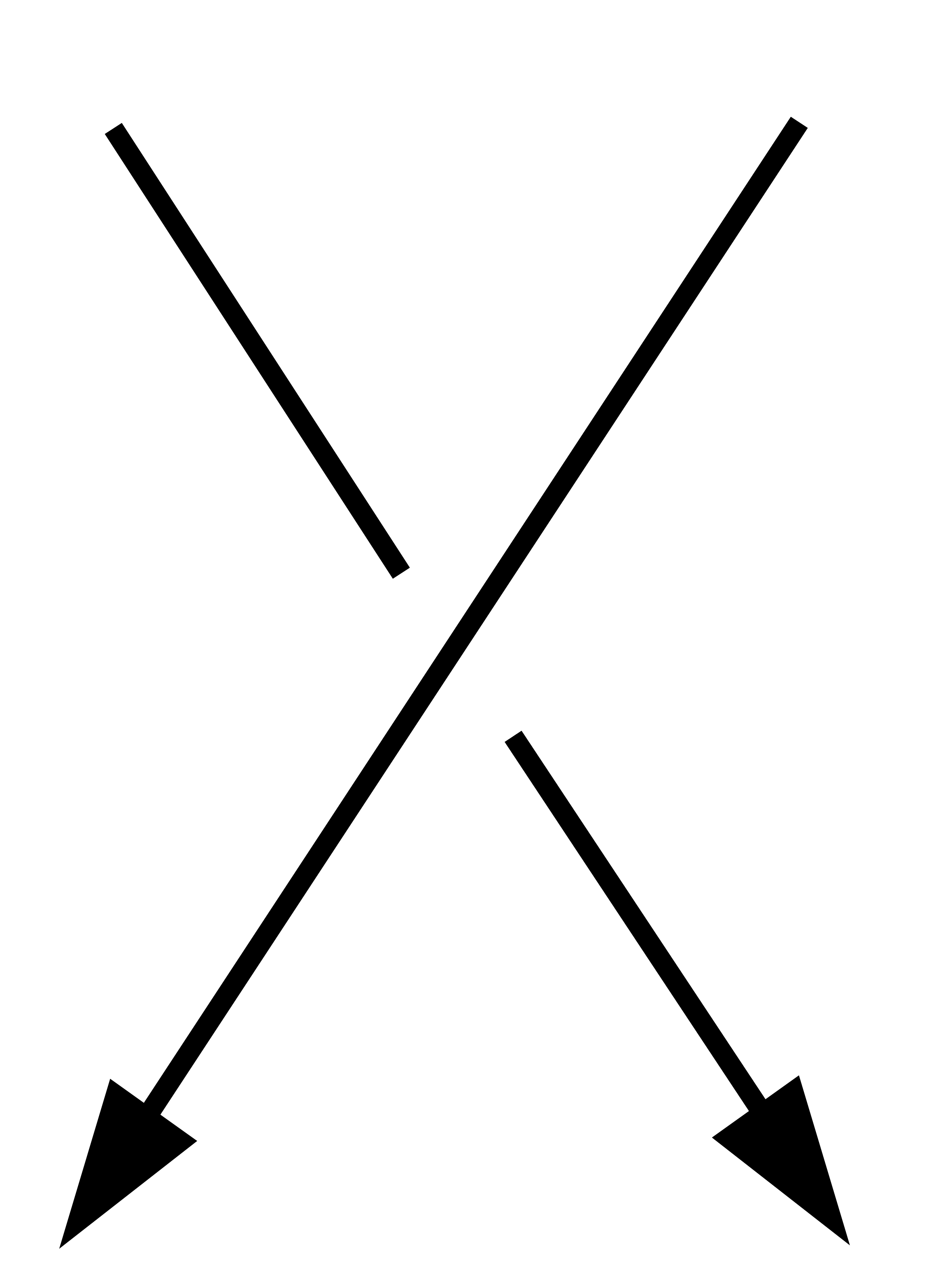}}%
  \end{picture}%
\endgroup%
 $=$ \def \svgwidth{5mm} 
\begingroup%
  \makeatletter%
  \providecommand\color[2][]{%
    \errmessage{(Inkscape) Color is used for the text in Inkscape, but the package 'color.sty' is not loaded}%
    \renewcommand\color[2][]{}%
  }%
  \providecommand\transparent[1]{%
    \errmessage{(Inkscape) Transparency is used (non-zero) for the text in Inkscape, but the package 'transparent.sty' is not loaded}%
    \renewcommand\transparent[1]{}%
  }%
  \providecommand\rotatebox[2]{#2}%
  \newcommand*\fsize{\dimexpr\f@size pt\relax}%
  \newcommand*\lineheight[1]{\fontsize{\fsize}{#1\fsize}\selectfont}%
  \ifx\svgwidth\undefined%
    \setlength{\unitlength}{510.41066087bp}%
    \ifx\svgscale\undefined%
      \relax%
    \else%
      \setlength{\unitlength}{\unitlength * \real{\svgscale}}%
    \fi%
  \else%
    \setlength{\unitlength}{\svgwidth}%
  \fi%
  \global\let\svgwidth\undefined%
  \global\let\svgscale\undefined%
  \makeatother%
  \begin{picture}(1,1.73024817)%
    \lineheight{1}%
    \setlength\tabcolsep{0pt}%
    \put(0,0){\includegraphics[width=\unitlength,page=1]{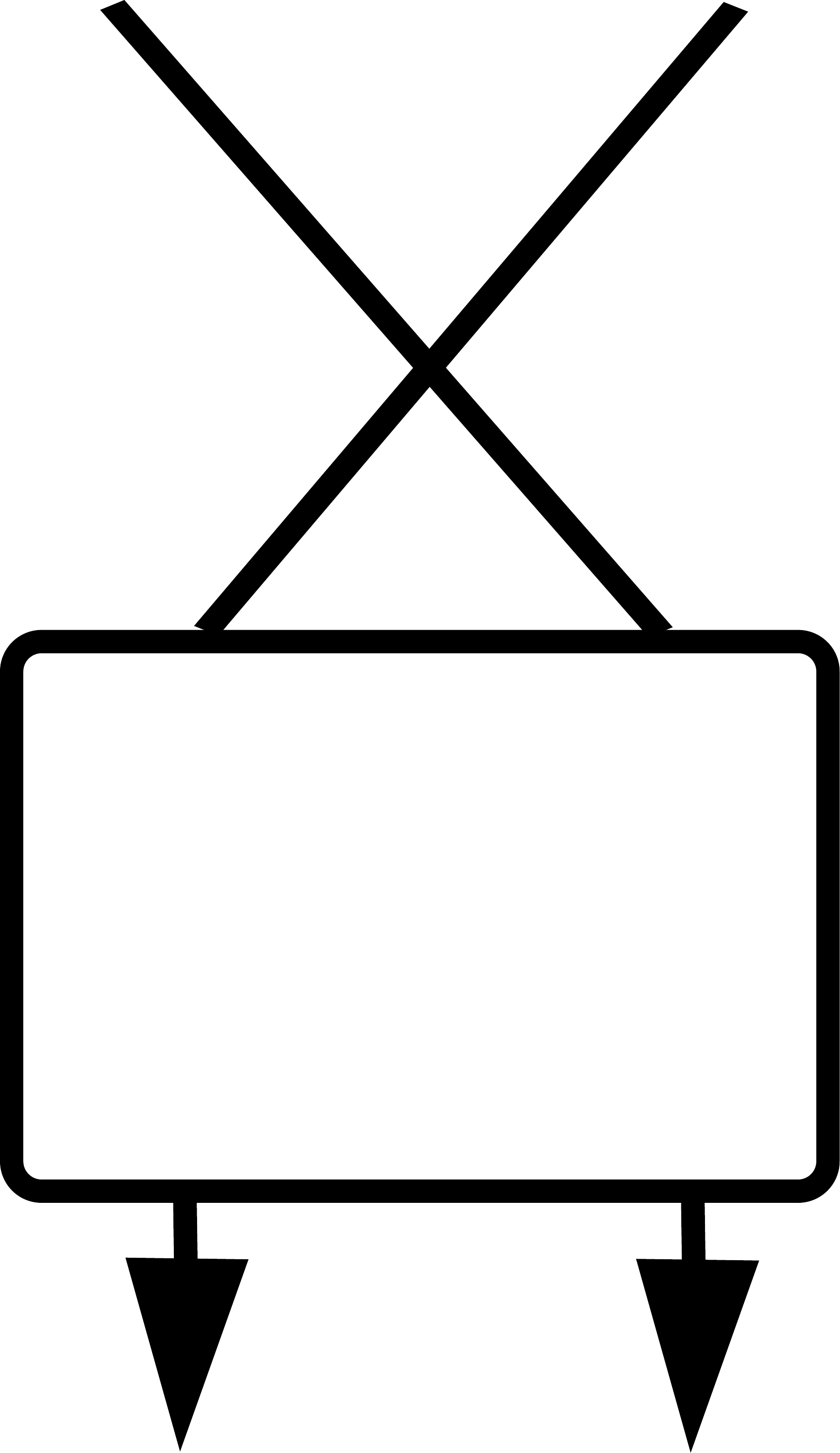}}%
    \put(0.20624071,0.4597886){\color[rgb]{0,0,0}\makebox(0,0)[lt]{\lineheight{1.25}\smash{\begin{tabular}[t]{l}$R$\end{tabular}}}}%
  \end{picture}%
\endgroup%
, \def \svgwidth{7mm} 
\begingroup%
  \makeatletter%
  \providecommand\color[2][]{%
    \errmessage{(Inkscape) Color is used for the text in Inkscape, but the package 'color.sty' is not loaded}%
    \renewcommand\color[2][]{}%
  }%
  \providecommand\transparent[1]{%
    \errmessage{(Inkscape) Transparency is used (non-zero) for the text in Inkscape, but the package 'transparent.sty' is not loaded}%
    \renewcommand\transparent[1]{}%
  }%
  \providecommand\rotatebox[2]{#2}%
  \newcommand*\fsize{\dimexpr\f@size pt\relax}%
  \newcommand*\lineheight[1]{\fontsize{\fsize}{#1\fsize}\selectfont}%
  \ifx\svgwidth\undefined%
    \setlength{\unitlength}{666.14173228bp}%
    \ifx\svgscale\undefined%
      \relax%
    \else%
      \setlength{\unitlength}{\unitlength * \real{\svgscale}}%
    \fi%
  \else%
    \setlength{\unitlength}{\svgwidth}%
  \fi%
  \global\let\svgwidth\undefined%
  \global\let\svgscale\undefined%
  \makeatother%
  \begin{picture}(1,1.34042553)%
    \lineheight{1}%
    \setlength\tabcolsep{0pt}%
    \put(0,0){\includegraphics[width=\unitlength,page=1]{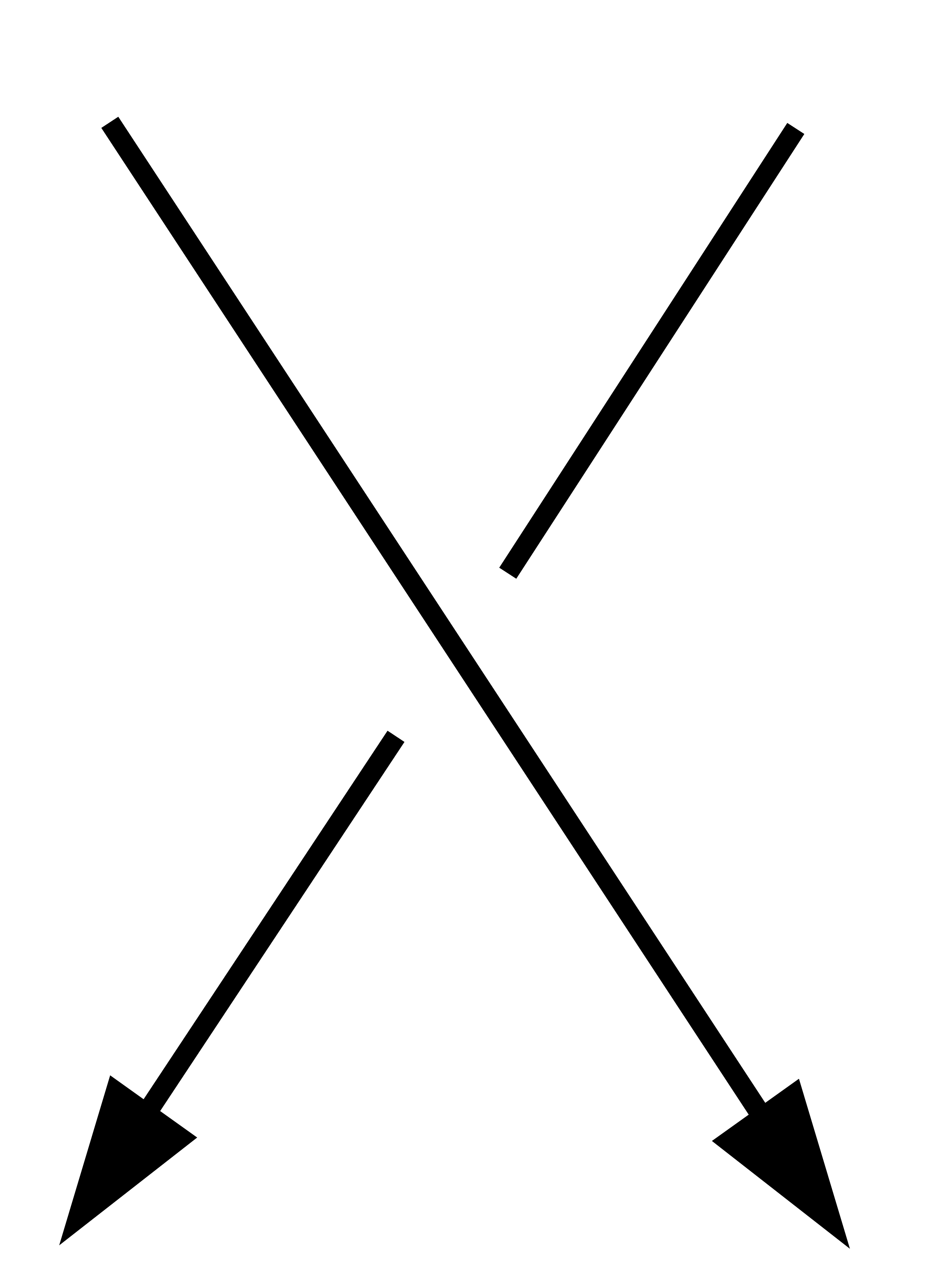}}%
  \end{picture}%
\endgroup%
 $=$ \def \svgwidth{7mm} 
\begingroup%
  \makeatletter%
  \providecommand\color[2][]{%
    \errmessage{(Inkscape) Color is used for the text in Inkscape, but the package 'color.sty' is not loaded}%
    \renewcommand\color[2][]{}%
  }%
  \providecommand\transparent[1]{%
    \errmessage{(Inkscape) Transparency is used (non-zero) for the text in Inkscape, but the package 'transparent.sty' is not loaded}%
    \renewcommand\transparent[1]{}%
  }%
  \providecommand\rotatebox[2]{#2}%
  \newcommand*\fsize{\dimexpr\f@size pt\relax}%
  \newcommand*\lineheight[1]{\fontsize{\fsize}{#1\fsize}\selectfont}%
  \ifx\svgwidth\undefined%
    \setlength{\unitlength}{595.27559055bp}%
    \ifx\svgscale\undefined%
      \relax%
    \else%
      \setlength{\unitlength}{\unitlength * \real{\svgscale}}%
    \fi%
  \else%
    \setlength{\unitlength}{\svgwidth}%
  \fi%
  \global\let\svgwidth\undefined%
  \global\let\svgscale\undefined%
  \makeatother%
  \begin{picture}(1,1.41428571)%
    \lineheight{1}%
    \setlength\tabcolsep{0pt}%
    \put(0,0){\includegraphics[width=\unitlength,page=1]{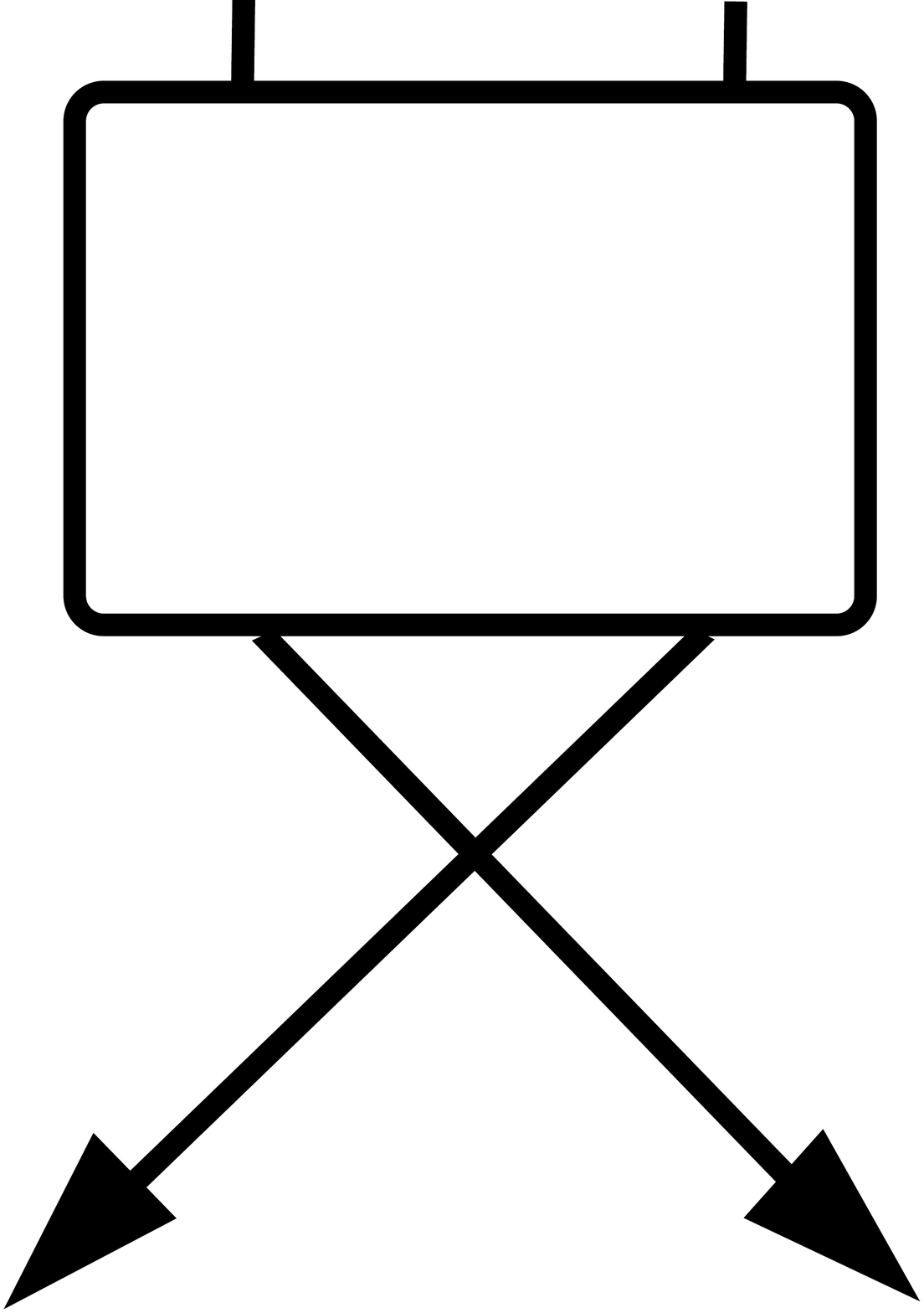}}%
    \put(0.22172944,0.85214321){\color[rgb]{0,0,0}\makebox(0,0)[lt]{\lineheight{1.25}\smash{\begin{tabular}[t]{l}$R^{\scalebox{0.5}{-1}}$\end{tabular}}}}%
  \end{picture}%
\endgroup%
. And also with (1,1) coupons for the pivotal element: \def \svgwidth{7mm}  $=$ \def \svgwidth{9mm} 
\begingroup%
  \makeatletter%
  \providecommand\color[2][]{%
    \errmessage{(Inkscape) Color is used for the text in Inkscape, but the package 'color.sty' is not loaded}%
    \renewcommand\color[2][]{}%
  }%
  \providecommand\transparent[1]{%
    \errmessage{(Inkscape) Transparency is used (non-zero) for the text in Inkscape, but the package 'transparent.sty' is not loaded}%
    \renewcommand\transparent[1]{}%
  }%
  \providecommand\rotatebox[2]{#2}%
  \newcommand*\fsize{\dimexpr\f@size pt\relax}%
  \newcommand*\lineheight[1]{\fontsize{\fsize}{#1\fsize}\selectfont}%
  \ifx\svgwidth\undefined%
    \setlength{\unitlength}{1094.17322835bp}%
    \ifx\svgscale\undefined%
      \relax%
    \else%
      \setlength{\unitlength}{\unitlength * \real{\svgscale}}%
    \fi%
  \else%
    \setlength{\unitlength}{\svgwidth}%
  \fi%
  \global\let\svgwidth\undefined%
  \global\let\svgscale\undefined%
  \makeatother%
  \begin{picture}(1,0.54404145)%
    \lineheight{1}%
    \setlength\tabcolsep{0pt}%
    \put(0,0){\includegraphics[width=\unitlength,page=1]{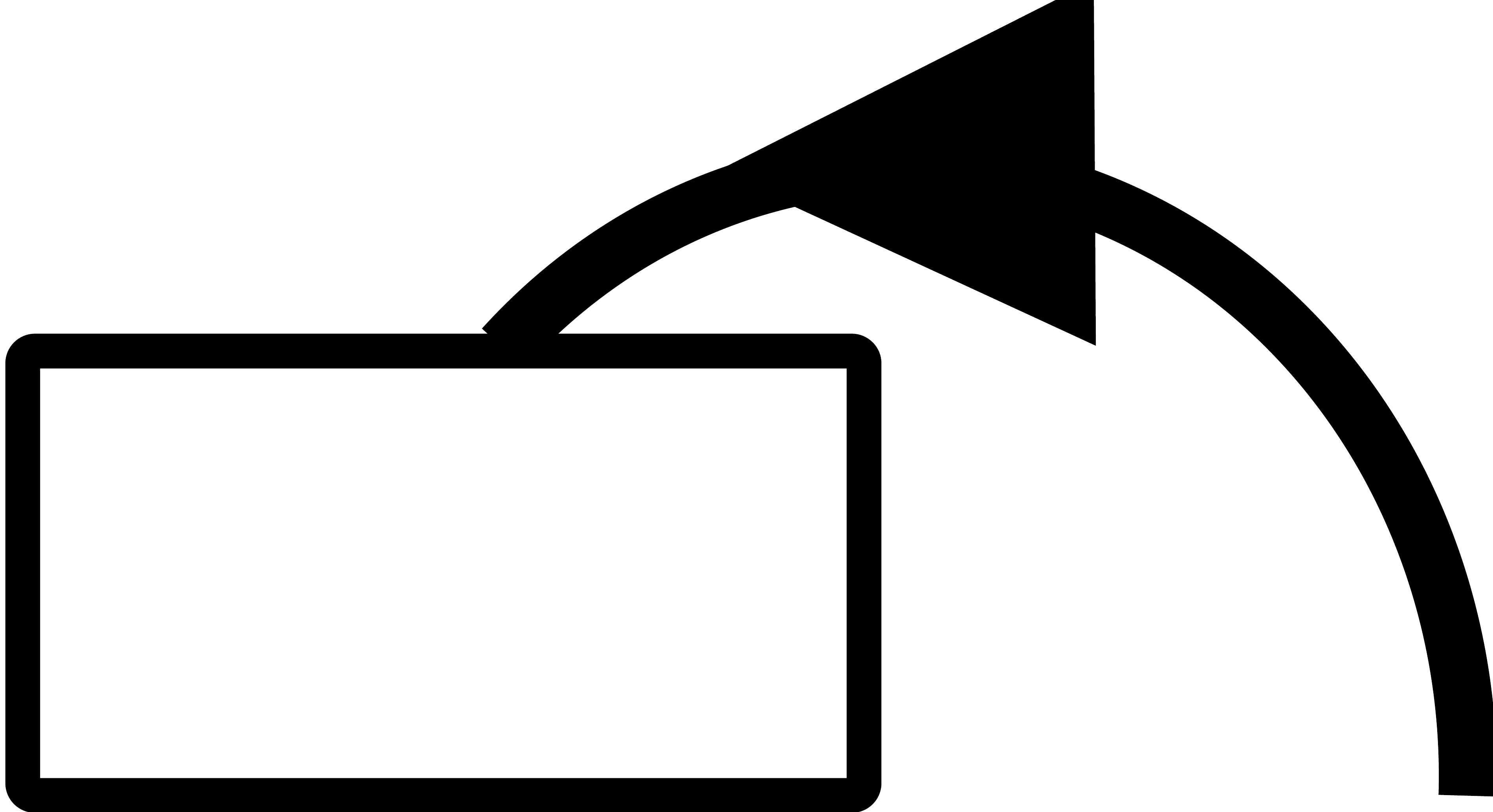}}%
    \put(0.04708376,0.04569897){\color[rgb]{0,0,0}\makebox(0,0)[lt]{\lineheight{1.25}\smash{\begin{tabular}[t]{l} \scalebox{0.75}{$K$}\end{tabular}}}}%
  \end{picture}%
\endgroup%
, \def \svgwidth{7mm}  $=$ \def \svgwidth{9mm} 
\begingroup%
  \makeatletter%
  \providecommand\color[2][]{%
    \errmessage{(Inkscape) Color is used for the text in Inkscape, but the package 'color.sty' is not loaded}%
    \renewcommand\color[2][]{}%
  }%
  \providecommand\transparent[1]{%
    \errmessage{(Inkscape) Transparency is used (non-zero) for the text in Inkscape, but the package 'transparent.sty' is not loaded}%
    \renewcommand\transparent[1]{}%
  }%
  \providecommand\rotatebox[2]{#2}%
  \newcommand*\fsize{\dimexpr\f@size pt\relax}%
  \newcommand*\lineheight[1]{\fontsize{\fsize}{#1\fsize}\selectfont}%
  \ifx\svgwidth\undefined%
    \setlength{\unitlength}{1094.17322835bp}%
    \ifx\svgscale\undefined%
      \relax%
    \else%
      \setlength{\unitlength}{\unitlength * \real{\svgscale}}%
    \fi%
  \else%
    \setlength{\unitlength}{\svgwidth}%
  \fi%
  \global\let\svgwidth\undefined%
  \global\let\svgscale\undefined%
  \makeatother%
  \begin{picture}(1,0.54404145)%
    \lineheight{1}%
    \setlength\tabcolsep{0pt}%
    \put(0,0){\includegraphics[width=\unitlength,page=1]{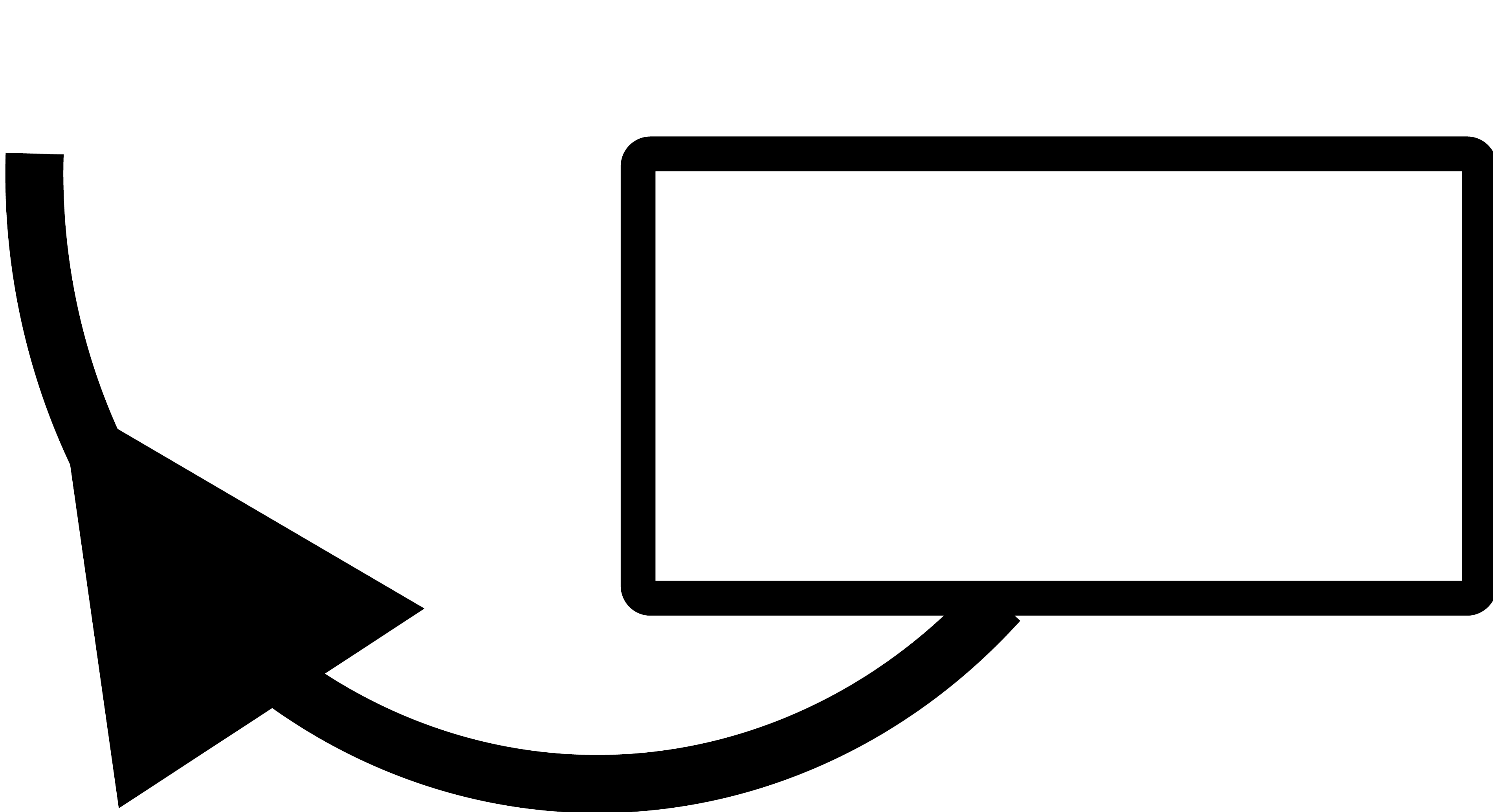}}%
    \put(0.44729559,0.17725649){\color[rgb]{0,0,0}\makebox(0,0)[lt]{\lineheight{1.25}\smash{\begin{tabular}[t]{l} \scalebox{0.65}{$K^{-1}$}\end{tabular}}}}%
  \end{picture}%
\endgroup%
. The rest of the tangle remains unchanged.

By replacing the  $R$ matrix with its formula, we get sums of diagrams with (1,1)-coupons  $E^n$, $F^{(n)}$, $K$ and (2,2)-coupons $ q^{\frac{H \otimes H}{2}}$, $q^{-\frac{H \otimes H}{2}}$ (that we can decompose into sum of (1,1)-coupons if seen as exponentials).\\ 
To compute the universal invariant, start from the top of the tangle and multiply (to the right) every coupons encountered.

\paragraph*{}
Now let's see that we can separate the universal invariant into two pieces. The example of the trefoil knot will illustrate the process all along.\\
The first step is to represent the knot with $R$ matrices (2,2)-coupons and $K^{\pm 1}$ (1,1)-coupons for the pivotal elements, as illustrated in Figure \ref{trefoil_coupons1} .\\
Now we write $R$ as a sum, so the (2,2)-coupons labeled by $R$ become the composition of (2,2)-coupons labeled by $ q^{\frac{H \otimes H}{2}}$ and (1,1)-coupons labeled by $E^n$ or $F^{(n)}$ (see Figure \ref{trefoil_coupons2} ).

\begin{figure}[h!]
\begin{subfigure}[b]{0.5\textwidth}
 \centering
  \def\svgwidth{25mm}
    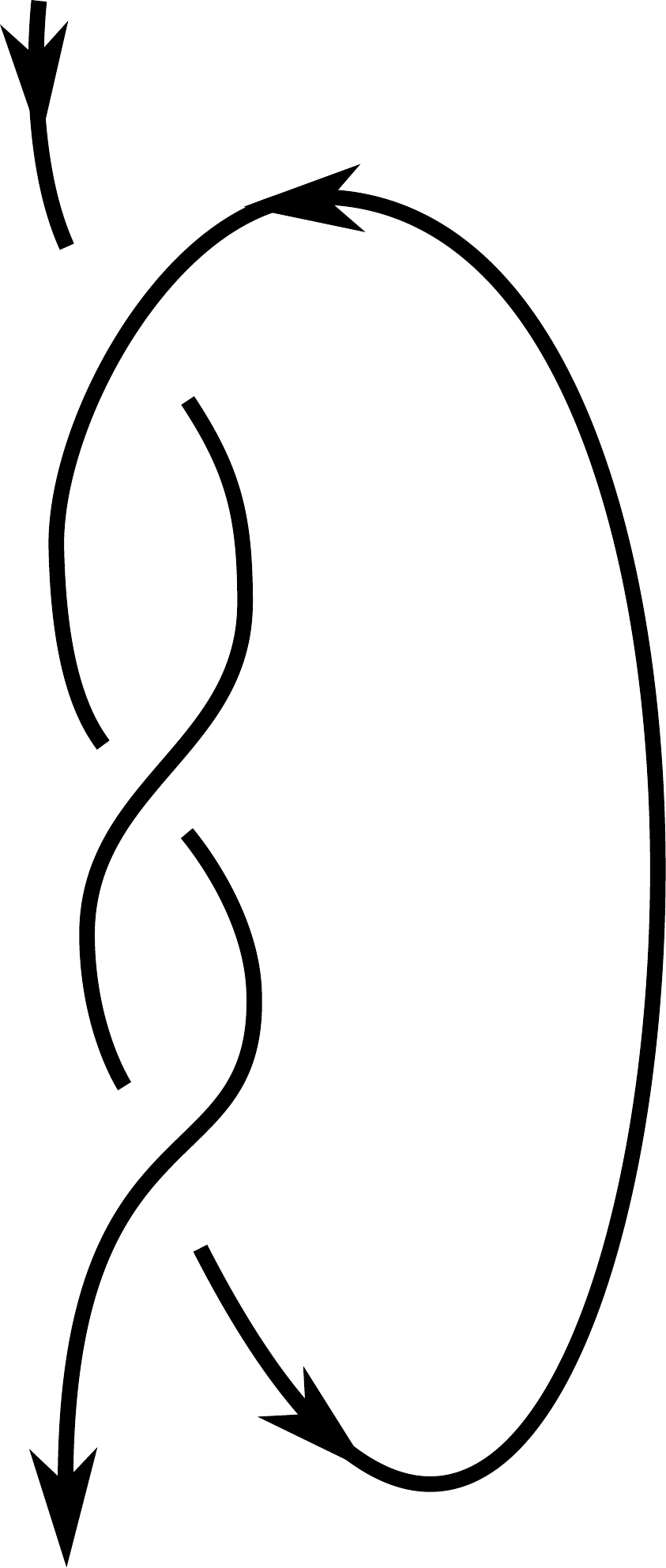%
   \caption{The trefoil knot.}
	\label{trefoil}
 \end{subfigure}%
 \begin{subfigure}[b]{0.5\textwidth}
 \centering
  \def\svgwidth{30mm}
\begingroup%
  \makeatletter%
  \providecommand\color[2][]{%
    \errmessage{(Inkscape) Color is used for the text in Inkscape, but the package 'color.sty' is not loaded}%
    \renewcommand\color[2][]{}%
  }%
  \providecommand\transparent[1]{%
    \errmessage{(Inkscape) Transparency is used (non-zero) for the text in Inkscape, but the package 'transparent.sty' is not loaded}%
    \renewcommand\transparent[1]{}%
  }%
  \providecommand\rotatebox[2]{#2}%
  \newcommand*\fsize{\dimexpr\f@size pt\relax}%
  \newcommand*\lineheight[1]{\fontsize{\fsize}{#1\fsize}\selectfont}%
  \ifx\svgwidth\undefined%
    \setlength{\unitlength}{296.43961508bp}%
    \ifx\svgscale\undefined%
      \relax%
    \else%
      \setlength{\unitlength}{\unitlength * \real{\svgscale}}%
    \fi%
  \else%
    \setlength{\unitlength}{\svgwidth}%
  \fi%
  \global\let\svgwidth\undefined%
  \global\let\svgscale\undefined%
  \makeatother%
  \begin{picture}(1,1.92285286)%
    \lineheight{1}%
    \setlength\tabcolsep{0pt}%
    \put(0,0){\includegraphics[width=\unitlength,page=1]{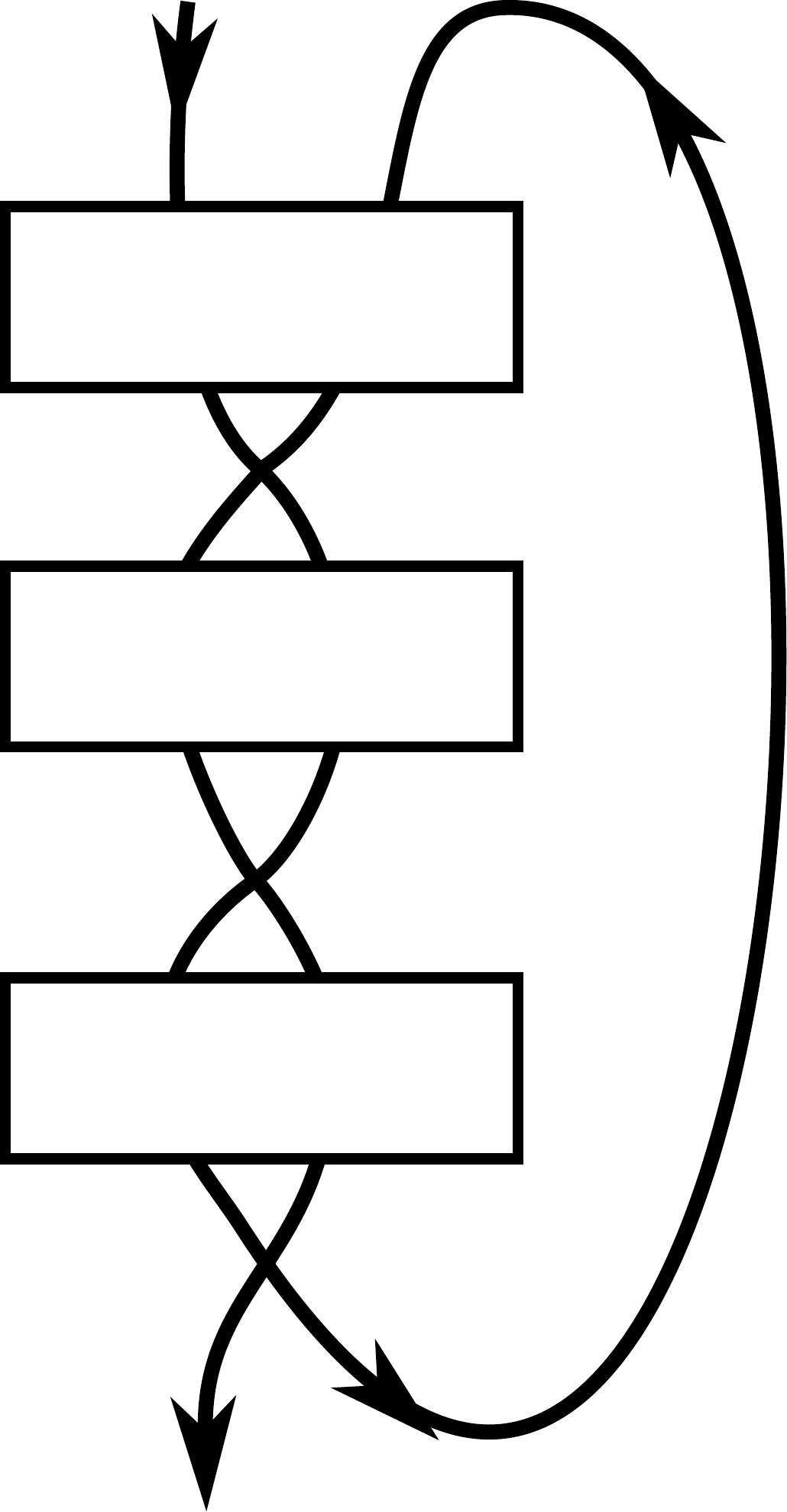}}%
    \put(0.29760448,1.52199067){\color[rgb]{0,0,0}\makebox(0,0)[lt]{\lineheight{1.25}\smash{\begin{tabular}[t]{l}$R$\end{tabular}}}}%
    \put(0.29760448,1.04444788){\color[rgb]{0,0,0}\makebox(0,0)[lt]{\lineheight{1.25}\smash{\begin{tabular}[t]{l}$R$\end{tabular}}}}%
    \put(0.29760448,0.52052683){\color[rgb]{0,0,0}\makebox(0,0)[lt]{\lineheight{1.25}\smash{\begin{tabular}[t]{l}$R$\end{tabular}}}}%
    \put(0,0){\includegraphics[width=\unitlength,page=2]{trefoil_coupons_1.pdf}}%
    \put(0.47083806,1.7494649){\color[rgb]{0,0,0}\makebox(0,0)[lt]{\lineheight{1.25}\smash{\begin{tabular}[t]{l}\scalebox{0.75}{$K $}\end{tabular}}}}%
  \end{picture}%
\endgroup%
   \caption{The universal invariant.}
   \label{trefoil_coupons1}
 \end{subfigure}%
 
 \begin{subfigure}[b]{1\textwidth}
 \centering
  \def\svgwidth{135mm}
    \resizebox{85mm}{!}{ 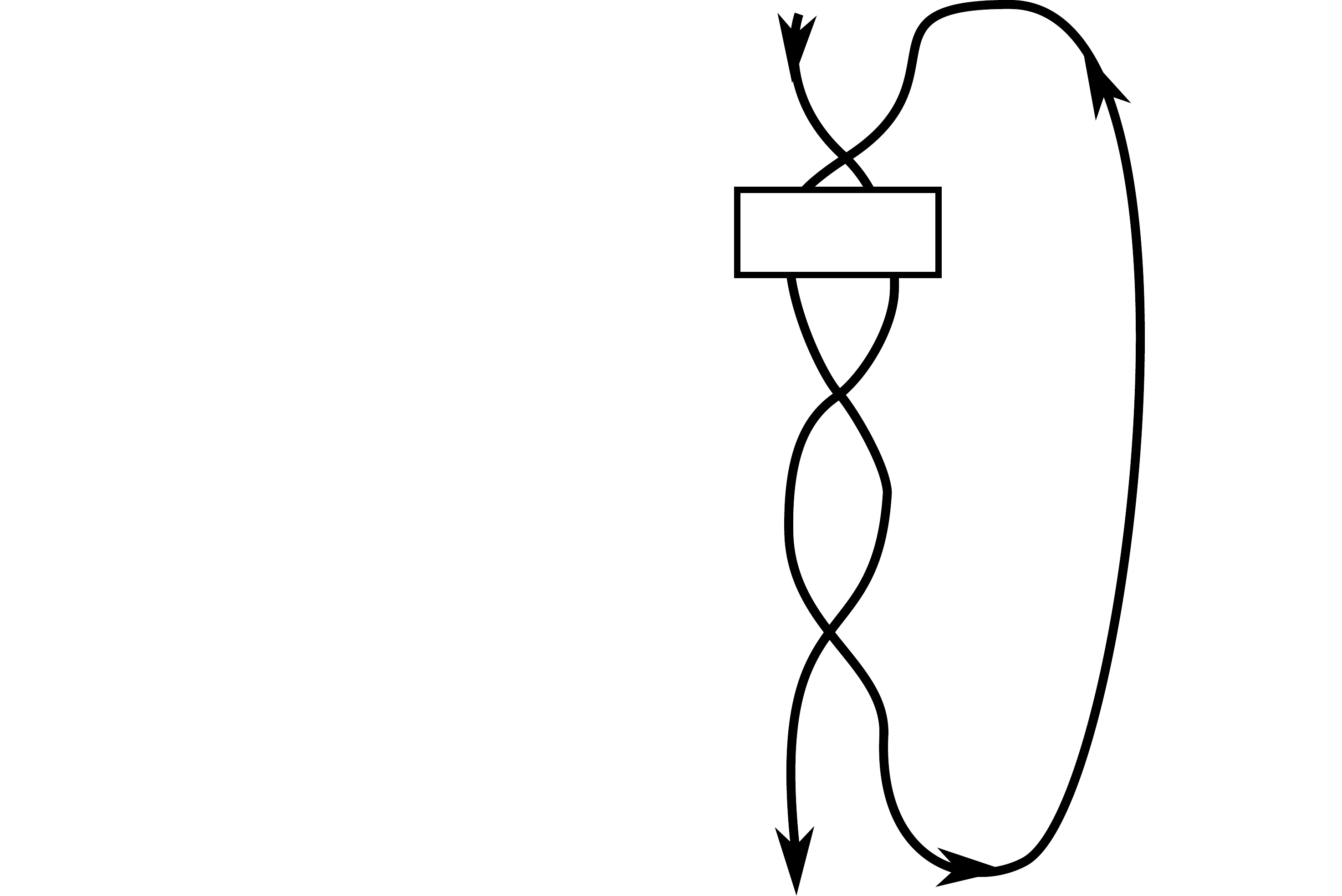}%
   \caption{R-matrices seen has sums.}
   \label{trefoil_coupons2}
 \end{subfigure}%
 \caption{The example of the trefoil knot.}
 \label{Trefoil_universal}
 \end{figure}

\paragraph*{}
We now slide down - following the orientation - the (1,1)-coupons $E^n$, $F^{(n)}$ and $K^n$, taking first (at any step) the closest to the bottom (see Figure \ref{trefoil_coupons_move1}). During this process, the only non trivial commutations that appear are between $E^n \otimes 1$  or $1 \otimes E^n$ or $F^{(n)} \otimes 1$ or $1 \otimes F^{(n)}$ and $q^{\frac{H \otimes H}{2}}$ or $q^{-\frac{H \otimes H}{2}}$. By Lemma \ref{lemmaCommutation}, this only add some $(1,1)$ coupons labeled by $K^{\pm n}$ (see Figure \ref{trefoil_coupons_move2}).

\begin{figure}[h!]
\begin{subfigure}[b]{1\textwidth}
 \centering
  \def\svgwidth{135mm}
    \resizebox{85mm}{!}{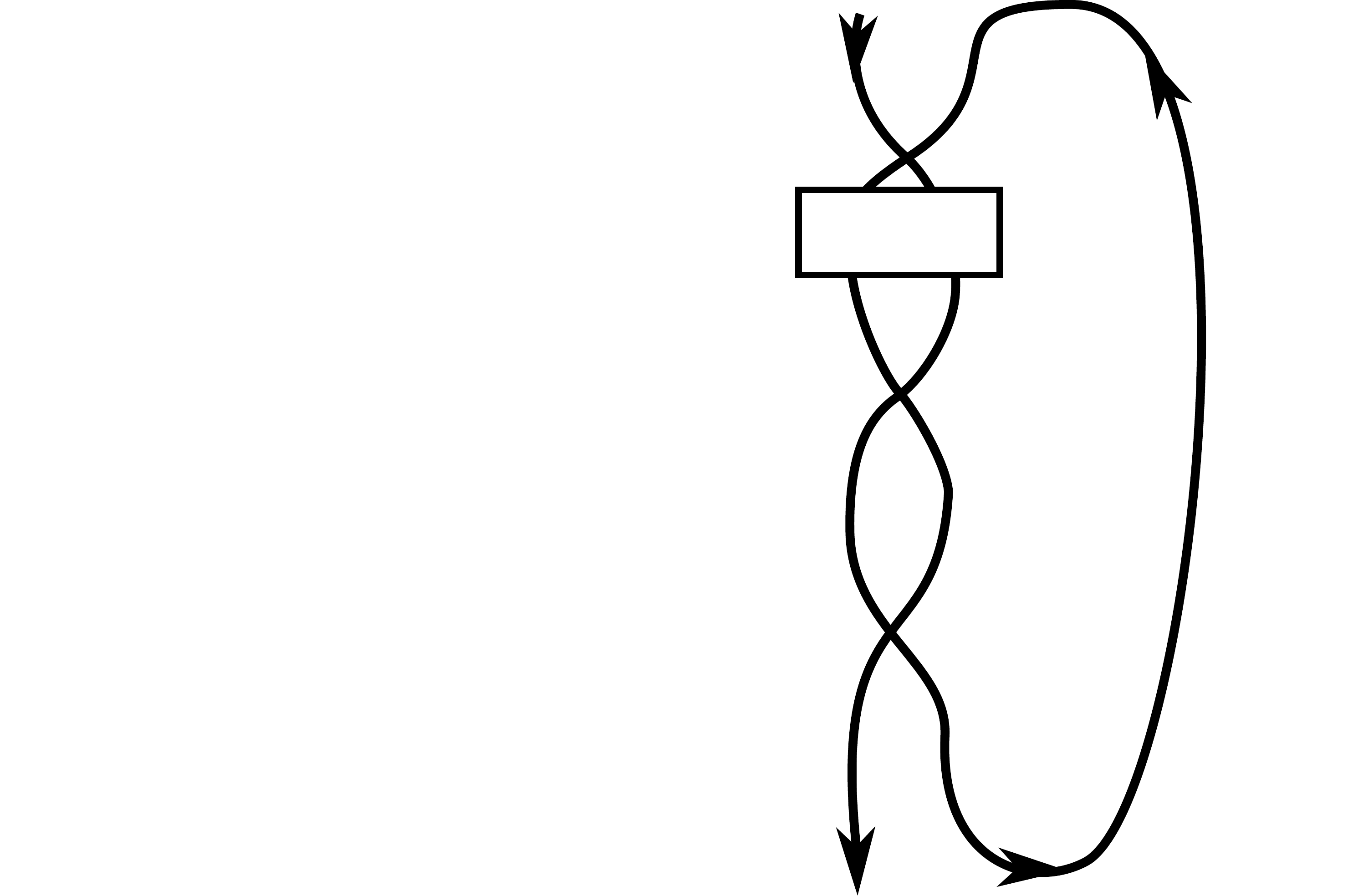}%
   \caption{The first two steps, sliding coupons.}
	\label{trefoil_coupons_move1}
 \end{subfigure}%
 
 \begin{subfigure}[b]{1\textwidth}
 \centering
  \def\svgwidth{135mm}
    \resizebox{85mm}{!}{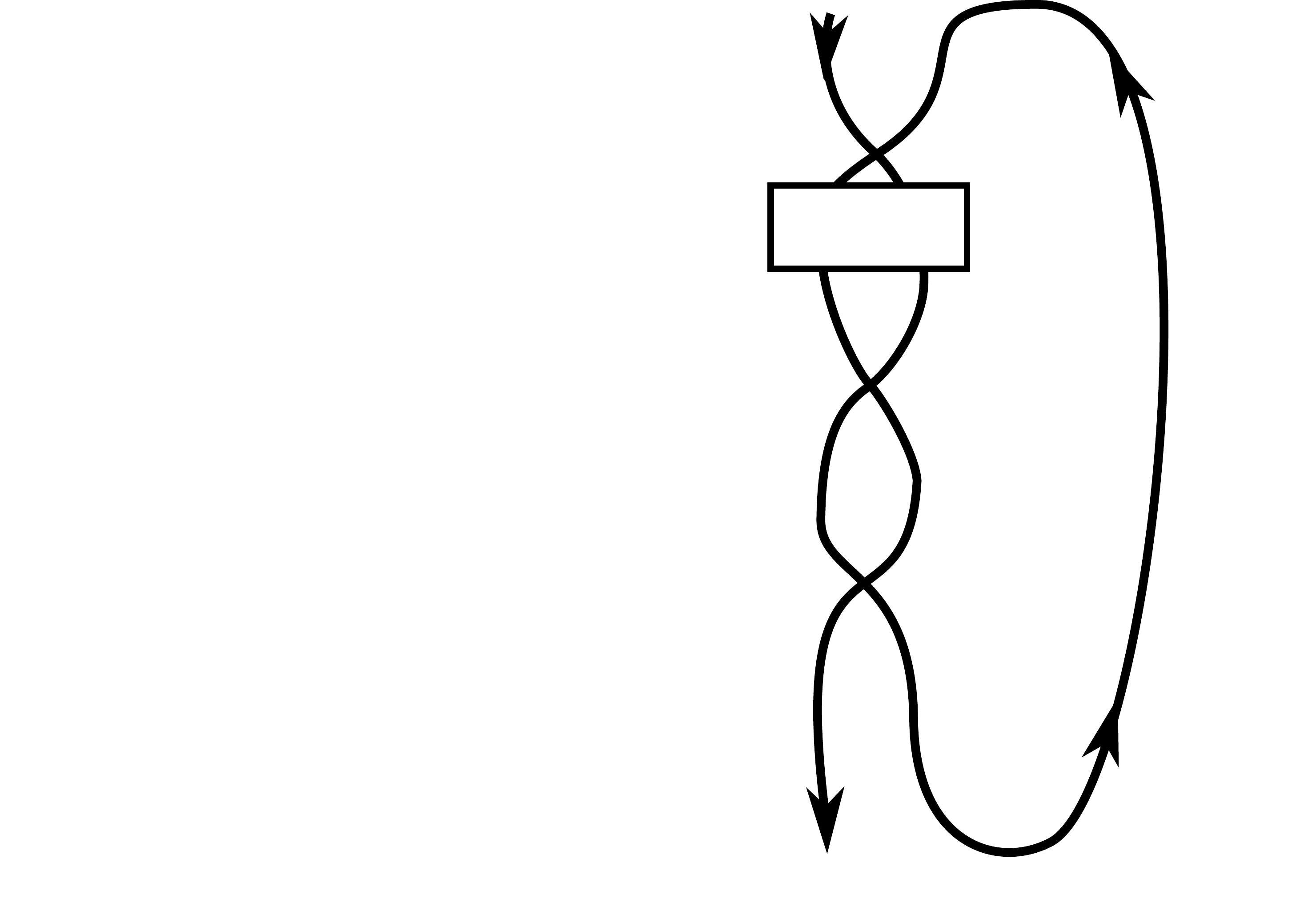}%
  \caption{The third step, passing through the (2,2) coupons.}
   \label{trefoil_coupons_move2}
 \end{subfigure}%
 \caption{Sliding coupons example with the trefoil knot.}
 \label{Trefoil_move}
 \end{figure}

\paragraph*{}
We are only left with coupons labeled by $q^{\frac{H \otimes H}{2}}$ and $q^{-\frac{H \otimes H}{2}}$, you can see Figure \ref{trefoil_coupons_quad1} for the example of the trefoil knot. These are exponentials and we can see them as the sum $q^{\pm \frac{H \otimes H}{2}}= \sum (\frac{(\pm h)^n}{2^n n!}) H^n \otimes H^n$, the (2,2)-coupons now become a sum of (1,1)-coupons as shown for the trefoil knot in Figure \ref{trefoil_coupons_quad2}.\\
Now take one of the (1,1)-coupon labeled by $H^n$ and slide it towards the second (1,1)-coupon labeled by $H^n$ (see Figure \ref{trefoil_coupons_quad2}). Since one is only left with coupons labeled by powers of $H$, everything commutes and we get coupons $H^{2n}$ (as shown in Figure \ref{trefoil_coupons_quad3}), summing them over $n$ gives us (1,1)-coupons labeled by $q^{\pm \frac{H^2}{2}}$ (see Figure \ref{trefoil_coupons_quad4}).

\begin{figure}[h!]
\begin{subfigure}[b]{0.45\textwidth}
 \centering
  \def\svgwidth{32mm}
    \resizebox{25mm}{!}{
\begingroup%
  \makeatletter%
  \providecommand\color[2][]{%
    \errmessage{(Inkscape) Color is used for the text in Inkscape, but the package 'color.sty' is not loaded}%
    \renewcommand\color[2][]{}%
  }%
  \providecommand\transparent[1]{%
    \errmessage{(Inkscape) Transparency is used (non-zero) for the text in Inkscape, but the package 'transparent.sty' is not loaded}%
    \renewcommand\transparent[1]{}%
  }%
  \providecommand\rotatebox[2]{#2}%
  \newcommand*\fsize{\dimexpr\f@size pt\relax}%
  \newcommand*\lineheight[1]{\fontsize{\fsize}{#1\fsize}\selectfont}%
  \ifx\svgwidth\undefined%
    \setlength{\unitlength}{259.77534016bp}%
    \ifx\svgscale\undefined%
      \relax%
    \else%
      \setlength{\unitlength}{\unitlength * \real{\svgscale}}%
    \fi%
  \else%
    \setlength{\unitlength}{\svgwidth}%
  \fi%
  \global\let\svgwidth\undefined%
  \global\let\svgscale\undefined%
  \makeatother%
  \begin{picture}(1,2.180772)%
    \lineheight{1}%
    \setlength\tabcolsep{0pt}%
    \put(0,0){\includegraphics[width=\unitlength,page=1]{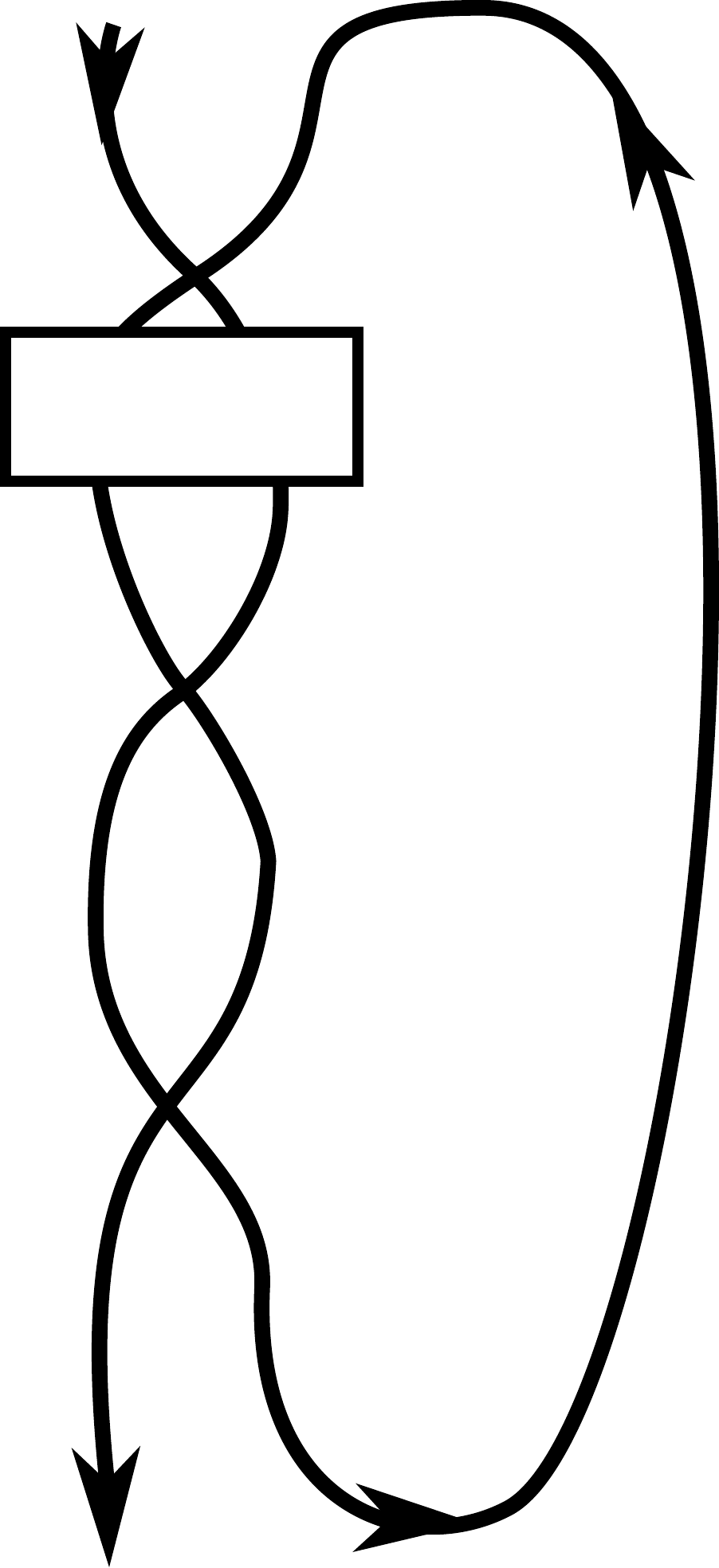}}%
    \put(0.12366914,1.58007269){\color[rgb]{0,0,0}\makebox(0,0)[lt]{\lineheight{1.25}\smash{\begin{tabular}[t]{l}$q^{\frac{H \otimes H}{2}}$\end{tabular}}}}%
    \put(0,0){\includegraphics[width=\unitlength,page=2]{trefoil_coupons_quad_1.pdf}}%
    \put(0.12366914,1.01436425){\color[rgb]{0,0,0}\makebox(0,0)[lt]{\lineheight{1.25}\smash{\begin{tabular}[t]{l}$q^{\frac{H \otimes H}{2}}$\end{tabular}}}}%
    \put(0,0){\includegraphics[width=\unitlength,page=3]{trefoil_coupons_quad_1.pdf}}%
    \put(0.12366914,0.4281635){\color[rgb]{0,0,0}\makebox(0,0)[lt]{\lineheight{1.25}\smash{\begin{tabular}[t]{l}$q^{\frac{H \otimes H}{2}}$\end{tabular}}}}%
  \end{picture}%
\endgroup%
}%
   \caption{Quadratic part of trefoil universal invariant.}
	\label{trefoil_coupons_quad1}
 \end{subfigure}\hfill%
 \begin{subfigure}[b]{0.45\textwidth}
 \centering
  \def\svgwidth{55mm}
    \resizebox{42mm}{!}{
\begingroup%
  \makeatletter%
  \providecommand\color[2][]{%
    \errmessage{(Inkscape) Color is used for the text in Inkscape, but the package 'color.sty' is not loaded}%
    \renewcommand\color[2][]{}%
  }%
  \providecommand\transparent[1]{%
    \errmessage{(Inkscape) Transparency is used (non-zero) for the text in Inkscape, but the package 'transparent.sty' is not loaded}%
    \renewcommand\transparent[1]{}%
  }%
  \providecommand\rotatebox[2]{#2}%
  \newcommand*\fsize{\dimexpr\f@size pt\relax}%
  \newcommand*\lineheight[1]{\fontsize{\fsize}{#1\fsize}\selectfont}%
  \ifx\svgwidth\undefined%
    \setlength{\unitlength}{428.36134543bp}%
    \ifx\svgscale\undefined%
      \relax%
    \else%
      \setlength{\unitlength}{\unitlength * \real{\svgscale}}%
    \fi%
  \else%
    \setlength{\unitlength}{\svgwidth}%
  \fi%
  \global\let\svgwidth\undefined%
  \global\let\svgscale\undefined%
  \makeatother%
  \begin{picture}(1,1.32250679)%
    \lineheight{1}%
    \setlength\tabcolsep{0pt}%
    \put(0,0){\includegraphics[width=\unitlength,page=1]{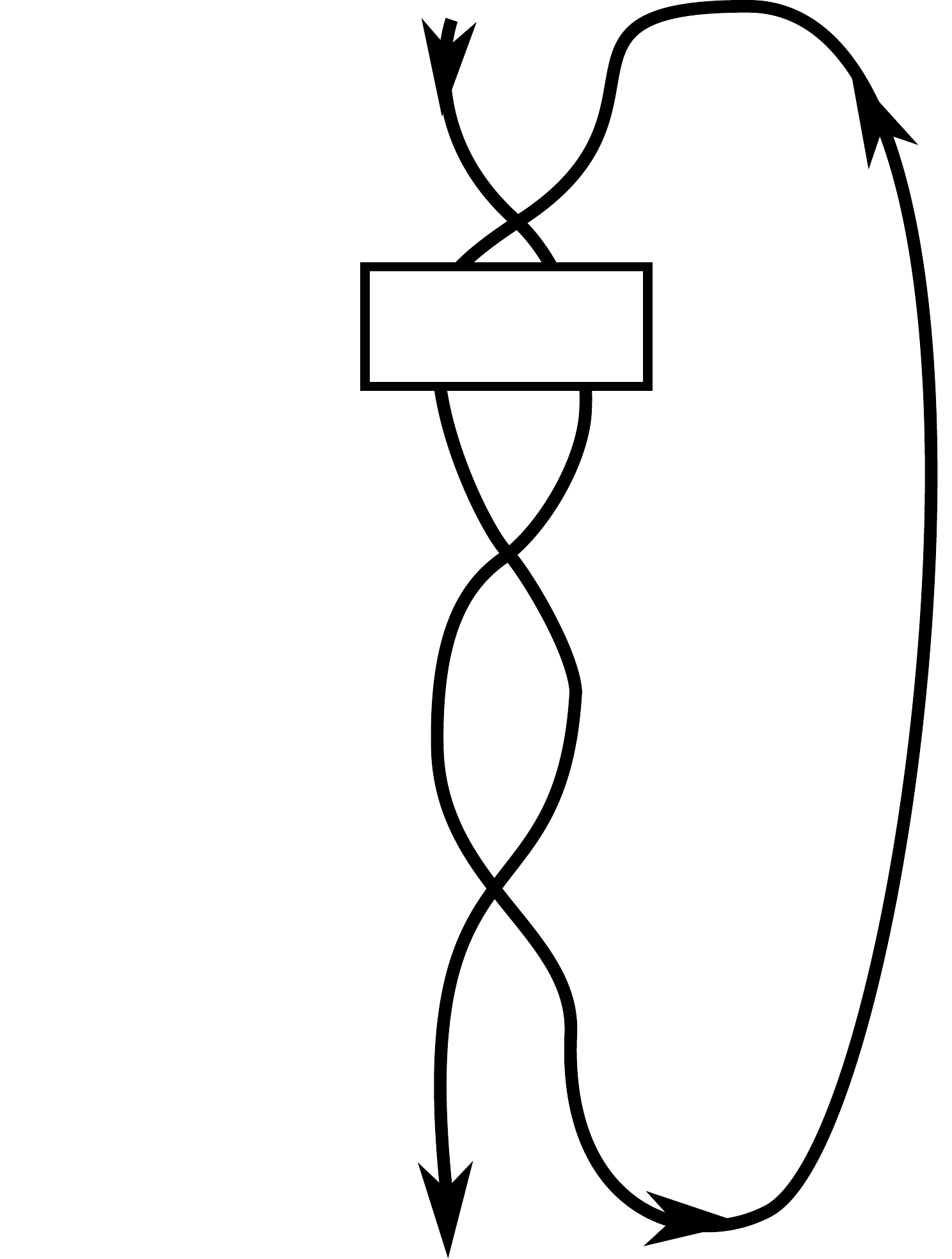}}%
    \put(0.45375596,0.95821884){\color[rgb]{0,0,0}\makebox(0,0)[lt]{\lineheight{1.25}\smash{\begin{tabular}[t]{l}$q^{\frac{H \otimes H}{2}}$\end{tabular}}}}%
    \put(0,0){\includegraphics[width=\unitlength,page=2]{trefoil_coupons_quad_2.pdf}}%
    \put(0.45375596,0.61515081){\color[rgb]{0,0,0}\makebox(0,0)[lt]{\lineheight{1.25}\smash{\begin{tabular}[t]{l}$q^{\frac{H \otimes H}{2}}$\end{tabular}}}}%
    \put(0,0){\includegraphics[width=\unitlength,page=3]{trefoil_coupons_quad_2.pdf}}%
    \put(0.40784122,0.21960767){\color[rgb]{0,0,0}\makebox(0,0)[lt]{\lineheight{1.25}\smash{\begin{tabular}[t]{l}$H^n$\end{tabular}}}}%
    \put(0.55163256,0.21960767){\color[rgb]{0,0,0}\makebox(0,0)[lt]{\lineheight{1.25}\smash{\begin{tabular}[t]{l}$H^n$\end{tabular}}}}%
    \put(0,0.63294421){\color[rgb]{0,0,0}\makebox(0,0)[lt]{\lineheight{1.25}\smash{\begin{tabular}[t]{l}\scalebox{1.3}{$\underset{n}{\sum} \frac{h^n}{2^n [n]!} \times $}\end{tabular}}}}%
    \put(0,0){\includegraphics[width=\unitlength,page=4]{trefoil_coupons_quad_2.pdf}}%
  \end{picture}%
\endgroup%
}%
  \caption{Illustration of quadratic simplification: first step.}
   \label{trefoil_coupons_quad2}
 \end{subfigure}%

 \begin{subfigure}[b]{0.45\textwidth}
 \centering
  \def\svgwidth{55mm}
    \resizebox{42mm}{!}{
\begingroup%
  \makeatletter%
  \providecommand\color[2][]{%
    \errmessage{(Inkscape) Color is used for the text in Inkscape, but the package 'color.sty' is not loaded}%
    \renewcommand\color[2][]{}%
  }%
  \providecommand\transparent[1]{%
    \errmessage{(Inkscape) Transparency is used (non-zero) for the text in Inkscape, but the package 'transparent.sty' is not loaded}%
    \renewcommand\transparent[1]{}%
  }%
  \providecommand\rotatebox[2]{#2}%
  \newcommand*\fsize{\dimexpr\f@size pt\relax}%
  \newcommand*\lineheight[1]{\fontsize{\fsize}{#1\fsize}\selectfont}%
  \ifx\svgwidth\undefined%
    \setlength{\unitlength}{428.36134543bp}%
    \ifx\svgscale\undefined%
      \relax%
    \else%
      \setlength{\unitlength}{\unitlength * \real{\svgscale}}%
    \fi%
  \else%
    \setlength{\unitlength}{\svgwidth}%
  \fi%
  \global\let\svgwidth\undefined%
  \global\let\svgscale\undefined%
  \makeatother%
  \begin{picture}(1,1.32250679)%
    \lineheight{1}%
    \setlength\tabcolsep{0pt}%
    \put(0,0){\includegraphics[width=\unitlength,page=1]{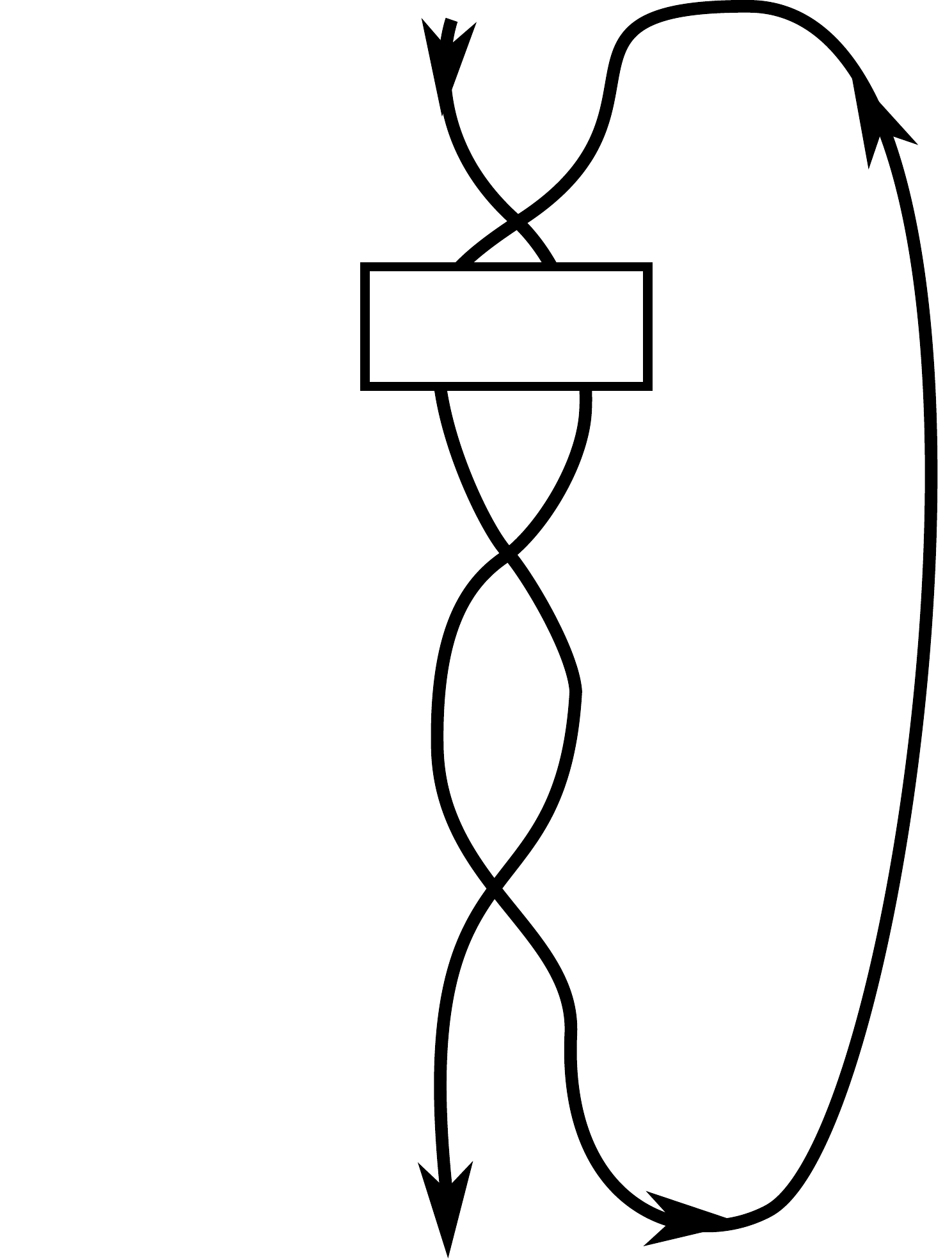}}%
    \put(0.45375603,0.95821887){\color[rgb]{0,0,0}\makebox(0,0)[lt]{\lineheight{1.25}\smash{\begin{tabular}[t]{l}$q^{\frac{H \otimes H}{2}}$\end{tabular}}}}%
    \put(0,0){\includegraphics[width=\unitlength,page=2]{trefoil_coupons_quad_3.pdf}}%
    \put(0.45375603,0.61515074){\color[rgb]{0,0,0}\makebox(0,0)[lt]{\lineheight{1.25}\smash{\begin{tabular}[t]{l}$q^{\frac{H \otimes H}{2}}$\end{tabular}}}}%
    \put(0,0){\includegraphics[width=\unitlength,page=3]{trefoil_coupons_quad_3.pdf}}%
    \put(0.40033757,0.21710638){\color[rgb]{0,0,0}\makebox(0,0)[lt]{\lineheight{1.25}\smash{\begin{tabular}[t]{l}$H^{2n}$\end{tabular}}}}%
    \put(-0,0.63294413){\color[rgb]{0,0,0}\makebox(0,0)[lt]{\lineheight{1.25}\smash{\begin{tabular}[t]{l}\scalebox{1.3}{$\underset{n}{\sum} \frac{h^n}{2^n [n]!} \times $}\end{tabular}}}}%
  \end{picture}%
\endgroup%
}%
   \caption{Illustration of quadratic simplification: second step.}
	\label{trefoil_coupons_quad3}
 \end{subfigure}\hfill%
 \begin{subfigure}[b]{0.45\textwidth}
 \centering
  \def\svgwidth{32mm}
    \resizebox{25mm}{!}{
\begingroup%
  \makeatletter%
  \providecommand\color[2][]{%
    \errmessage{(Inkscape) Color is used for the text in Inkscape, but the package 'color.sty' is not loaded}%
    \renewcommand\color[2][]{}%
  }%
  \providecommand\transparent[1]{%
    \errmessage{(Inkscape) Transparency is used (non-zero) for the text in Inkscape, but the package 'transparent.sty' is not loaded}%
    \renewcommand\transparent[1]{}%
  }%
  \providecommand\rotatebox[2]{#2}%
  \newcommand*\fsize{\dimexpr\f@size pt\relax}%
  \newcommand*\lineheight[1]{\fontsize{\fsize}{#1\fsize}\selectfont}%
  \ifx\svgwidth\undefined%
    \setlength{\unitlength}{259.77529691bp}%
    \ifx\svgscale\undefined%
      \relax%
    \else%
      \setlength{\unitlength}{\unitlength * \real{\svgscale}}%
    \fi%
  \else%
    \setlength{\unitlength}{\svgwidth}%
  \fi%
  \global\let\svgwidth\undefined%
  \global\let\svgscale\undefined%
  \makeatother%
  \begin{picture}(1,2.18077253)%
    \lineheight{1}%
    \setlength\tabcolsep{0pt}%
    \put(0,0){\includegraphics[width=\unitlength,page=1]{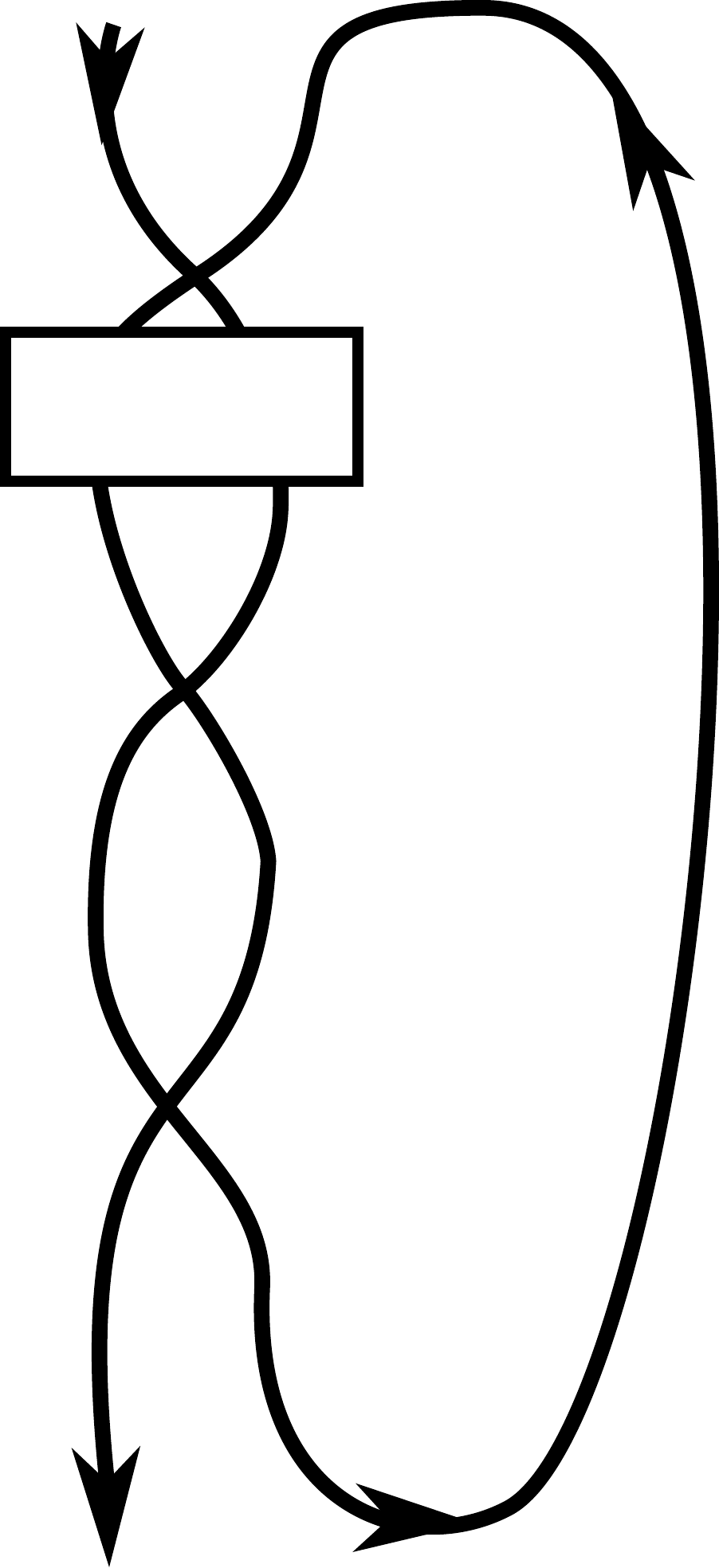}}%
    \put(0.12366899,1.5800732){\color[rgb]{0,0,0}\makebox(0,0)[lt]{\lineheight{1.25}\smash{\begin{tabular}[t]{l}$q^{\frac{H \otimes H}{2}}$\end{tabular}}}}%
    \put(0,0){\includegraphics[width=\unitlength,page=2]{trefoil_coupons_quad_4.pdf}}%
    \put(0.12366899,1.01436449){\color[rgb]{0,0,0}\makebox(0,0)[lt]{\lineheight{1.25}\smash{\begin{tabular}[t]{l}$q^{\frac{H \otimes H}{2}}$\end{tabular}}}}%
    \put(0,0){\includegraphics[width=\unitlength,page=3]{trefoil_coupons_quad_4.pdf}}%
    \put(0.03570925,0.36639101){\color[rgb]{0,0,0}\makebox(0,0)[lt]{\lineheight{1.25}\smash{\begin{tabular}[t]{l}$q^{\frac{H^2}{2}}$\end{tabular}}}}%
  \end{picture}%
\endgroup%
}%
  \caption{Illustration of quadratic simplification: third step.}
   \label{trefoil_coupons_quad4}
 \end{subfigure}%
 \caption{Quadratic factorization and simplification for the trefoil knot.}
 \label{Trefoil_quad}
 \end{figure}

\paragraph*{}
We thus have the following proposition:

\begin{prop}
If $\mathcal{K}$ a knot and $D$ a diagram of a 1-1 tangle $T$ whose closure is $\mathcal{K}$, then:
\[ Q^{U_h}(\mathcal{K})= q^{f \frac{H^2}{2}} Q^{\mathcal{\tilde{U}}} (D) \]
where $Q^{\mathcal{\tilde{U}}} (D) \in \mathcal{\tilde{U}}$ and $f$ is the writhe of the diagram.
\end{prop}

\subsection{The Verma module}

Now it is time to construct the Verma module on which the universal invariant will act as $F_{\infty}(q,A,D)$.

Let $V^{\alpha}$ be a $\hat{R}^{\hat{I}}$-module freely generated by vectors $\{ v_0, v_1, \dots\} $, and we endow it with a action of $\mathcal{\tilde{U}}$:
\[ Ev_0 = 0, \ \ Ev_{i+1} = v_i, \ \ Kv_i=q^{\alpha -2i} v_i, \ \ F^{(n)} v_i=\qbinom{n+i}{i}_q \{ \alpha -i ; n \}_q v_{n+i}\]

\begin{prop}
The endomorphism of $V^{\alpha}$ are scalars i.e. \[End_{\mathcal{\tilde{U}}} (V^{\alpha}) =\hat{R}^{\hat{I}} Id_{V^{\alpha}}.\]
\end{prop}
\begin{proof}
Let $f \in End_{\mathcal{\tilde{U}}} (V^{\alpha}) $.\\
$K f(v_i)=f(K v_i)=q^{\alpha-2i} f(v_i)$, thus $\exists \lambda_i \in \hat{R}^{\hat{I}}$ such that $f(v_i)= \lambda_i v_i$.\\
Now since $E f(v_{i+1})=f(Ev_{i+1})=f(v_i)$, then $\lambda_{i+1} v_i= \lambda_i v_i$ hence we define $\lambda:=\lambda_i$ and we have $f=\lambda Id_{V^{\alpha}}$. 
\end{proof}

We set $q^{\pm \frac{H^2}{2}} v_i = q^{\pm \frac{(\alpha-2i)^2}{2}} v_i \in q^{\pm \frac{\alpha^2}{2}} V^{\alpha}$.

\begin{prop}
$ Q^{U_h}(\mathcal{K})$ is in the center of $U_h$.
\end{prop}
\begin{proof}
See \cite{habiro2006bottom} Proposition 8.2. 
\end{proof}
\begin{prop}
If $\mathcal{K}$ is a knot and $D$ is a diagram of a 1-1 tangle $T$ whose closure is $\mathcal{K}$, then:
\[Q^{U_h}(\mathcal{K}) v_0 = q^{f \frac{H^2}{2}} Q^{\mathcal{\tilde{U}}} (D) v_0= F_{\infty}(q,A,D) v_0.\]
Hence, in particular, $F_{\infty}(q,A,D)$ is independent of the choice of the diagram and we denote it $F_{\infty}(q,A,  \mathcal{K})$.
\end{prop}
\begin{proof}
Using analogously the useful form of ADO invariant defined in Proposition \ref{ADO_form} to create a similar "useful form" for the action of $Q^{U_h}(\mathcal{K})$, one gets exactly the definition of $F_{\infty} (q, A,D)$ in Definition \ref{unified_form}. 
\end{proof}

\subsection{Connection to Habiro's work}
This subsection will be dedicated to connect our setup (the ring setup $\hat{R}^{\hat{I}}$ and the quantum algebra setup $\tilde{\mathcal{U}}$) to Habiro's algebra setup in \cite{habiro2007integral}. As we will see, we will get that our ring $\hat{R}^{\hat{I}}$ is contained in some $h$-adic ring, and hence that it is an integral domain. Moreover we will see how to get back our unified invariant $F_{\infty}(q,A,\mathcal{K})$ from Habiro's quantum algebra completions.

\medskip
Let us define $U_h^0$ the $\Q[[h]]$ subalgebra of $U_h$ topologically generated by $H$. And $\mathcal{U}^0$ the $\Z[q^{\pm1}]$ subalgebra of $\mathcal{U}$ generated by $H$. We can complete the algebra into  \[\hat{\mathcal{U}}^0= \underset{\underset{n}{\leftarrow}}{\lim} \dfrac{\mathcal{U}^0}{< \{H+m;n \}, \ m \in \Z>}.\]

We then have the following proposition.

\begin{prop}~\\
$U_h^0 \cong \Q[\alpha] [[h]]$ the $h$-adic completion of the polynomial ring $\Q[\alpha]$,\\
$\hat{\mathcal{U}}^0 \cong \hat{R}^{\hat{I}}$.
\end{prop}
\begin{proof}
The first statement is the definition of $U_h^0$ replacing $H$ with formal variable $\alpha$.\\
The second statement comes from the fact that, replacing $K$ by $A$, $U^0 \cong \Z[q^{\pm1},A^{\pm1}]$ and $\{H+m;n\} \cong I_n$.
\end{proof}

Now, one can use Proposition 6.8 and 6.9 in Habiro's article \cite{habiro2007integral} and we have:
\begin{prop}\label{integral_domain_completion_ring}
We have that $\hat{R}^{\hat{I}} \subset \Q[\alpha] [[h]]$, and thus $\hat{R}^{\hat{I}}$ is an integral domain.
\end{prop}

Moreover, elements in $\hat{R}^{\hat{I}}$ can be uniquely expressed. This fact comes from Corollary 5.5 in \cite{habiro2007integral}. Recall that $q^{\alpha}:=A$ and let $\{\alpha;n\}'=\prod_{i=0}^{n-1} (q^{2\alpha}-q^{i})$ we have the following proposition:
\begin{prop}\label{unique_expression_completion_ring} We have the following isomorphism:
\[\hat{R}^{\hat{I}}\cong \underset{\underset{n}{\leftarrow}}{\lim} \dfrac{\widehat{\Z[q]}[A]}{(\{\alpha;n\}')}.\]
Moreover, any element $t\in \hat{R}^{\hat{I}}$ can be uniquely written $\sum_{n=0}^{\infty} t_n \{\alpha;n\}'$ where $t_n \in \widehat{\Z[q]}+ \widehat{\Z[q]}A$.
\end{prop}

\begin{proof}
See Corollary 5.5 in \cite{habiro2007integral}.
\end{proof}

\begin{rem}
This means that the unified invariant $F_{\infty}(q,A,\mathcal{K})$ can be uniquely written as a serie $\sum_{n=0}^{\infty} t_n \{\alpha;n\}'$ where $t_n \in \widehat{\Z[q]}+ \widehat{\Z[q]}A$.
\end{rem}

\medskip
Now let us present the quantum algebra setup used by Habiro. Our algebra and completion was done for the sole purpose of getting a nice form for our unified invariant, allowing us to factorize it at each roots of unity. Habiro's quantum algebra setup has been studied more in details, and thus have more proprieties.
\medskip

\noindent Let $\mathcal{U}_{Hab}$ be the $\Z[q^{\pm 1}]$ subalgebra of $U_h$ generated by elements $K^{\pm1}, \ e, \ F^{[n]}$ where $e= \{1\}E$ and $F^{[n]}= \frac{F^n}{\{n\}!}$. Let $\tilde{J}_n$ be the ideal generated by elements $e^i \{H+m; n-i \}$ for all $m \in \Z$.\\
We denote $\mathcal{\hat{U}}_{Hab}:= \underset{\underset{n}{\leftarrow}}{\lim} \dfrac{\mathcal{U}_{Hab}}{\tilde{J}_n}$.

Now that his algebra setup is stated, let us make the connection with our unified invariant $F_{\infty}(q,A,\mathcal{K})$. To do so, we will build a unified invariant with Habiro's setup and prove that it is in fact $F_{\infty}(q,A,\mathcal{K})$.

First remark that:
\begin{rem}
If $\mathcal{K}$ a knot and $D$ a diagram of a 1-1 tangle $T$ whose closure is $\mathcal{K}$, then:
\[ Q^{U_h}(\mathcal{K})= q^{f \frac{H^2}{2}} Q^{\mathcal{\tilde{U}}} (D) \]
where $Q^{\mathcal{\tilde{U}}} (D) \in \mathcal{\hat{U}}_{Hab}$ and $f$ is the writhe of the diagram.
\end{rem}

We can then define the corresponding Verma module on $\hat{R}^{\hat{I}}$.\\
Let $V^{\alpha}_{Hab}$ be the $\hat{R}^{\hat{I}}$-module freely generated by vectors $\{ v_0, v_1, \dots\} $, endowed with an action of $\mathcal{\hat{U}}_{Hab}$:
\[ ev_0 = 0, \ \ ev_{i+1} = \{ \alpha -i \}_q v_i, \ \ Kv_i=q^{\alpha -2i} v_i, \ \ F^{[n]} v_i=\qbinom{n+i}{i}_q  v_{n+i}.\]

Then in a similar fashion we have that:

\begin{rem}
If $\mathcal{K}$ is a knot and $D$ is a diagram of a 1-1 tangle $T$ whose closure is $\mathcal{K}$, then there exists an element $F_{\infty}^{Hab}(q,A,\mathcal{K}) \in  q^{f \frac{\alpha^2}{2}} \times \hat{R}^{\hat{I}}$ such that:
\[Q^{U_h}(\mathcal{K}) v_0 = q^{f \frac{H^2}{2}} Q^{\mathcal{\tilde{U}}} (D) v_0= F_{\infty}^{Hab}(q,A,\mathcal{K}) v_0.\]
\end{rem}

Now, since $V^{\alpha}_{Hab} \hat{\otimes} \Q[[H]] \cong V^{\alpha} \hat{\otimes} \Q[[H]]$ as $U_h$ module with coefficient in $\Q[\alpha][[h]]$. Then $F_{\infty}^{Hab}(q,A,\mathcal{K})=F_{\infty}(q,A,\mathcal{K})$ as elements of $\Q[\alpha][[h]]$. And since $\hat{R}^{\hat{I}} \subset \Q[\alpha] [[h]]$, $F_{\infty}^{Hab}(q,A,\mathcal{K})=F_{\infty}(q,A,\mathcal{K})$ as elements of $\hat{R}^{\hat{I}}$.

\subsection{About the colored Jones polynomials}
\paragraph{The colored Jones polynomials and the study of $C_{\infty}(1, A,\mathcal{K})$}~\\
Knowing that the unified invariant comes from the universal invariant, we can use this fact to recover the colored Jones polynomials. When we evaluate $A=q^{\alpha}$ at $q^n$ in $F_{\infty}(q, A, \mathcal{K})$, we obtain the $n$-colored Jones polynomial denoted $J_n (q, \mathcal{K})$ (with normalization $J_n(q, unknot)=1$). Moreover we can use this fact to compute
$C_{\infty}(1, A,\mathcal{K})$ as the inverse of the Alexander polynomial, simplifying the factorization in Proposition \ref{ADO_factor_prop}.

\begin{lemme}
If we denote $V_n$ the $(n+1)$ dimensional highest weight module of $\mathcal{\tilde{U}}$ and $V^n$ the Verma module of highest weight $q^n$ and highest weight vector $v_0$, then $V_n \cong \mathcal{\tilde{U}} v_0 \subset V^n$.
\end{lemme}

\begin{proof}
It comes down to $F^{(n+1)} v_0 = \{ n ;n+1 \} v_{n+1} = 0 \times v_{n+1}=0$. 
\end{proof}

\begin{prop}
If $V^n$ is the Verma of highest weight $q^n$, \[Q^{U_h}(\mathcal{K})v_0=q^{\frac{fn^2}{2}} q^{fn} J_n(q^2, \mathcal{K})v_0.\]
\end{prop}

Moreover, if we denote $\widehat{\Z[q^{\pm 1}]}$ Habiro's ring completion of $\Z[q^{\pm 1}]$ by ideals $(\{n\}!)$, we have some well defined evaluation maps $j_n:\hat{R}^{\hat{I}} \to \widehat{\Z[q^{\pm 1}]}$ $q^{\alpha} \mapsto q^n$.

This allows us to state the following corollary.

\begin{cor}\label{cor_Jones}
$F_{\infty}(q,q^n,  \mathcal{K})=q^{\frac{f n^2}{2}} q^{fn} J_n(q^2, \mathcal{K})$
\end{cor}

Hence, $F_{\infty}(q,A,  \mathcal{K})$ plays a double role in this dance, evaluating its first variable $q$ at a root of unity $\zeta_{2r}$, gives us the $r$-th ADO polynomial multiplied by this $C_{\infty}(r,A,  \mathcal{K})$ element. But if one evaluates the second variable $A$ at $q^n$, one gets the $n$-th colored Jones polynomial.

We will use this double role to study the factorization of ADO polynomials. Indeed, the Melvin-Morton-Rozansky conjecture (MMR) proved by Bar-Natan and Garoufalidis in \cite{bar1996melvin} makes the junction between the inverse of the Alexander polynomial and the colored Jones polynomials. We will use the $h$-adic version that we state below in Theorem \ref{MMR} (see \cite{garoufalidis2005analytic} Theorem 2).
\medskip

We denote $A_{\mathcal{K}}(t)$ the Alexander polynomial of the knot $\mathcal{K}$, with normalisation $A_{unknot}(t)=1$ and $A_{\mathcal{K}}(1)=1$.

\begin{thm}(Bar-Natan, Garoufalidis) \label{MMR} ~\\
For $\mathcal{K}$ a knot, we have the following equality in $\Q[[h]]$:
\[ \underset{n \to \infty}{\lim} \  J_n (e^{h/n}) = \frac{1}{A_{\mathcal{K}} (e^h)}\]
\end{thm}

For the sake of simplicity, let's assume that the knot $\mathcal{K}$ is $0$ framed so $f=0$.

Now note that since $F_{\infty}(q,q^n,  \mathcal{K})= J_n(q^2, \mathcal{K}) \in \Z[q^{\pm1}]$ and $\Z[q^{\pm1}] \subset \Q[[h]]$ we have a map $\Q[[h]] \to \Q[[h]]$, $h \mapsto \frac{h}{n}$ that sends $F_{\infty}(q,q^n,  \mathcal{K}) \mapsto F_{\infty}(q^{1/n}, q,  \mathcal{K})$ as elements of $\Q[[h]]$. But now $F_{\infty}(q^{1/n}, q,  \mathcal{K})$ converges to $F_{\infty}(1, q,  \mathcal{K})$ in the sense stated in \cite{garoufalidis2005analytic} below Theorem 2, namely:
\begin{align*} &\underset{n \to \infty}{\lim} F_{\infty}(q^{1/n}, q,  \mathcal{K}) =F_{\infty}(1, q,  \mathcal{K}) \\ \iff &\underset{n \to \infty}{\lim} \text{coeff}(F_{\infty}(q^{1/n}, q,  \mathcal{K}), h^m) =\text{coeff}(F_{\infty}(1, q,  \mathcal{K}),h^m), \ \forall m \in \N, \end{align*} where, for any analytic function $f$, $\text{coeff}(f(h),h^m)=\frac{1}{m!} \frac{d^m}{dh^m} f(h)|_{h=0}.$
\medskip

\noindent By Theorem \ref{MMR}, $F_{\infty}(1, q, \mathcal{K})= \frac{1}{A_{\mathcal{K}} (q^2)}$ in $ \Q[[h]]$.

\bigskip

On the other hand, if we denote $\widehat{\Z[A^{\pm 1}]^{\{1\}_A}} $ the ring completion of $\Z[A^{\pm 1}]$ by ideals $((A-A^{-1})^n)$, then $F_{\infty}(1, A, \mathcal{K})= C_{\infty}(1,A, \mathcal{K})$ in $\widehat{\Z[A^{\pm 1}]^{\{1\}_A}}$. Indeed setting $q=1$ in Definition \ref{unified_form} and looking at the definition of $C_{\infty}(1,A, \mathcal{K})$ in Proposition \ref{ADO_factor_prop}, one gets $F_{\infty}(1, A, \mathcal{K})=C_{\infty}(1,A, \mathcal{K})$.

\bigskip
Since  $\widehat{\Z[A^{\pm 1}]^{\{1\}_A}}  \xhookrightarrow{} \Q[[h]]$, $A \to e^h$ (see \cite{habiro2007integral} Proposition 6.1 and \cite{habiro2004cyclotomic} Corollary 4.1), then we have the following proposition:

\begin{prop} If $\mathcal{K}$ is $0$ framed, $C_{\infty}(1,A, \mathcal{K})= \frac{1}{A_{\mathcal{K}} (A^2)}$
\end{prop}

By the discussion in the paragraph preceding Proposition \ref{propFC}, we have:
\begin{cor} If $\mathcal{K}$ is $0$ framed, $C_{\infty}(r,A, \mathcal{K})= \frac{1}{A_{\mathcal{K}} (A^{2r})}$
\end{cor}

This allows us to state a factorization theorem:
\begin{thm} (Factorization) \\ \label{ADO_factor_thm}
For a knot $\mathcal{K}$ and an integer $r \in \N^*$, we have the following factorization in $\hat{R_r^I}$:
\[ F_{\infty} (\zeta_{2r}, A, \mathcal{K}) =  \frac{A^{rf} \times ADO_r(A,\mathcal{K})}{A_{\mathcal{K}} (A^{2r})} \]
where $f$ is the framing of the knot.
\end{thm}

\paragraph{The colored Jones polynomials determine the ADO polynomials}~\\
One may ask what is the relationship between ADO invariants and the colored Jones polynomials.\\ Does one family of polynomial determines completely the other? For the sake of simplicity the knot $\mathcal{K}$ is supposed $0$ framed in this paragraph.

\bigskip
This is the case for $\{ ADO_r(A,\mathcal{K}) \}_{r \in \N^*} \to \{ J_n(q^2, \mathcal{K}) \}_{n \in \N^*}$, knowing the ADO polynomials allows to find the colored Jones polynomials. This result was stated in \cite{costantino2015relations} Corollary 15.\\
With our setup, we can get back this result as follows.

\begin{rem}
Notice that $F_{\infty}(\zeta_{2r}, \zeta_{2r}^N, \mathcal{K})= J_N(\zeta_r, \mathcal{K})=\frac{ADO_r(\zeta_{2r}^N,\mathcal{K})}{A_{\mathcal{K}} (1)}=ADO_r(\zeta_{2r}^N,\mathcal{K})$.
\end{rem}

Since $J_N$ is a polynomial knowing an infinite number of value of it determines it. Given the family of polynomials $\{ ADO_r(A,\mathcal{K}) \}_{r \in \N^*}$ we then know each value of $J_N$ at any root of unity hence we know $J_N$ entirely.

\bigskip
But we can also have the other way around: \[\{ J_n(q^2, \mathcal{K}) \}_{n \in \N^*} \to\{ ADO_r(A,\mathcal{K}) \}_{r \in \N^*} .\] Knowing only the colored Jones polynomials recover the ADO polynomials. We will prove it by seeing that the colored Jones polynomials determines the unified invariant $F_{\infty}(q, A, \mathcal{K}) $.
\medskip

Let $\forall k \in \N$, $f_k: \Q[\alpha][[h]] \to \Q[[h]]$, $\alpha \mapsto k$ the evaluation map.

\begin{prop}
$\underset{k \in \N}{\cap} ker(f_k)= \{0\}$
\end{prop}
\begin{proof}
Let $x \in \underset{k \in \N}{\cap} ker(f_k)$, we write $x= \sum_n g_n(\alpha) h^n$ where $g_n(\alpha) \in \Q[\alpha]$.\\
 Then, for each $k \in \N$, we have that $g_n(k)=0$ $\forall n$. Since $g_n$ are polynomials that vanish at an infinite number of point, they are $0$.\\
Hence $\underset{k \in \N}{\cap} ker(f_k)= \{0\}$.
\end{proof}

Let $f: \Q[\alpha][[h]] \to \underset{k \in \N}{\prod} \Q[[h]]$, $x \mapsto (f_k(x))_{k \in \N}$. $ker(f)= \underset{k \in \N}{\cap} ker(f_k)= \{0\}$, hence $f$ is injective.

\begin{rem}
For any knot $\mathcal{K}$, $f(F_{\infty}(q, A, \mathcal{K}))= \{ J_n(q^2, \mathcal{K}) \}_{n \in \N^*}$.
\end{rem}

\begin{prop}\label{prop_F_Jones}
For any knot $\mathcal{K}$,  $F_{\infty}(q, A, \mathcal{K})= f^{-1}(\{ J_n(q^2, \mathcal{K}) \}_{n \in \N^*})$.
\end{prop}

Setting $g: \hat{R}^{\hat{I}} \to \underset{r \in \N^*}{\prod} \hat{R_r^I}$, $x \mapsto (ev_r \times \frac{1}{FC_r}(x))_{r\in \N^*}$, we get the following theorem:

\begin{thm}\label{Thm_ADO_Jones}
The map $h=g \circ f^{-1}: Im(f|_{\hat{R}^{\hat{I}}}) \to \underset{r \in \N^*}{\prod} \hat{R_r^I}$ is such that for every knot $\mathcal{K}$, \[ \{ ADO_r(A,\mathcal{K}) \}_{r \in \N^*}= h(\{ J_n(q^2, \mathcal{K}) \}_{n \in \N^*}). \]
\end{thm}

\paragraph{Application 1: the unified invariant and the ADO invariants are $q$-holonomic}~\\
\noindent The fact that the colored Jones polynomials are $q$-holonomic was proved in \cite{garoufalidis2005colored}. Let us state what it means and then let's prove that the unified invariant and the ADO polynomials verify the same holonomic rule. For the sake of simplicity we will work with $0$-framed knot.
\medskip

Let $Q: \Z[q^{\pm1}]^{\N^*} \to \Z[q^{\pm1}]^{\N^*}$ and $E: \Z[q^{\pm1}]^{\N^*} \to \Z[q^{\pm1}]^{\N^*}$ such that: \[(Qf)(n)=q^{2n} f(n), \ \ (Ef)(n)=f(n+1).\]
Note that these operators can be extended to operators on $\Q[[h]]^{\N^*}$.
\medskip

Let us denote $J_{\bullet}(q^2,\mathcal{K})=\{ J_n(q^2, \mathcal{K}) \}_{n \in \N^*}$ the colored Jones function. Now, from Theorem 1 in \cite{garoufalidis2005colored}, for any knot $\mathcal{K}$ there exists a polynomial $\alpha_{\mathcal{K}}(Q,E,q^2)$ such that $\alpha_{\mathcal{K}}(Q,E,q^2)J_{\bullet}(q^2,\mathcal{K})=0$. We say that $J_{\bullet}(q^2,\mathcal{K})$ is $q$-holonomic.
\medskip

We define similar operators on $\Q[\alpha][[h]]$ and show that the same polynomial $\alpha_{\mathcal{K}}$, taken in terms of those new operators, annihilates $F_{\infty}(q,q^{\alpha}, \mathcal{K})$.

Let $\tilde{Q}:\Q[\alpha][[h]] \to \Q[\alpha][[h]]$ and $\tilde{E}: \Q[\alpha][[h]] \to \Q[\alpha][[h]]$ such that if we take $x(\alpha):= \sum_{k=0}^{+\infty} x_k(\alpha) h^k \in \Q[\alpha][[h]]$ with $x_k(\alpha) \in \Q[\alpha]$: 
\[ \tilde{Q}(x(\alpha))=q^{2\alpha} x(\alpha), \ \ \tilde{E}(x(\alpha))=x(\alpha+1).\]
where $x(\alpha+1):= \sum_{k=0}^{+\infty} x_k(\alpha+1) h^k$
\begin{rem}
Here, $\tilde{Q}$ is just the multiplication of any element with $q^{2\alpha}$.
\end{rem}
Notice that if you take the injective map $f: \Q[\alpha][[h]] \to \Q[[h]]^{\N^*}$, $x(\alpha) \mapsto (x(k))_{k \in \N^*}$ previously defined, you have: \[ f \circ \tilde{Q}=Q \circ f, \ \ f \circ \tilde{E}= E \circ f.\]
Hence, $f \circ\alpha_{\mathcal{K}}(\tilde{Q},\tilde{E},q^2)= \alpha_{\mathcal{K}}(Q,E,q^2) \circ f$. Since $\alpha_{\mathcal{K}}(Q,E,q^2)J_{\bullet}(q^2,\mathcal{K})=0$ and $f(F_{\infty}(q,q^{\alpha}, \mathcal{K}))= J_{\bullet}(q^2,\mathcal{K})$ we obtain $f \circ \alpha_{\mathcal{K}}(\tilde{Q},\tilde{E},q^2)(F_{\infty}(q,q^{\alpha}, \mathcal{K}))=0$. The injectivity of $f$ gives the following theorem.

\begin{thm}\label{thm_unified_holonomy}
For any $0$ framed knot $\mathcal{K}$, $\alpha_{\mathcal{K}}(\tilde{Q},\tilde{E},q^2)(F_{\infty}(q,q^{\alpha}, \mathcal{K}))=0$.
\end{thm}
\medskip

Now let us look at what happens at roots of unity. To do so we must restrict ourselves to a ring allowing evaluation at roots of unity such as $\hat{R}^{\hat{I}}$.\\
Since $\tilde{Q}(I_n) \subset I_n$ and $\tilde{E}(I_n) \subset I_n$, we can restrict the operators $\tilde{Q}$ and $\tilde{E}$ to $\hat{R}^{\hat{I}}$, for the sake of simplicity we will still write them $\tilde{Q}: \hat{R}^{\hat{I}} \to \hat{R}^{\hat{I}}$ and $\tilde{E}: \hat{R}^{\hat{I}} \to \hat{R}^{\hat{I}}$.
\medskip

Now let $r\in \N^*$, let $\overline{Q}: \hat{R_r^I} \to \hat{R_r^I}$ and $\overline{E}:\hat{R_r^I} \to \hat{R_r^I}$ such that if we take $x(\alpha)=\sum_{k=0}^{\infty} x_k(\alpha) \{r\alpha\}^k \in \hat{R_r^I}$ with $x_k(\alpha) \in \Z[\zeta_{2r}, A]$ (recall that we denote $\zeta_{2r}^{\alpha}:=A$): 
\[ \overline{Q}(x(\alpha))= \zeta_{2r}^{2\alpha} x(\alpha), \ \ \overline{E}(x(\alpha))=x(\alpha+1)\]
where $x(\alpha+1)= \sum_{k=0}^{\infty} x_k(\alpha+1)(-1)^k \{r\alpha\}^k$.
\medskip

\noindent Since $ev_r \circ \tilde{Q} = \overline{Q} \circ ev_r$ and $ev_r \circ \tilde{E} = \overline{E} \circ ev_r$, the same formula holds: \[\alpha_{\mathcal{K}}(\overline{Q},\overline{E},\zeta_{2r}^2)(F_{\infty}(\zeta_{2r},\zeta_{2r}^{\alpha}, \mathcal{K}))=0. \]\\
By Theorem \ref{ADO_factor_thm}, $\alpha_{\mathcal{K}}(\overline{Q},\overline{E},\zeta_{2r}^2)( \frac{ADO_r(\zeta_{2r}^{\alpha}, \mathcal{K})}{A_{\mathcal{K}}(\zeta_{2r}^{2r\alpha})}) = 0$.

\begin{rem} We have the following identities:\\
\(\overline{Q}(\frac{ADO_r(\zeta_{2r}^{\alpha}, \mathcal{K})}{A_{\mathcal{K}}(\zeta_{2r}^{2r\alpha})})=q^{2\alpha}\frac{ADO_r(\zeta_{2r}^{\alpha}, \mathcal{K})}{A_{\mathcal{K}}(\zeta_{2r}^{2r\alpha})}= \frac{\overline{Q} (ADO_r(\zeta_{2r}^{\alpha}, \mathcal{K}))}{A_{\mathcal{K}}(\zeta_{2r}^{2r\alpha})}\)\\
\(\overline{E}(\frac{ADO_r(\zeta_{2r}^{\alpha}, \mathcal{K})}{A_{\mathcal{K}}(\zeta_{2r}^{2r\alpha})})= \frac{\overline{E} (ADO_r(\zeta_{2r}^{\alpha}, \mathcal{K}))}{\overline{E}(A_{\mathcal{K}}(\zeta_{2r}^{2r\alpha}))}= \frac{\overline{E} (ADO_r(\zeta_{2r}^{\alpha}, \mathcal{K}))}{A_{\mathcal{K}}(\zeta_{2r}^{2r\alpha})}\) (because $\zeta_{2r}^{2r(\alpha+1)}=\zeta_{2r}^{2r\alpha}$ )
\end{rem}

Hence $\frac{\alpha_{\mathcal{K}}(\overline{Q},\overline{E},\zeta_{2r}^2) (ADO_r(\zeta_{2r}^{\alpha}, \mathcal{K}))}{A_{\mathcal{K}}(\zeta_{2r}^{2r\alpha})} = 0$, which proves the following theorem:

\begin{thm}\label{thm_ado_holonomy}
For any $0$ framed knot $\mathcal{K}$, $\alpha_{\mathcal{K}}(\overline{Q},\overline{E},\zeta_{2r}^2)(ADO_r(\zeta_{2r}^{\alpha}, \mathcal{K}))=0$.
\end{thm}

\begin{rem}
In the upcoming article \cite{brown_dimofte_geer2020holonomic}, Brown, Dimofte and Geer proved that the ADO invariant of links are q-holonomic (Theorem 4.3), which generalise Theorem \ref{thm_ado_holonomy}. Their Theorem 4.4 actually gives a converse statement of Theorem \ref{thm_ado_holonomy}: any polynomial annihilating the ADO family will also annihilate the colored Jones. This proves that the ADO and colored Jones family are annihilated by the same polynomials.
\end{rem}

\paragraph{Application 2: the unified invariant is the loop expansion of the colored Jones function}~\\
Let's first introduce the loop expansion of the colored (see section 2 in \cite{garoufalidis2005analytic}). We can write the colored Jones polynomials as an expansion (see \cite{rozansky1998universalr} for more details):
\[ J_n(e^{2h}, \mathcal{K})=\sum_{k=0}^{+\infty} \frac{P_{k}(e^{2nh})}{A_{\mathcal{K}}(e^{2nh})^{2k+1}} h^k \]
where $P_k(X) \in \Q[X,X^{-1}]$.
\medskip

Hence we get an element: \[ J_{\alpha}(q^2, \mathcal{K})= \sum_{k=0}^{+\infty} \frac{P_{k}(e^{2\alpha h})}{A_{\mathcal{K}}(e^{2\alpha h})^{2k+1}}h^k \in \Q[\alpha][[h]] \]
that is such that $f(J_{\alpha}(q^2, \mathcal{K}))= J_{\bullet}(q^2,\mathcal{K})$.\\
This means that it evaluates into the colored Jones at $\alpha =n$, we call it \textit{loop expansion of the colored Jones function}.

\begin{prop}\label{prop_loop_expansion_unified}
For any knot $\mathcal{K}$, we have the following identity in $\Q[\alpha][[h]]$:
 \[J_{\alpha}(q^2, \mathcal{K})= F_{\infty}(q,q^{\alpha},\mathcal{K}). \]
\end{prop}

\begin{proof}
The fact that $f$ is injective proves the proposition. 
\end{proof}

\bigskip
\begin{rem}
Putting everything together, this subsection implies that the unified invariant $F_{\infty}(q,A,\mathcal{K})$ is an integral version of the colored Jones function, built in a ring allowing evaluations at roots of unity. The integrality and the existence of evaluation maps allow us to recover the ADO polynomials, the fact that the completion ring is a subring of an $h$-adic ring allows us to connect it to other notions of colored Jones function/ invariants.

Another approach, described by Gukov and Manolescu in \cite{gukov2019two}, would be to see the unified invariant as a power serie in $q,A$ (as opposed to a quantum factorial expansion as we have here). This would be another integral version of it.\\
Indeed, because it verifies Proposition \ref{prop_loop_expansion_unified} and Theorem \ref{thm_unified_holonomy}, the unified invariant $F_{\infty}(q,A,\mathcal{K})$ except being a power serie, also verifies conjecture 1.5 and 1.6 in \cite{gukov2019two}.\\
Thus, if $F_{\infty}(q,A,\mathcal{K})$ could be written as a power serie, it would fully verify the conjectures.\\
This is the case for positive braid knots, as show by Park in \cite{park2020large}. This means that for a positive braid knot, the unified invariant and the GM power serie coincide.
\end{rem}
\section{Some computations}~\label{section_compute}
This section will be dedicated to compute the unified invariant $F_{\infty}(q,A,\mathcal{K})$ on some examples. We will also explicitly compute $C_{\infty}(1,A,\mathcal{K})$ and see that it is equal to the inverse of the Alexander polynomial.
\medskip

\noindent To do so we will use state diagrams and compute the unified invariant from it. You can also use them to compute the ADO polynomials (see \cite{murakami2008colored} section 4). Recall that $q^{\alpha} := A$.

\begin{figure}[h!]
\begin{subfigure}[b]{0.5\textwidth}
 \centering
  \def\svgwidth{30mm}
\begingroup%
  \makeatletter%
  \providecommand\color[2][]{%
    \errmessage{(Inkscape) Color is used for the text in Inkscape, but the package 'color.sty' is not loaded}%
    \renewcommand\color[2][]{}%
  }%
  \providecommand\transparent[1]{%
    \errmessage{(Inkscape) Transparency is used (non-zero) for the text in Inkscape, but the package 'transparent.sty' is not loaded}%
    \renewcommand\transparent[1]{}%
  }%
  \providecommand\rotatebox[2]{#2}%
  \newcommand*\fsize{\dimexpr\f@size pt\relax}%
  \newcommand*\lineheight[1]{\fontsize{\fsize}{#1\fsize}\selectfont}%
  \ifx\svgwidth\undefined%
    \setlength{\unitlength}{278.06085061bp}%
    \ifx\svgscale\undefined%
      \relax%
    \else%
      \setlength{\unitlength}{\unitlength * \real{\svgscale}}%
    \fi%
  \else%
    \setlength{\unitlength}{\svgwidth}%
  \fi%
  \global\let\svgwidth\undefined%
  \global\let\svgscale\undefined%
  \makeatother%
  \begin{picture}(1,2.05525577)%
    \lineheight{1}%
    \setlength\tabcolsep{0pt}%
    \put(0,0){\includegraphics[width=\unitlength,page=1]{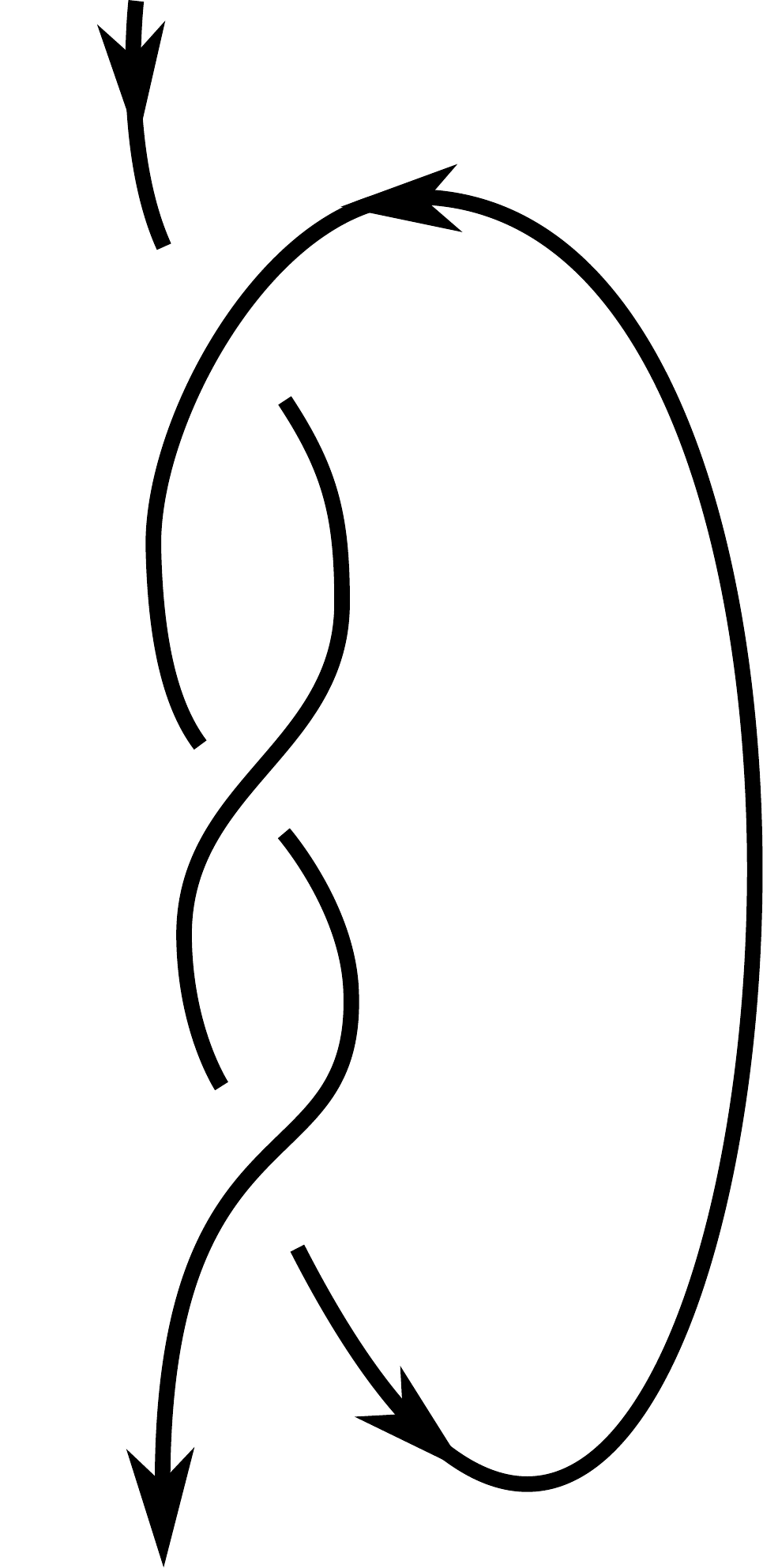}}%
    \put(0.10029114,0.24915924){\color[rgb]{0,0,0}\makebox(0,0)[lt]{\lineheight{1.25}\smash{\begin{tabular}[t]{l}$0$\end{tabular}}}}%
    \put(0.51682291,0.74544307){\color[rgb]{0,0,0}\makebox(0,0)[lt]{\lineheight{1.25}\smash{\begin{tabular}[t]{l}$0$\end{tabular}}}}%
    \put(-0.00358229,1.84347108){\color[rgb]{0,0,0}\makebox(0,0)[lt]{\lineheight{1.25}\smash{\begin{tabular}[t]{l}$0$\end{tabular}}}}%
    \put(0.5031998,1.34486289){\color[rgb]{0,0,0}\makebox(0,0)[lt]{\lineheight{1.25}\smash{\begin{tabular}[t]{l}$0$\end{tabular}}}}%
    \put(0.0512477,1.26544799){\color[rgb]{0,0,0}\makebox(0,0)[lt]{\lineheight{1.25}\smash{\begin{tabular}[t]{l}$i$\end{tabular}}}}%
    \put(0.09408526,0.76996479){\color[rgb]{0,0,0}\makebox(0,0)[lt]{\lineheight{1.25}\smash{\begin{tabular}[t]{l}$i$\end{tabular}}}}%
    \put(0.87060631,1.0015588){\color[rgb]{0,0,0}\makebox(0,0)[lt]{\lineheight{1.25}\smash{\begin{tabular}[t]{l}$i$\end{tabular}}}}%
  \end{picture}%
\endgroup%
   \caption{The trefoil knot.}
	\label{trefoil_w}
 \end{subfigure}%
 \begin{subfigure}[b]{0.5\textwidth}
 \centering
  \def\svgwidth{55mm}
    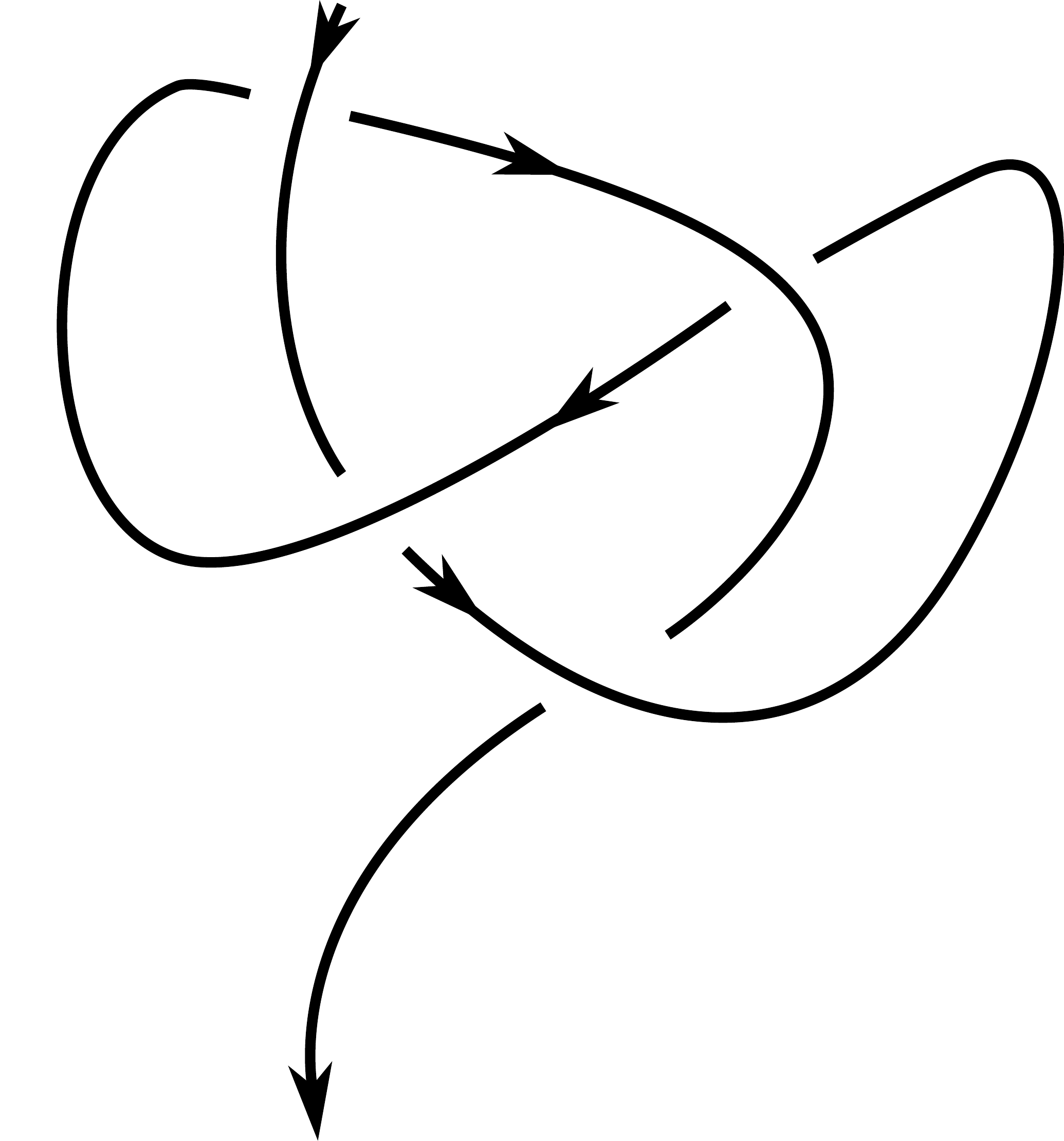%
   \caption{The figure eight knot.}
	\label{figure8_w}
 \end{subfigure}%
 \caption{Examples of state diagrams to compute the invariants.}
 \end{figure}

\paragraph{The Trefoil Knot:}~

We will denote the trefoil knot $3_1$:
\begin{flalign*}
F_{\infty} (q, A, 3_1)&= q^{\frac{3 \alpha^2}{2}} \underset{i}{\sum} q^{\alpha -2i} q^{\frac{i(i-1)}{2}} \{ \alpha; i \}_q q^{-3i \alpha}&
\end{flalign*}
\begin{flalign*}
C_{\infty} (1, A, 3_1)&= q^{\alpha} \underset{i}{\sum} q^{- 3 i \alpha }  \{ \alpha \}_q^i& \\
&= q^{ \alpha} \frac{1}{1-q^{- 3 \alpha }  \{ \alpha \}_q}&\\
&=\frac{q^{3\alpha}}{A_{3_1} (q^{2\alpha})}&
\end{flalign*}

\paragraph{The Figure Eight Knot:}~

We will denote the figure eight knot $4_1$:
\begin{flalign*}
F_{\infty} (q, A, 4_1)&=  \underset{i,j}{\sum} q^{2(i-j)} q^{- i \alpha} q^{j \alpha} (-1)^i q^{-\frac{i(i-1)}{2}} \qbinom{i+j}{j}_q \{\alpha -j;i\}_q q^{(i+j) \alpha}& \\ & \times q^{-2ij} q^{-\frac{j(j-1)}{2}} \{\alpha ;j\}_q q^{-(i+j) \alpha} q^{2ij}&\\
 &= \underset{i,j}{\sum} q^{2(i-j)} q^{(j- i) \alpha} (-1)^i  \qbinom{i+j}{j}_q   q^{-\frac{i(i-1)}{2}} q^{-\frac{j(j-1)}{2}} \{\alpha ;i+j\}_q &
\end{flalign*}
\begin{flalign*}
C_{\infty} (1, A, 4_1)&= \underset{i,j}{\sum}  q^{(j- i) \alpha} (-1)^i  \binom{i+j}{j}  \{\alpha \}_q^{i+j}&\\
&= \underset{N}{\sum} \overset{N}{\underset{i=0}{\sum}} q^{N \alpha} q^{-2i \alpha} (-1)^i  \binom{i+j}{j}\{\alpha \}_q^N&\\
&= \underset{N}{\sum}  q^{N \alpha}\{\alpha \}_q^N \overset{N}{\underset{i=0}{\sum}} (-q^{-2 \alpha})^i  \binom{i+j}{j}&\\
&=  \underset{N}{\sum}  q^{N \alpha}\{\alpha \}_q^N (1-q^{-2 \alpha})^N\\
&= \underset{N}{\sum}  \{\alpha \}_q^{2N}&\\
&=\frac{1}{1-\{\alpha \}_q^{2}}&\\
&=\frac{1}{A_{4_1} (q^{2\alpha})}&
\end{flalign*}


\paragraph{The Cinquefoil Knot:}~

We will denote it by $5_1$:
\begin{flalign*}
F_{\infty} (q, A, 5_1)&= q^{\frac{5 \alpha^2}{2}} \underset{i,j,k}{\sum} q^{ \alpha -2(i-j+k)} q^{-5(i-j+k) \alpha} q^{2i(k-j)} q^{2k(i-j)} q^{\frac{i(i-1)}{2}} q^{\frac{j(j-1)}{2}} &  \\ &  \times q^{\frac{k(k-1)}{2}} \{ \alpha; i \}_q \{ \alpha -k+j; j \}_q \{ \alpha -i+j; k \}_q  &  \\ &  \times \qbinom{k}{k-j}_q \qbinom{i-j+k}{k}_q&
\end{flalign*}
\begin{flalign*}
C_{\infty} (1, A, 5_1)&= q^{\alpha} \underset{i,j,k}{\sum}  q^{-5(i-j+k) \alpha} \{ \alpha \}_q^{i+j+k} \binom{k}{k-j}\binom{i-j+k}{k}&\\
&= q^{ \alpha} \underset{i,j,k}{\sum}  q^{-5(i-j+k) \alpha} \{ \alpha \}_q^{i+j+k} \binom{i-j+k}{j,k-j,i-j}&\\
&= q^{ \alpha} \underset{N}{\sum} \underset{j,k}{\sum}  q^{-5N \alpha} \{ \alpha \}_q^{N+2j} \binom{N}{j,k-j,N-k}&\\
&= q^{\alpha} \underset{N}{\sum}  q^{-5N \alpha} \{ \alpha \}_q^{N} \underset{j,k}{\sum}  \{ \alpha \}_q^{2j} \binom{N}{j,k-j,N-k}&\\
&=q^{\alpha} \underset{N}{\sum}  q^{-5N \alpha} \{ \alpha \}_q^{N} (2+ \{ \alpha \}_q^2)^N&\\
&=q^{\alpha} \frac{1}{1-q^{-5 \alpha} \{ \alpha \}_q (2+ \{ \alpha \}_q^2)}&\\
&= \frac{q^{5 \alpha}}{A_{5_1} (q^{2 \alpha})}&
\end{flalign*}

\paragraph{The three twist Knot:}~

We will denote it by $5_2$:
\begin{flalign*}
F_{\infty} (q, A, 5_2)&= q^{\frac{-5 \alpha^2}{2}} \underset{i,j,k}{\sum} q^{2(i-j+k) - \alpha} q^{(5i+5j-3k) \alpha} q^{2ij} q^{2(j-k)(i+j)} q^{\frac{i(i-1)}{2}}  &  \\ &  \times  q^{\frac{j(j-1)}{2}} q^{\frac{k(k-1)}{2}} (-1)^{i+j+k} \{ \alpha; i \}_q \{ \alpha -i; j \}_q \{ \alpha -j+k; k \}_q &  \\ &  \times \qbinom{j}{j-k}_q \qbinom{i+j}{j}_q&
\end{flalign*}
\begin{flalign*}
C_{\infty} (1, A, 5_2)&=q^{ -\alpha} \underset{i,j,k}{\sum}  q^{(5i+5j-3k) \alpha} (-1)^{i+j+k}  \{ \alpha \}_q^{i+j+k} \binom{j}{j-k} \binom{i+j}{j}&\\
&=q^{ \alpha} \underset{i,j,k}{\sum}  q^{(5i+5j-3k) \alpha} (-1)^{i+j+k}  \{ \alpha \}_q^{i+j+k} \binom{i+j}{k,j-k,i}&\\
&=q^{-\alpha} \underset{N}{\sum} \underset{j,k}{\sum} q^{(5N-3k) \alpha} (-1)^{N+k}  \{ \alpha \}_q^{N+k} \binom{N}{k,j-k,N-i}&\\
&=q^{ - \alpha} \underset{N}{\sum} q^{5N \alpha} (-1)^{N}  \{ \alpha \}_q^{N} \underset{j,k}{\sum} q^{-3k \alpha} (-1)^{k}  \{ \alpha \}_q^{k}&  \\ &  \hspace{150pt} \times \binom{N}{k,j-k,N-i}&\\
&=q^{- \alpha} \underset{N}{\sum} q^{5N \alpha} (-1)^{N}  \{ \alpha \}_q^{N} (2-q^{-3 \alpha} \{ \alpha \}_q)^N&\\
&=q^{- \alpha} \frac{1}{1+q^{5 \alpha}  \{ \alpha \}_q (2-q^{-3 \alpha} \{ \alpha \}_q)}&\\
&= \frac{q^{-5 \alpha}}{A_{5_2} (q^{2 \alpha})}&
\end{flalign*}

 \begin{figure}[h!]
\begin{subfigure}[b]{0.45\textwidth}
 \centering
  \def\svgwidth{40mm}
    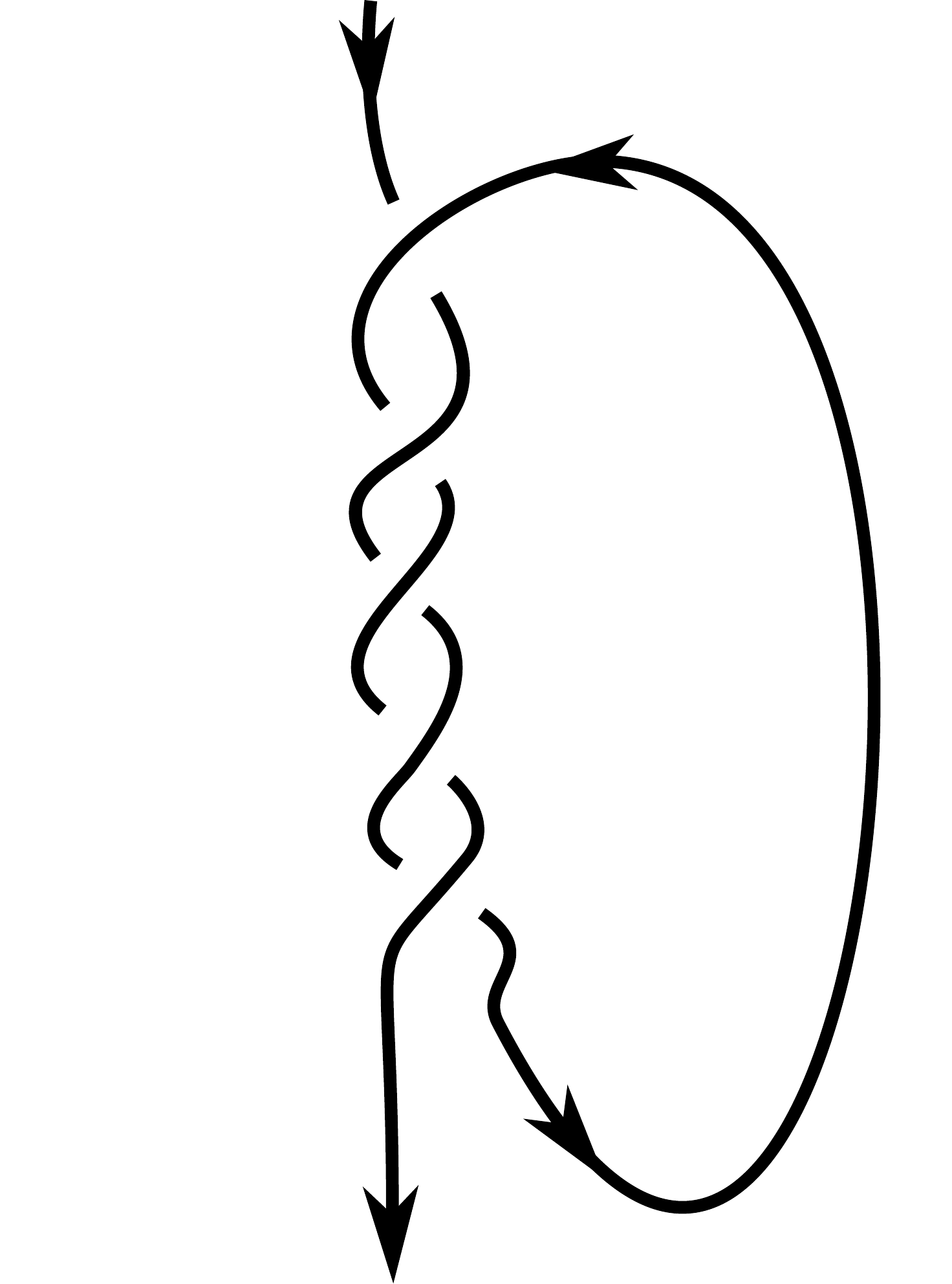%
   \caption{The Cinquefoil Knot.}
	\label{torus_2_5_w}
 \end{subfigure}%
 \begin{subfigure}[b]{0.4\textwidth}
 \centering
  \def\svgwidth{50mm}
    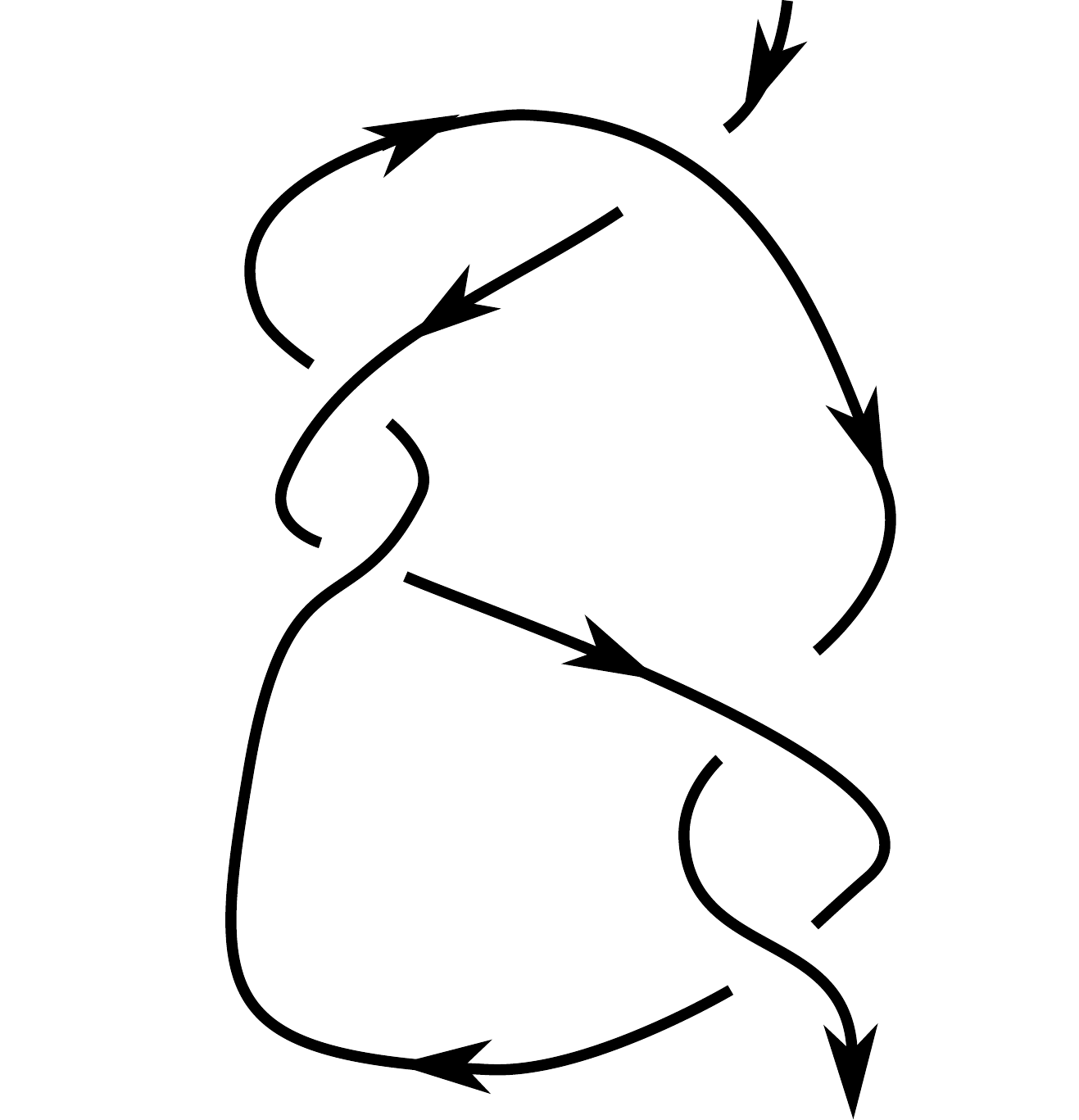%
   \caption{The three twist Knot.}
	\label{figure_of_nine_w}
 \end{subfigure}%
 \caption{Examples of state diagrams to compute the invariants.}
 \end{figure}

\bibliographystyle{abbrv}
\bibliography{unification_ADO_and_colored_Jones_of_knot}
\address
\end{document}